\documentclass[a4paper, 11pt]{amsart}

\usepackage[margin=3cm]{geometry}

\usepackage{amsmath,amssymb,amsfonts,amsthm}
\usepackage[all]{xy}
\usepackage[latin1]{inputenc}
\usepackage[english]{babel}
\usepackage{enumerate}
\usepackage[colorlinks=true,linkcolor=red,citecolor=blue]{hyperref}

\usepackage{tikz, tikz-3dplot}
\usepackage{tikz-cd}
\usepackage{mathrsfs}
\usepackage{pgfplots}
\usepgfplotslibrary{fillbetween}

\usepackage{fourier}

\newtheorem{thm}{Theorem}[section]
\newtheorem*{thm*}{Theorem}
\newtheorem{prop}[thm]{Proposition}
\newtheorem{lem}[thm]{Lemma}

\newtheorem{coro}[thm]{Corollary}

\theoremstyle{remark}
\newtheorem{rem}[thm]{Remark}
\newtheorem{ex}[thm]{Example}
\newtheorem*{ex*}{Example}

\theoremstyle{definition}
\newtheorem{defi}[thm]{Definition}

\renewcommand{\leq}{\leqslant}
\renewcommand{\geq}{\geqslant}
\newcommand{\ZZ}{\mathbb{Z}}
\newcommand{\QQ}{\mathbb{Q}}
\newcommand{\RR}{\mathbb{R}}
\newcommand{\CC}{\mathbb{C}}
\newcommand{\PP}{\mathbb{P}}

\renewcommand{\to}{\rightarrow}
\newcommand{\To}{\longrightarrow}
\renewcommand{\H}{\operatorname{H}}

\newcommand{\dR}{\mathrm{dR}}

\newcommand{\dd}{\mathrm{d}}
\newcommand{\dlog}{\mathrm{dlog}}
\newcommand{\rk}{\mathrm{cr}}

\newcommand{\W}{\mathrm{W}}
\newcommand{\F}{\mathrm{F}}
\newcommand{\gr}{\mathrm{gr}}

\newcommand{\Res}{\mathrm{Res}}

\newcommand{\can}{\mathrm{can}}

\newcommand{\omegacan}{\varpi}

\newcommand{\mapdef}[5]{
\begin{array}{cccc}
#1 \colon & #2 &\longrightarrow & #3 \\
& #4 &\longmapsto& #5
\end{array}
}

\newcommand{\logforms}[1]{\Omega^{#1}_{\,\log}}

\newcommand{\diagram}[1]{\SelectTips{lu}{10}\xymatrix{#1}}

\renewcommand{\varnothing}{\emptyset}

\usepackage{microtype}

\definecolor{MyBlue}{HTML}{2789f0}
\definecolor{MyGreen}{HTML}{289d43}
\definecolor{MyDarkBlue}{HTML}{22427c}
\definecolor{MyRed}{HTML}{e83b3b}

\title{Positive geometries and canonical forms via mixed Hodge theory}
\date{}

\author{Francis Brown}
\address{All Souls College, Oxford, Oxford OX1 4AL, UK}
\email{francis.brown@all-souls.ox.ac.uk}

\author{Cl\'{e}ment Dupont}
\address{Institut Montpelli\'erain Alexander Grothendieck, Universit\'e de Montpellier, CNRS,
Montpellier, France}
\email{clement.dupont@umontpellier.fr}

\begin{document}

\begin{abstract}
``Positive geometries'' are a class of semi-algebraic domains which admit a unique ``canonical form'': a logarithmic form whose residues  match the boundary structure of the domain. The study of such geometries is motivated by recent progress in particle physics, where the corresponding canonical forms are interpreted as the integrands of scattering amplitudes. We recast these concepts in the language of mixed Hodge theory, and identify ``genus zero pairs'' of complex algebraic varieties as a natural and general framework for the study of positive geometries and their canonical forms. In this framework, we prove some basic properties of canonical forms which have previously been proved or conjectured in the literature. We give many examples and study in detail the case of arrangements of hyperplanes and convex polytopes.
\end{abstract}
 
\maketitle

Prompted by Arkani-Hamed and Trnka's discovery of amplituhedra \cite{arkanihamedtrnkaamplituhedron}, the concept of \emph{positive geometry} \cite{arkanihamedbailampositive} recently emerged as an important tool in the study of scattering amplitudes in particle physics. Roughly speaking, a positive geometry is a semi-algebraic domain $\sigma$ for which there exists a unique logarithmic form $\omegacan_\sigma$, called the \emph{canonical form}, whose residues match the boundary structure of $\sigma$. Semi-algebraic domains are defined by the positivity of certain polynomials, hence the terminology. Convex polytopes, defined by the positivity of linear functions, provide the simplest examples of positive geometries.

The aim of this article is to recast these notions as natural byproducts of Deligne's mixed Hodge theory \cite{delignehodge2}, a central organizing principle in complex algebraic geometry which is intimately linked to properties of logarithmic forms and their residues. 

Our main object of study is the relative homology group $\H_n(X,Y)$, for $X$ a compact complex algebraic variety of dimension $n$ and $Y\subset X$ a closed subvariety, such that $X\setminus Y$ is smooth. It is a finite-dimensional $\QQ$-vector space spanned by classes of oriented domains $\sigma\subset X$ with $\partial(\sigma) \subset Y$, and is equipped with a \emph{mixed Hodge structure}, which consists of  linear algebra data from which one may extract a set of \emph{Hodge numbers} $h^{-p,-q}$, for $0\leq p,q\leq n$.
Let us call $(X,Y)$ a \emph{genus zero pair} if $\H_n(X,Y)$ has vanishing Hodge numbers $h^{-p,0}$ for all $p>0$.
A class of examples is given by $X=\PP^n_\CC$ and $Y$ a union of hyperplanes; in this situation, $\H_n(X,Y)$ is spanned by the classes of projective polytopes $\sigma$ bounded by $Y$. 

Genus zero pairs are a natural framework for  positive geometries and their canonical forms, as summarized by the following result  (see \S\ref{sec:canonical-forms} for more precise statements):

\begin{thm*}
For every genus zero pair $(X,Y)$ as above we define a map
\begin{equation}\label{eq:can-intro}
\mapdef{\can}
{\H_n(X,Y)}{\logforms{n}(X\setminus Y)}
{\sigma}{\omegacan_\sigma}
\end{equation}
and call $\omegacan_\sigma:=\can(\sigma)$ the \emph{canonical form} of $\,\sigma$. The maps $\can$ satisfy the following compatibilities.
\begin{enumerate}[(a)]
\item Linearity, i.e., compatibility with triangulations and with orientation: the total canonical form of a triangulation is the sum of the individual canonical forms of the triangulation, and changing the orientation of $\sigma$  changes the sign of the canonical form.
\item Recursion, i.e., compatibility between boundaries of relative chains and residues of logarithmic forms: the canonical form of the boundary of $\sigma$ is the residue of the canonical form of $\sigma$.
\item Invariance under modifications (e.g., blow-ups).
\item Functoriality: the canonical form of the pushforward (resp. pullback) of $\sigma$ is the pushforward (resp. pullback) of the canonical form of $\sigma$.
\item Multiplicativity, i.e., compatibility with products of pairs of complex varieties: the canonical form of a product is the wedge product of the associated canonical forms. 
\end{enumerate}
\end{thm*}

Before going into the details of our construction, we start with an illustrative example (many more examples, of increasing complexity, can be found in \S\ref{sec:examples} and \S\ref{sec:hyperplane-arrangements}).

\begin{ex*}[The one-dimensional case]
Let $X$ be a smooth compact complex curve (also known as a compact Riemann surface) of genus $g$, and let $Y=\{a,b\}$ consist of two distinct points of $X$. There exists a holomorphic $1$-form $\omega$ on $X\setminus \{a,b\}$ with logarithmic poles at $a$ and $b$ and residues $\Res_a(\omega)=-1$, $\Res_b(\omega)=1$, which is only well-defined up to adding  to $\omega$ a global holomorphic form on $X$, whose residues vanish everywhere. The space of global holomorphic forms on $X$ has dimension $g$, and therefore $\omega$ is unique if and only if $g=0$, i.e., if $X=\PP^1_\CC$. In the language of mixed Hodge theory, the obstruction to the uniqueness of $\omega$ is measured by the Hodge number $h^{-1,0}$ of $\H_1(X,Y)$, which equals $g$.

Let now $X=\PP^1_\CC$ and $Y=\{a,b\}\subset \CC\subset \PP^1_\CC$ consist of two distinct complex numbers. The homology group $\H_1(\PP^1_\CC,\{a,b\})$ has dimension $1$ and its only non-zero Hodge number $h^{0,0}=1$; hence $(\PP^1_\CC,\{a,b\})$ is a genus zero pair. All continuous paths $\gamma_{a,b}$ from $a$ to $b$ in $\PP^1_\CC$ have the same homology class, and the corresponding canonical form is
$$\omegacan_{\gamma_{a,b}}= \mathrm{dlog}\left(\frac{z-b}{z-a}\right) = \frac{(b-a)\,\dd z}{(z-a)(z-b)}.$$
It is the unique holomorphic form on $\PP^1_\CC\setminus\{a,b\}$ with logarithmic poles at $a$ and $b$ and respective residues $-1$ and $1$. The latter is the  recursive property of canonical forms described in the  theorem.  Indeed, the residues of $\omegacan_{\gamma_{a,b}}$ are the canonical forms of the boundaries of $\gamma_{a,b}$:
$$\Res_a(\omegacan_{\gamma_{a,b}})=-1=\omegacan_{\partial_a(\gamma_{a,b})} \qquad \mbox{ and } \qquad \Res_b(\omegacan_{\gamma_{a,b}})=1=\omegacan_{\partial_b(\gamma_{a,b})},$$
where $\partial_a(\gamma_{a,b})=-\{a\}$, $\partial_b(\gamma_{a,b})=\{b\}$ are the boundaries of $\gamma_{a,b}$, and the canonical form of a point is the zero-degree form $1$.
The linearity property of canonical forms, also described in the previous theorem, is illustrated by the following identities:
$$\omegacan_{\gamma_{a,c}} = \omegacan_{\gamma_{a,b}} + \omegacan_{\gamma_{b,c}} \qquad , \qquad  \omegacan_{\gamma_{b,a}}= -\omegacan_{\gamma_{a,b}}.$$
\end{ex*}

In the rest of the introduction, we sketch our construction of the map $\can$  and make some general remarks about canonical forms from the point of view of mixed Hodge theory. Let $X$ be a compact complex variety of dimension $n$ and $Y\subset X$ be a closed subvariety such that $X\setminus Y$ is smooth.

\subsection*{Logarithmic forms}

The  vector space $\logforms{n}(X\setminus Y)$ appearing in the theorem above consists of global holomorphic $n$-forms on $X\setminus Y$ \emph{with logarithmic poles at infinity}, i.e., on some smooth compactification of $X\setminus Y$ obtained by adding a normal crossing divisor. Mixed Hodge theory implies that this notion is independent of the choice of compactification and therefore intrinsic to $X\setminus Y$, which is already a remarkable statement in itself. Indeed, there is a natural isomorphism
\begin{equation}\label{eq:intro-log-forms-and-hodge-filtration}
\logforms{n}(X\setminus Y)\simeq \F^n\H^n(X\setminus Y;\CC),
\end{equation}
where $\F$ is the Hodge filtration on cohomology with complex coefficients.

\subsection*{The definition of canonical forms}

By combining \eqref{eq:intro-log-forms-and-hodge-filtration} with the natural (Poincar\'{e}) duality between $\H_n(X,Y)$ and $\H^n(X\setminus Y)$, mixed Hodge theory provides a surjective linear map
\begin{equation}\label{eq:intro-R-map}
R\colon \logforms{n}(X\setminus Y) \twoheadrightarrow \gr_0^\W\H_n(X,Y;\CC),
\end{equation}
where $\gr_0^\W$ denotes the weight zero quotient of homology. The kernel of \eqref{eq:intro-R-map} has dimension equal to the \emph{genus} of $(X,Y)$, defined (in this paper) to be the sum of the Hodge numbers $h^{-p,0}$ of $\H_n(X,Y)$, for $p>0$. If $(X,Y)$ has genus zero, then the $R$-map (which stands for ``iterated residue map'') is an isomorphism and we define \eqref{eq:can-intro} as the composition of the natural quotient map 
$$\H_n(X,Y)\twoheadrightarrow \gr_0^\W\H_n(X,Y)$$
with the inverse of  \eqref{eq:intro-R-map}. Note that, by definition, \eqref{eq:can-intro} becomes surjective after tensoring with $\CC$, and hence \emph{every} holomorphic $n$-form on $X \setminus Y$ with logarithmic poles at infinity is a canonical form of a relative cycle with complex coefficients.

If the genus $g$ of $(X,Y)$ is $>0$, then since \eqref{eq:intro-R-map} is surjective, logarithmic forms with analogous properties to the canonical forms will still exist, but are only defined modulo the $g$-dimensional vector space $\ker(R)$. Therefore, as observed in the case of curves, the genus zero hypothesis is equivalent to the \emph{uniqueness} of canonical forms.

\subsection*{The genus of a pair of varieties}

The genus of a pair $(X,Y)$ generalizes the geometric genus, which is defined for a smooth compact complex variety $X$ of dimension $n$ to be the maximal number of independent holomorphic $n$-forms on $X$,
$$g(X) = \dim \H^0(X,\Omega^n_X).$$
If $X$ is a (possibly singular) compact complex curve and $Y$ is a finite set of points, then the genus of $(X,Y)$ is the classical genus of $X$. In \S\ref{sec:genus} we study the more general notion of the genus of a pair, defined as above as a sum of mixed Hodge numbers of relative homology, and provide tools to compute it. Roughly speaking, the genus of $(X,Y)$ has contributions from the genus of $X$, the genus of the irreducible components of $Y$, and their multiple intersections. 

An important class of examples arises for $X=\PP^n_\CC$ and $Y=$ a hypersurface of degree $d$. The genus is always zero if $d\leq n$, but can still be zero when $d\geq n+1$ if $Y$ is sufficiently singular, e.g., any singular cubic curve in the plane, or any union of hyperplanes in $\PP^n_\CC$.

Being of  genus zero is a weaker condition, and therefore more general,  than being  mixed Tate (which means that  the Hodge numbers $h^{-p,-q}$ vanish for $p\neq q$). For instance, a smooth cubic threefold $X\subset \PP^4_\CC$ has genus zero but has non-zero Hodge numbers $h^{-1,-2}=h^{-2,-1}=5$. 

\subsection*{Computation of canonical forms: recursive and non-recursive}

The properties (a)--(e) listed in the theorem above are important tools for the computation of canonical forms. In particular, the recursion property (b) (see Proposition \ref{prop:recursion}) states that, under natural assumptions, $\omegacan_\sigma$ is the unique holomorphic $n$-form on $X\setminus Y$ with logarithmic poles at infinity whose residue along $Y$ is equal to  the canonical form of the boundary $\partial(\sigma)\subset Y$:
\begin{equation}\label{eq:intro-recursion}
\Res(\omegacan_{\sigma}) = \omegacan_{\partial(\sigma)}.
\end{equation}
This property, which mirrors the \emph{recursive definition} of positive geometries \cite{arkanihamedbailampositive}, sometimes allows one to compute canonical forms by induction on the dimension. Note, however, that this approach simply does not work in  situations where  the genus zero assumption is not met on the boundary, i.e., there exist genus zero pairs which do not meet  the (recursive) definition of positive geometries, and are therefore more general (see \S\ref{sec:genus-zero-not-recursive} and \S\ref{sec:example-non-recursive-positive-geometry}).

Even in situations where the recursive approach holds, it may still be less practical than our non-recursive definition. If $X$ is smooth and $Y$ is a normal crossing divisor, then the $R$-map \eqref{eq:intro-R-map} is computed by  \emph{corner residues}, i.e., iterated $n$-fold residues along local branches of $Y$, and the canonical form $\omegacan_\sigma$ is the unique holomorphic $n$-form on $X\setminus Y$ with logarithmic poles at infinity whose corner residues match the corner boundaries of $\sigma$ (Corollary \ref{coro:corners}). In some situations it is more efficient, and more instructive,  to compute these iterated residues directly rather than applying a recursive definition.

As an application of our non-recursive approach to canonical forms in a more singular situation, we give (in Proposition \ref{prop:canonical-form-nbc}) a general formula for canonical forms of convex polytopes, which we believe to be new.

\subsection*{The combinatorial rank}

Besides the genus, there is another important numerical invariant of the pair $(X,Y)$, which is the Hodge number $h^{0,0}$ of $\H_n(X,Y)$, or equivalently the dimension of the weight zero quotient $\gr_0^\W\H_n(X,Y)$. We call it the \emph{combinatorial rank} because in the case where $X$ is smooth and $Y$ is a normal crossing divisor, it is completely determined by the dual complex $\Delta(Y)$ (see Proposition \ref{prop:rank-as-combinatorial-invariant}). The latter is a classical invariant which appears in tropical and non-archimedean geometry. In our framework, the combinatorial rank is the rank of the $R$-map \eqref{eq:intro-R-map}, and in the genus zero case, equals the dimension of the space of canonical forms.

To summarize: the \emph{existence} of a non-zero canonical form is equivalent to the non-vanishing of the combinatorial rank; its \emph{uniqueness} is equivalent to the vanishing of the genus.

\subsection*{Complex geometry, real geometry}

In the recent literature on positive geometries, the pair $(X,Y)$ is assumed to be the complexification of a pair $(X_\RR,Y_\RR)$ of \emph{real} algebraic varieties, and the domain $\sigma$ to be a semi-algebraic subset of $X_\RR(\RR)$ whose boundary lies on $Y_\RR(\RR)$. As our results show, this assumption is unnecessary for the formalism of canonical forms to hold, even though the real case exhibits interesting phenomena related to convexity, see for example \cite{laminvitation}.

\subsection*{Canonical forms and periods}

By Serre's GAGA theorem, all canonical forms are global algebraic $n$-forms on $X\setminus Y$. Furthermore, if the pair $(X,Y)$ is obtained by extension of scalars from a pair $(X_K,Y_K)$ defined over a subfield $K\subset \CC$, then canonical forms are defined over a finite extension $L$ of $K$, and hence give rise to a linear map
$$\H_n(X,Y) \to \H^n_\dR(X_K\setminus Y_K)\otimes_K L,$$
where $\H^\bullet_\dR$ denotes algebraic de Rham cohomology. This context is discussed in our earlier work \cite{brownnotes, BD1} (there, the mixed Hodge structures $\H^n(X,Y)$ corresponding to genus zero pairs of complex varieties were called ``separated'') where canonical forms were implicit in the definition of the ``de Rham projection''. The latter allows one to pass from a period to a ``single-valued period'' and is roughly given in symbols by
\begin{equation}\label{eq:intro-from-period-to-sv-period}
\int_\sigma \nu \qquad \leadsto \qquad (2\pi\mathrm{i})^{-n} \int_X \omegacan_\sigma\wedge \overline{\nu},
\end{equation}
where $\nu$ is some other algebraic $n$-form whose poles are sufficiently far from $Y$. This operation preserves all algebraic relations ``coming from geometry'', which are believed to exhaust all algebraic relations among periods (Grothendieck's period conjecture). A useful slogan is that: in the genus zero situation 
 motivic periods have a well-defined ``de Rham version'' (which appears as the leading term in the motivic coaction), and consequently there exists a well-defined single-valued projection map which formalizes the operation \eqref{eq:intro-from-period-to-sv-period}.

\subsection*{Relative homology and locally finite homology}
The formalism of canonical forms only depends upon the open variety $X\setminus Y$, and, since $X$ is compact, the relative homology group $\H_n(X,Y)$ is identified with the locally finite (or ``Borel--Moore'') homology group $\H_n^{\mathrm{lf}}(X\setminus Y)$. We have chosen to work in the setting of relative homology because in many concrete situations the compactification $X$ is already part of the data. However, locally finite homology is more canonical in situations when  the geometry  in the initial data does not already come with a preferred compactification (see the case of cluster varieties \cite[Conjecture 1]{laminvitation}). 

\subsection*{The compactness assumption}

In practice one may encounter a situation where the cycle $\sigma$ lies  on some non-compact variety $X$ (for instance when studying integrals of the form  \eqref{eq:intro-from-period-to-sv-period}). As explained in \cite[\S 4]{BD1}, all the relevant parts of the mixed Hodge structure are preserved by compactifying the geometry, hence it is sufficient to consider the case when $X$ is compact.

\subsection*{Beyond genus zero}

If the genus $g$ of $(X,Y)$ is $>0$, then, as noted above, canonical forms are only well-defined up to the $g$-dimensional space $\ker(R)$, where $R$ is  the map \eqref{eq:intro-R-map}. Thus, if one is happy to  work with cosets of logarithmic forms and hence drop the uniqueness property, then the entire theory may be generalized. If $X$ is smooth of positive dimension, then $\ker(R)$ contains all global holomorphic $n$-forms on $X$; if furthermore $Y$ is a simple normal crossing divisor, then $\ker(R)$ consists of those logarithmic $n$-forms whose iterated residue is holomorphic on some stratum of $Y$ of positive dimension (see Remark \ref{rem:ker-of-R}).

In some particular  situations,  the existence of additional structure may lead to a preferred splitting of $R$, and therefore still give rise to preferred choices of canonical forms. See \S\ref{subsec:beyond-genus-zero} for a discussion in the case of (elliptic) curves.

\subsubsection*{Contents}

In \S\ref{sec:generalities} we review the general formalism of mixed Hodge theory and discuss the notion of logarithmic  form. The construction of canonical forms and the proof of our main theorem are carried out in \S\ref{sec:canonical-forms}. In \S\ref{sec:genus} and \S\ref{sec:combinatorial-rank} we study the general notion of genus and combinatorial rank, respectively. Many examples illustrating our construction are given in \S\ref{sec:examples}, and \S\ref{sec:hyperplane-arrangements} contains a specific study of the case of arrangements of hyperplanes and convex polytopes.

\subsubsection*{Conventions and notation}

A \emph{complex variety} is a reduced (not necessarily irreducible) scheme over $\CC$ which is separated and of finite type. We sometimes view such an object as a  complex analytic variety, or as a topological space for the analytic topology. For convenience, we sometimes implicitly assume that our complex varieties are equidimensional. For a $\QQ$-vector space $V$, we let $V_\CC:=V\otimes_\QQ\CC$ denote its complexification.

\subsubsection*{Signs}

We have endeavored to work out the correct signs in all our formulas (in particular in the formulas for canonical forms given in \S\ref{sec:examples} and \S\ref{sec:hyperplane-arrangements}); the difficulty  comes from signs in the boundary and residue maps. We warn the reader that our sign convention for residues is not entirely standard (Remark \ref{rem:sign-convention-residue}) but has the advantage of making the recursion formula \eqref{eq:intro-recursion} sign-free. Furthermore, our formulas for canonical forms sometimes differ from those in the literature by a sign, due to different conventions.

\subsubsection*{Acknowledgements}
We have benefited from helpful discussions and correspondence with Michi Borinsky, Melody Chan, Gabriele Dian, Henri Guenancia, Johannes Henn, Pierre Lairez, Thomas Lam, Lionel Mason, Ana\"{e}lle Pfister, Anna-Laura Sattelberger, Bernd Sturmfels, Raluca Vlad, and Lauren Williams.
F.~B.\ acknowledges support from STFC grant ST/X000761/1. C.~D.\ was supported by the Agence Nationale de la Recherche through grant ANR-23-CE40-0011 CYCLADES.

\section{Generalities on relative (co)homology, mixed Hodge theory, and logarithmic forms}\label{sec:generalities}

\subsection{Relative (co)homology}

    \subsubsection{Notation}

    Basic facts about  relative (co)homology can be found in standard textbooks on algebraic topology, e.g., \cite{hatcherbook}.

    \begin{defi}
    A \emph{pair} of complex varieties is a pair $(X,Y)$ where $X$ is a complex variety and $Y\subset X$ is a closed subvariety. A \emph{morphism} of pairs $f:(X',Y')\to (X,Y)$ is a morphism of complex varieties $f:X'\to X$ which satisfies $f(Y')\subset Y$.
    \end{defi}
    
    For a pair $(X,Y)$ of complex varieties, we denote by $\H_\bullet(X,Y)$ the corresponding relative (singular) homology groups with coefficients in $\QQ$. They are finite-dimensional $\QQ$-vector spaces. By definition, an element of $\H_k(X,Y)$ is the class $[\sigma]$ of a relative $k$-chain $\sigma$, i.e., a formal $\QQ$-linear combination of $k$-simplices, which are continuous maps   $\Delta^k\to X$, whose boundary $\partial\sigma$ is (a linear combination of simplices) contained in $Y$. When there is no ambiguity we will abuse notation and simply write $\sigma$ instead of $[\sigma]$.  We will also work with the corresponding relative cohomology groups, which are the  dual vector spaces $\H^k(X,Y) = \H_k(X,Y)^\vee$.  The usual singular (co)homology groups of a complex variety $X$ are the relative (co)homology groups of the pair $(X,\varnothing)$. Recall the long exact sequence in relative homology for a triple $(X,Y,Z)$, where $Z\subset Y\subset X$ are closed subvarieties, which reads:
    \begin{equation}\label{eq:long-exact-sequence-relative-homology}\cdots \to \H_\bullet(Y,Z)\to \H_\bullet(X,Z) \to \H_\bullet(X,Y)\to \H_{\bullet-1}(Y,Z)\to \cdots
    \end{equation}
    Dually, we have the long exact sequence in relative cohomology:
    \begin{equation}\label{eq:long-exact-sequence-relative-cohomology}\cdots \to \H^{\bullet-1}(Y,Z)\to \H^\bullet(X,Y) \to \H^\bullet(X,Z)\to \H^\bullet(Y,Z)\to \cdots
    \end{equation}

    Relative (co)homology is functorial for morphisms of pairs: a morphism $f:(X',Y')\to (X,Y)$ induces a linear map $f_*:\H_\bullet(X',Y')\to \H_\bullet(X,Y)$, called a pushforward map, and dually induces  a pullback map $f^*:\H^\bullet(X,Y)\to \H^\bullet(X',Y')$.

    If $X$ is compact, then $\H^\bullet(X,Y)$ is isomorphic to the compactly supported cohomology of the complement $X\setminus Y$, and dually $\H_\bullet(X,Y)$ is isomorphic to the locally finite homology of $X\setminus Y$ (which is also known as  Borel-Moore homology):
    \begin{equation}\label{eq:relative-homology-is-lf-homology}
    \H_\bullet(X,Y)\simeq \H^{\mathrm{lf}}_\bullet(X\setminus Y) \qquad \mbox{ and } \qquad \H^\bullet(X,Y)\simeq \H^\bullet_{\mathrm{c}}(X\setminus Y).
    \end{equation}

    \subsubsection{Modifications and the excision theorem}
    
    The classical excision theorem implies that the relative (co)homology of $(X,Y)$ is isomorphic to the reduced (co)homology of the topological quotient $X/Y$ (which is almost never an algebraic variety). This explains why the (co)homology of the pair $(X,Y)$ is insensitive to ``changes occurring purely inside $Y$'', such as (partial) resolutions of singularities. The relevant notion is as follows.
    
    \begin{defi}\label{defi:modification}
    Let $(X,Y)$ be a pair of complex varieties. A \emph{modification} of $(X,Y)$ is a morphism of pairs $f:(X',Y')\to (X,Y)$ where $f:X'\to X$ is proper, $Y'=f^{-1}(Y)$, and $f|_{X'\setminus Y'}:X'\setminus Y'\to X\setminus Y$ is an isomorphism.
    \end{defi}
    
    The main example of a modification is a blow-up
    $$\pi:(\widetilde{X},\widetilde{Y}\cup E) \to (X,Y)$$
    along a closed subvariety $Z\subset Y$, where $E=\pi^{-1}(Z)$ denotes the exceptional divisor, and $\widetilde{Y}$ the strict transform of $Y$. By Hironaka's theorem \cite{hironaka} on embedded resolution of singularities, for every pair $(X,Y)$ of complex varieties, there exists a modification $(X',Y')\to (X,Y)$ with $X'$ smooth and $Y'$ a normal crossing divisor.

\begin{prop}[Excision] 
    Let $f:(X',Y')\to (X,Y)$ be a modification. Then the natural pushforward and pullback maps induce isomorphisms in homology and cohomology:
    \begin{equation}\label{eq:iso-relative-cohomology-excision}
    f_*: \H_\bullet(X',Y')\stackrel{\sim}{\to} \H_\bullet(X,Y)\qquad\mbox{ and } \qquad f^*: \H^\bullet(X,Y)\stackrel{\sim}{\to} \H^\bullet(X',Y').
    \end{equation}
\end{prop}

     \begin{proof} 
    We give two proofs of the  equivalent statements \eqref{eq:iso-relative-cohomology-excision}. 
    \begin{enumerate}[1)]
    \item The assumptions imply that the continuous, bijective map $X'/Y'\to X/Y$ induced by $f$ on the topological quotients is a homeomorphism since  $X/Y$ is Hausdorff and $f$ is proper and hence closed. The claim follows from the interpretation of the cohomology of $X$ relative to $Y$ as the reduced cohomology of the topological quotient $X/Y$.  
    \item Using the language of sheaves. If we denote by $j: X\setminus Y \hookrightarrow X$ and $j': X'\setminus Y' \hookrightarrow X'$ the two open immersions, there is a  natural isomorphism of sheaves $f_*j'_!\QQ_{X'\setminus Y'}\simeq j_!\QQ_{X\setminus Y}$, since $f_*= f_!$ by properness of $f$. The claim follows on taking cohomology over $X$. 
    \end{enumerate}
    \end{proof}

    \subsubsection{Partial boundary maps}

    Let $Y_1,Y_2\subset X$ be closed subvarieties of $X$. The long exact sequence \eqref{eq:long-exact-sequence-relative-homology} for the triple $(X,Y_1\cup Y_2,Y_2)$ gives rise to a \emph{partial boundary map}
    \begin{equation}\label{eq:partial-boundary-map-defi}
    \partial_{Y_1}:\H_\bullet(X,Y_1\cup Y_2) \to \H_{\bullet-1}(Y_1\cup Y_2,Y_2)\simeq \H_{\bullet-1}(Y_1,Y_1\cap Y_2),
    \end{equation}
    where the isomorphism on the right follows from excision \eqref{eq:iso-relative-cohomology-excision} applied to the closed immersion of pairs  $(Y_1, Y_1 \cap Y_2)\hookrightarrow (Y_1 \cup Y_2,Y_2)$. The map \eqref{eq:partial-boundary-map-defi} sends (the class of) a relative cycle $\sigma\subset X$ with $\partial(\sigma)\subset Y_1\cup Y_2$ to ``the part of the boundary $\partial(\sigma)$ which is contained in  $Y_1$''.

    \subsubsection{A spectral sequence in relative (co)homology}

	Let $X$ be a complex variety, $Y_1,\ldots,Y_N$ be closed subvarieties of $X$, and write $Y=Y_1\cup\cdots \cup Y_N$. For every subset $I\subset \{1,\ldots,N\}$ we write 
    $$Y_I = \bigcap_{i\in I}Y_i$$
    with the convention $Y_\varnothing=X$. There is a spectral sequence in relative homology:
    \begin{equation}\label{eq:relative-homology-spectral-sequence}
    E^1_{p,q} = \bigoplus_{|I|=p}\H_q(Y_I) \quad \Longrightarrow \quad \H_{p+q}(X,Y),
    \end{equation}
    where the $d^1$ differentials are given by the (alternating sums of the) maps induced in homology by the inclusions $Y_{I\cup\{i\}}\subset Y_I$ with $i\notin I$.
    Dually, there is a spectral sequence in relative cohomology:
    \begin{equation}\label{eq:relative-cohomology-spectral-sequence}
    E_1^{p,q} = \bigoplus_{|I|=p}\H^q(Y_I) \quad \Longrightarrow \quad \H^{p+q}(X,Y).
    \end{equation}
    
\subsection{Mixed Hodge theory}

    \subsubsection{Pure Hodge structures}
			
		\begin{defi}\label{defi:pure-hodge-structure}
		Let $w\in\ZZ$. A \emph{pure Hodge structure} of weight $w$ is the data of a finite dimensional $\QQ$-vector space $H$ and a $\CC$-linear direct sum decomposition
		$$H_\CC = \bigoplus_{\substack{p,q\in\ZZ\\p+q=w}} H^{p,q},$$
        called the \emph{Hodge decomposition}, which satisfies the \emph{Hodge symmetry}
		$$\overline{H^{p,q}}=H^{q,p} \quad \mbox{ for all } p,q.$$
		The corresponding \emph{Hodge numbers} are defined to be  the dimensions 
        $$h^{p,q}=\dim_\CC H^{p,q}=\dim_\CC H^{q,p} = h^{q,p}.$$
		\end{defi}
		
		A landmark result of Hodge \cite{hodge} (see, e.g., \cite{voisin}) implies that if $X$ is a smooth projective complex variety, the cohomology group $\H^k(X)$ carries a natural pure Hodge structure of weight $k$, where the subspace $\H^{p,q}(X)$ appearing in the Hodge decomposition is the space of cohomology classes which can be represented by a smooth differential form of type $(p,q)$. We have an isomorphism:
        \begin{equation}\label{eq:hodge-p-q-as-sheaf-cohomology}
        \H^{p,q}(X)\simeq \H^q(X,\Omega^p_X).
        \end{equation}
		
		For a pure Hodge structure $H$ of weight $w$ one defines a decreasing filtration $\F$ on $H_\CC$, called the \emph{Hodge filtration}, defined by
		$$\F^pH_\CC := \bigoplus_{\substack{a+b=w\\ a\geq p}} H^{a,b}.$$
		One can recover the Hodge decomposition from the Hodge filtration: $H^{p,q} = \F^pH_\CC \cap \overline{\F^qH_\CC}$. A finite dimensional $\QQ$-vector space $H$ equipped with a filtration $\F$ on its complexification defines a pure Hodge structure of weight $w$ if and only if
		$$H_\CC = \F^pH_\CC \oplus \overline{\F^{w-p+1}H_\CC} \quad \mbox{ for all } p.$$
		This axiomatization of the notion of pure Hodge structure turns out to be more relevant than Definition \ref{defi:pure-hodge-structure} for the purpose of defining  mixed Hodge strucutres.

    \subsubsection{Mixed Hodge structures}
			
		\begin{defi}
		A \emph{mixed Hodge structure} is the data of a finite dimensional $\QQ$-vector space $H$ and
		\begin{itemize}
		\item an increasing filtration $\W$ on $H$, called the \emph{weight filtration},
		\item a decreasing filtration $\F$ on $H_\CC$, called the \emph{Hodge filtration},
		\end{itemize}
		such that for every integer $w$, the filtration induced by $\F$ on $\gr_w^\W H:=\W_wH/\W_{w-1}H$ defines a pure Hodge structure of weight $w$. The corresponding \emph{Hodge numbers} are the dimensions 
        $$h^{p,q} = \dim_\CC H^{p,q} = h^{q,p}  \qquad \mbox{ where } \qquad  H^{p,q} := (\gr_{p+q}^\W H)^{p,q}.$$
		\end{defi}
		
		As a first approximation, one can view a mixed Hodge structure $H$ as the collection of the pure Hodge structures $\gr_w^\W H$ of different weights $w$. (However it contains more subtle ``extension data'' between those pure components, which will not play a role in this article.)
  
		Deligne proved \cite{delignehodge1, delignehodge2, delignehodge3} that the cohomology groups of all complex varieties carry natural mixed Hodge structures, which coincide with the pure Hodge structures of classical Hodge theory in the smooth projective case. More generally, there are natural mixed Hodge structures on the relative (co)homology groups $\H_k(X,Y)$ and $\H^k(X,Y)$ for all pairs $(X,Y)$ of complex varieties. We note that if $X$ has dimension $n$, then for all $k$,
        \begin{equation}\label{eq:bound-on-hodge-numbers}
        \H^k(X,Y)^{p,q}\neq 0 \quad \Longrightarrow \quad  \begin{cases} p,q\in\{0,\ldots,k\} & \mbox{ if } k\leq n; \\ p,q\in\{k-n,\ldots,n\} & \mbox{ if } k\geq n. \end{cases}
        \end{equation}
        This is a consequence of \cite[Th\'{e}or\`{e}me 8.2.4 (i), (ii)]{delignehodge3} and the long exact sequence in relative cohomology \eqref{eq:long-exact-sequence-relative-cohomology} with $Z=\varnothing$. Also,
        \begin{equation}\label{eq:bound-on-hodge-numbers-smooth}
        \mbox{for $X$ smooth,}\qquad \H^k(X)^{p,q}\neq 0 \quad \Longrightarrow \quad p+q\geq k
        \end{equation}
        by \cite[Th\'{e}or\`{e}me 8.2.4, (iv)]{delignehodge3}.

    \subsubsection{The category of mixed Hodge structures}

        A morphism between mixed Hodge structures is by definition a $\QQ$-linear map which is compatible with the weight filtration, and whose complexification is compatible with the Hodge filtration. Mixed Hodge structures form an abelian category and the  functors $H\mapsto \gr_w^\W H$ and $H\mapsto \gr_\F^p H_\CC$ are  exact for every $w, p$,  which allows one to compute Hodge numbers via long exact sequences and spectral sequences. 
        
        If $f:(X',Y')\to (X,Y)$ is a morphism of pairs of complex varieties, then the pushforward and pullback maps $f_*:\H_\bullet(X',Y')\to \H_\bullet(X,Y)$ and $f^*:\H^\bullet(X,Y)\to \H^\bullet(X',Y')$ are morphisms of mixed Hodge structures. The long exact sequences \eqref{eq:long-exact-sequence-relative-homology} and \eqref{eq:long-exact-sequence-relative-cohomology}, the excision isomorphism \eqref{eq:iso-relative-cohomology-excision}, the partial boundary map \eqref{eq:partial-boundary-map-defi}, and the spectral sequences \eqref{eq:relative-homology-spectral-sequence} and \eqref{eq:relative-cohomology-spectral-sequence}, all exist in the category of mixed Hodge structures.
  
        The tensor product of two mixed Hodge structures $H_1$, $H_2$ is the vector space $H_1\otimes H_2$ with weight filtration $\W_n (H_1 \otimes H_2) = \sum_{i+j=n} \W_i H_1 \otimes \W_j H_2$ and  Hodge filtration $\F^n (H_1 \otimes H_2)_\CC = \sum_{i+j=n} \F^i (H_1)_{\CC} \otimes \F^j (H_2)_{\CC}$. The K\"{u}nneth isomorphism and the cup-product map in cohomology are compatible with mixed Hodge structures.
        
        The dual of  a mixed Hodge structure $H$ is the vector space $H^{\vee}= \mathrm{Hom}(H,\QQ)$  equipped with the weight filtration $\W_{-n} H^{\vee} =(H/\W_{n-1}H)^{\vee}= \ker (H^{\vee} \rightarrow (\W_{n-1}H)^{\vee})$ and the Hodge filtration $\F^{-n} H^{\vee}_{\CC} =(H_{\CC} /\F^{n+1} H_{\CC})^{\vee}=\ker(H_\CC^\vee\to (\F^{n+1}\H_\CC)^\vee)$. In terms of Hodge numbers, $h^{-p,-q}(H^\vee)=h^{p,q}(H)$ for all $p,q\in\ZZ$. The duality between relative homology and cohomology is compatible with mixed Hodge structures.

    \subsubsection{The mixed Tate case}

        For every $n\in\mathbb{Z}$, the \emph{pure Tate} Hodge structure $\QQ(-n)$ is the pure Hodge structure of weight $2n$ whose underlying vector space is $\QQ$. Its only non-zero Hodge number is $h^{n,n}=1$. For a mixed Hodge structure $H$, the \emph{Tate twist} $H(-n):=H\otimes \QQ(-n)$ satisfies $\W_{w+2n} H(-n)= \W_{w}H$ for all $w$ and $\F^{p+n} H(-n)_\CC = \F^p H_\CC$ for all $p$. If $H$ is pure of weight $w$ then $H(-n)$ is pure of weight $w+2n$.

        A Hodge structure $H$ is said to be \emph{mixed Tate} (or to be a \emph{mixed Hodge--Tate structure}) if its Hodge numbers satisfy $h^{p,q}=0$ for $p\neq q$, or equivalently if its weight-graded quotients $\gr_w^\W H$ are direct sums of pure Tate Hodge structures $\QQ(-w/2)$ for even $w$, and zero for odd $w$.

\subsection{Poincar\'{e} duality}

    Let $X$ be a compact complex variety of dimension $n$ and $Y\subset X$ a closed subvariety such that $X\setminus Y$ is smooth. Recall \eqref{eq:relative-homology-is-lf-homology} that the relative cohomology of the pair $(X,Y)$ is isomorphic to the compactly supported cohomology of $X\setminus Y$. Poincar\'{e} duality therefore produces an isomorphism of mixed Hodge structures:
    \begin{equation}\label{eq:poincare-duality}
    \H_k(X,Y) \stackrel{\sim}{\to} \H^{2n-k}(X\setminus Y)(n).
    \end{equation}
    By definition, \eqref{eq:poincare-duality} sends the class of a relative $k$-cycle $\sigma$ in $X$ with $\partial\sigma\subset Y$ to the class of a smooth closed $(2n-k)$-form $\varphi$ on $X\setminus Y$ which is characterized by the fact that for every smooth closed $k$-form $\eta$ on $X\setminus Y$ with compact support, the following equality holds: 
    \begin{equation}\label{eq:poincare-duality-integration}
    (2\pi\mathrm{i})^{-n} \int_X \varphi\wedge \eta = \int_\sigma\eta.
    \end{equation}

\subsection{Logarithmic forms, residues, and the mixed Hodge theory of smooth varieties}\label{sec:MHS-smooth-varieties}

	We discuss Deligne's construction \cite{delignehodge2, voisin} of the mixed Hodge structure on the cohomology of a smooth complex variety $U$. It uses a smooth compactification $\overline{U}$ of $U$ such that $D=\overline{U}\setminus U$ is a normal crossing divisor, but the mixed Hodge structure on the cohomology of $U$ does not depend on this choice. Such a compactification always exists by Nagata's embedding theorem \cite{nagata} and Hironaka's resolution of singularities \cite{hironaka}.
 
\subsubsection{Logarithmic forms and residues}

     The starting point of the construction is the fact that the cohomology $\H^\bullet(U)_\CC$ can be computed as the hypercohomology of the logarithmic de Rham complex of sheaves
	$$\Omega^\bullet_{\overline{U}}(\log D).$$
	The sections of the latter are the holomorphic forms on $U$ with logarithmic poles along $D$. Concretely, in local holomorphic coordinates $z_1,\ldots,z_n$ on $\overline{U}$ for which $D=\{z_1\cdots z_r=0\}$, they are the holomorphic forms which can be expressed as $\CC$-linear combinations of forms 
	\begin{equation}\label{eq:local-expression-log-form}
	\alpha\wedge \dlog(z_{i_1})\wedge\cdots \wedge \dlog(z_{i_s}),
	\end{equation}
	where $\alpha$ is holomorphic on $\overline{U}$ and $1\leq i_1<\cdots <i_s\leq r$.

    We recall the notion of the  residue of a logarithmic form, in the case where $D$ is a \emph{simple} normal crossing divisor, i.e., whose irreducible components are smooth. Let us denote these components by $D_1,\ldots,D_N$. 
     The residue along $D_N$ of a $k$-form on $\overline{U}$ with logarithmic poles along $D$ is a $(k-1)$-form on $D_N$ with logarithmic poles along $D_N\cap (D_1\cup\cdots\cup D_{N-1})$. Locally, if $D_N$ is defined by an equation $z_N=0$, we have
        \begin{equation}\label{eq:definition-residue}
        \Res_{D_N}\left(\omega\wedge \dlog(z_N)\right) =\omega|_{D_N}, 
        \end{equation}
        for any $(k-1)$-form $\omega$ with logarithmic poles along $D_1\cup\cdots \cup D_{N-1}$.

        \begin{rem}\label{rem:sign-convention-residue}
        The residue operation commutes with the exterior differential: $\Res\circ \dd = \dd\circ \Res$. However, since $\Res$ decreases the degree of forms by $1$, the Koszul sign convention would normally require that $\mathrm{Res}\circ\dd = - \dd \circ \mathrm{Res}$.  This is why in some references a different sign convention is adopted for residues, where $\omega\wedge \dlog(z_N)$ is replaced with $\dlog(z_N)\wedge \omega$ in \eqref{eq:definition-residue}. This multiplies the residue of $k$-forms by the overall sign $(-1)^{k-1}$. Our choice \eqref{eq:definition-residue} of sign convention will make certain formulas simpler, most notably the recursion formula in Proposition \ref{prop:recursion}.
        \end{rem}
        
\subsubsection{The Hodge filtration, and forms with logarithmic poles at infinity}
	
	The Hodge filtration on the cohomology of $U$ is induced by the ``stupid'' filtration of the logarithmic de Rham complex,
	$$\F^p\Omega^\bullet_{\overline{U}}(\log D) := \Omega^{\geq p}_{\overline{U}}(\log D).$$
	The corresponding hypercohomology spectral sequence degenerates at the $E_1$ page, giving rise to isomorphisms
	\begin{equation}\label{eq:cohomology-of-log-forms-hodge-filtration}
    \gr_\F^p\H^{p+q}(U)_\CC \simeq \H^q(\overline{U},\Omega^p_{\overline{U}}(\log D)).
    \end{equation}
    We now focus on the $q=0$ case.
	
	\begin{defi}
	The vector space of \emph{holomorphic $k$-forms on $U$ with logarithmic poles at infinity} is the space of global sections of $\Omega^k_{\overline{U}}(\log D)$, denoted by
	$$\logforms{k}(U) := \H^0(\overline{U},\Omega^k_{\overline{U}}(\log D)).$$ 
	\end{defi}

    The next proposition shows that it is a space of forms on $U$ which is independent of the choice of compactification $\overline{U}$, and so the notation 
    $\logforms{k}(U)$ is unambiguous.

    \begin{prop}\label{prop:logforms-independent-of-compactification}
    The subspace $\logforms{k}(U)\subset \H^0(U,\Omega^k_U)$ is independent of the choice of $\overline{U}$.
    \end{prop}

    \begin{proof}
    We repeat Deligne's argument from \cite[Th\'{e}or\`{e}me 3.2.5 (ii)]{delignehodge2}, which proves that the image of $\logforms{k}(U)$ inside the cohomology space $\H^k(U)_\CC$ is independent of the choice of $\overline{U}$. Let $\overline{U}_1$ and $\overline{U}_2$ be two compactifications of $U$, with $D_i=\overline{U}_i\setminus U_i$ a normal crossing divisor for $i=1,2$. By Hironaka's resolution of singularities \cite{hironaka}, one can find a third compactification $\overline{U}$ of $U$, with $D=\overline{U}\setminus U$ a normal crossing divisor, which fits into the following commutative diagram, where $\pi_1$ and $\pi_2$ are the identity on $U$: 
    $$\diagram{
    & \overline{U} \ar[dl]_{\pi_1} \ar[rd]^{\pi_2}& \\
    \overline{U}_1 & & \overline{U}_2 \\
    & U \ar@{_(->}[ul]^{j_1} \ar@{_(->}[uu]^j \ar@{_(->}[ur]_{j_2}& 
    }$$
    By functoriality of the sheaves of logarithmic forms we therefore deduce  the following commutative diagram of pullbacks:
    $$
    \diagram{
    & \H^0(\overline{U},\Omega^k_{\overline{U}}(\log D)) \ar@{^(->}[dd]_{j^*} & \\
    \H^0(\overline{U}_1,\Omega^k_{\overline{U}_1}(\log D_1)) \ar[ru]^{\pi_1^*} \ar@{^(->}[rd]_{j_1^*} & & \H^0(\overline{U}_2,\Omega^k_{\overline{U}_2}(\log D_2)) \ar[lu]_{\pi_2^*} \ar@{^(->}[ld]^{j_2^*} \\
    & \H^0(U,\Omega^k_U) & 
    }$$
    By \eqref{eq:cohomology-of-log-forms-hodge-filtration}, $\pi_1^*$ and $\pi_2^*$ are isomorphisms, and therefore $\mathrm{Im}(j_1^*)=\mathrm{Im}(j^*)=\mathrm{Im}(j_2^*)$.
    \end{proof}
	
	By the degeneration of the hypercohomology spectral sequence for the Hodge filtration at the $E_1$ page, all forms on $U$ with logarithmic poles at infinity are closed, and the map $\logforms{k}(U)\to \H^k(U)_\CC$ sending such a form to its cohomology class induces an identification with the last step of the Hodge filtration:
	\begin{equation}\label{eq:log-forms-to-cohomology}
    \logforms{k}(U)\simeq \F^k\H^k(U)_\CC\subset \H^k(U)_\CC.
    \end{equation}

    Note that forms with logarithmic poles can be multiplied together, making $\logforms{\bullet}(U)$ into a graded algebra.
 
    \begin{rem}
    By Serre's GAGA theorem \cite{serreGAGA} applied to $\Omega^k_{\overline{U}}(\log D)$, all elements of $\logforms{k}(U)$ are automatically \emph{algebraic}, for example, if $U$ is an open in a closed subvariety of projective space $\PP^N_\CC$ with homogeneous coordinates $x_0,\ldots, x_N$, then such a form may be written  as a linear combination of  forms $\frac{f}{g}  \dd x_{i_1} \wedge \ldots \wedge \dd x_{i_k}$,  for some polynomials $f, g \in \CC[x_0, \ldots, x_N]$.
    \end{rem}

    A particularly interesting phenomenon arises in the following extreme case.

    \begin{prop}\label{prop:2-pure}
    If $\,\H^k(U)$ is pure of weight $2k$ then we have a canonical isomorphism
    $$\logforms{k}(U)\simeq \H^k(U)_\CC.$$
    \end{prop}

    \begin{proof}
    If $\H^k(U)$ is pure of weight $2k$, then by \eqref{eq:bound-on-hodge-numbers} the Hodge numbers of $\H^k(U)$ satisfy $h^{p,q}=0$ for $(p,q)\neq (k,k)$, and hence $\F^k\H^k(U)_\CC=\H^k(U)_\CC$. The result then follows from \eqref{eq:cohomology-of-log-forms-hodge-filtration}.
    \end{proof}

    In particular, if $\H^k(U)$ is pure of weight $2k$ for all $k$ then we have a canonical ``formality'' isomorphism of graded algebras $\logforms{\bullet}(U)\simeq \H^\bullet(U)_\CC$. In other words, all cohomology classes can be canonically represented as global forms with logarithmic poles at infinity, in such a way that algebraic relations among cohomology classes are already satisfied by their representatives. This phenomenon arises, e.g., if $U$ is the complement of an arrangement of hyperplanes in affine or projective space, by a result of Brieskorn \cite{brieskorn}.

\subsection{Which forms have logarithmic poles at infinity?}

    \subsubsection{Some general remarks}

    The following proposition, which is a special case of the more general functoriality of logarithmic forms with respect to morphisms, gives a simple way to construct a supply of forms with logarithmic poles at infinity.

    \begin{prop}
    Let $U$ be a smooth complex variety, and let $f:U\to \CC^*$ be a non-vanishing holomorphic function. Then $\mathrm{dlog}(f):=\frac{\dd f}{f}\in \logforms{1}(U)$ is logarithmic. 
    \end{prop}
    
    On the other hand, we note that forms with logarithmic poles at infinity have at most simple poles in the sense of the following definition, which we only state in the case when the ambient variety $X$ is  smooth,  for simplicity. A \emph{hypersurface} $Y\subset X$ is a subvariety of codimension $1$.

    \begin{defi}
    Let $X$ be a smooth compact complex variety of dimension $n$, and $Y\subset X$ be a hypersurface. A holomorphic $n$-form $\omega$ on $X\setminus Y$  \emph{has at most a simple pole along $Y$}  if, for every local equation $f$ for $Y$, the product $f\omega$ extends to a holomorphic $n$-form on $X$.
    \end{defi} 

    The space of holomorphic $n$-forms on $X\setminus Y$ with at most a simple pole along $Y$ is therefore  the space of global sections of the line bundle $\Omega^n_X(Y)$, which we denote by
    $$\mathcal{S}^n(X,Y):=\H^0(X,\Omega^n_X(Y)).$$

    \begin{prop}\label{prop:log-forms-have-simple-poles}
    Let $X$ be a smooth compact complex variety of dimension $n$, and $Y\subset X$ be a hypersurface. We have the inclusion
    \begin{equation}\label{eq:inclusion-log-forms-simple-poles}
    \logforms{n}(X\setminus Y) \subset \mathcal{S}^n(X,Y),
    \end{equation}
    which is an equality if $Y$ is a normal crossing divisor.
    \end{prop}

    \begin{proof}
    Let us write $Y^{\mathrm{sing}}$ for the singular locus of $Y$, and $Y^{\mathrm{reg}}=Y\setminus Y^{\mathrm{sing}}$. Consider the open $U = X\setminus Y^{\mathrm{sing}}$, and note that $Y^{\mathrm{reg}}$ is a smooth subvariety of $U$. By embedded resolution of singularities, there is a modification $\pi:(X',Y')\to (X,Y)$ where $X'$ is smooth and $Y'$ is a normal crossing divisor. We can further assume that this modification is an isomorphism over $U$. Let $\omega\in \logforms{n}(X\setminus Y)$, then by definition $\pi^*\omega$ has logarithmic poles along $Y'$, and since $\pi$ is an isomorphism over $U$, the restriction $\omega|_U$ has logarithmic poles along $Y^{\mathrm{reg}}$. Therefore, $\omega|_U$ has at most a simple pole along $Y^{\mathrm{reg}}$, and for every local equation $f$ for $Y$, the product $(f\omega)|_U$ extends to a holomorphic $n$-form on $U$.
    Since $X\setminus U$ has codimension $\geq 2$ in $X$, the first claim then follows by Hartogs' theorem. The fact that all forms with at most simple poles are logarithmic in the normal crossing case is clear and follows from the definition. 
    \end{proof}

    \begin{rem}\label{rem:log-forms-on-projective-space}
    For $Y\subset \PP^n_\CC$ a hypersurface of degree $d$, we have an isomorphism of line bundles 
    $$\Omega^n_{\PP^n_\CC}(Y)\simeq \mathcal{O}(d-n-1).$$
    Therefore, if $d\geq n+1$, one can identify $\mathcal{S}^n(\PP^n_\CC, Y)$ with the space of homogeneous polynomials of degree $d-n-1$ in $n+1$ variables. Concretely, choosing homogeneous coordinates $x_0,\ldots, x_n$ on $\PP^n_\CC$ and a homogeneous polynomial $f=f(x_0,\ldots,x_n)$ of degree $d$ whose vanishing locus is $Y$, the elements of $\mathcal{S}^n(\PP^n_\CC, Y)$ are the forms
    $$\frac{\varphi}{f}\sum_{i=0}^n (-1)^i x_i \,\dd x_0\wedge\cdots\wedge \widehat{\dd x_i}\wedge \cdots\wedge \dd x_n,$$
    for $\varphi=\varphi(x_0,\ldots,x_n)$ a homogeneous polynomial of degree $d-n-1$. In particular, we have
    $$\dim\logforms{n}(\PP^n_\CC\setminus Y)\leq \binom{d-1}{n},$$
    with equality if $Y$ is a normal crossing divisor.
    \end{rem}
    
    The inclusion \eqref{eq:inclusion-log-forms-simple-poles} can be strict if $Y$ is not a normal crossing divisor (see, e.g., \S\ref{sec:cuspidal-cubic-and-line} for concrete computations). In the next paragraph, we prove a useful criterion to decide which forms, among those with at most simple poles, are logarithmic at infinity.

    \subsubsection{A useful criterion}

    Let $X$ be a smooth compact complex variety of dimension $n$, and $Y\subset X$ be a hypersurface. Let us write $Y^{\mathrm{sing}}$ for the singular locus of $Y$, and $Y^{\mathrm{reg}}=Y\setminus Y^{\mathrm{sing}}$. When restricted to $X\setminus Y^{\mathrm{sing}}$, a form in $\mathcal{S}^n(X,Y)$ is an $n$-form with logarithmic poles along the smooth subvariety $Y^{\mathrm{reg}}$, hence there is a well-defined residue map
    \begin{equation}\label{eq:residue-for-simple-poles}
    \Res_Y:\mathcal{S}^n(X,Y)\to \H^0(Y^{\mathrm{reg}},\Omega^{n-1}_{Y^{\mathrm{reg}}}).
    \end{equation}
    This map allows one to detect forms which  have logarithmic poles at infinity, as follows.

    \begin{prop}\label{prop:criterion-logarithmic}
    Let $X$ be a smooth compact complex variety of dimension $n$, and let  $Y\subset X$ be a hypersurface. Let $\omega\in\mathcal{S}^n(X,Y)$ be a form on $X\setminus Y$ with at most a simple pole along $Y$. 
    We have the equivalence: 
    $$\omega\in \logforms{n}(X\setminus Y) \qquad \Longleftrightarrow \qquad \Res_Y(\omega)\in \logforms{n-1}(Y^{\mathrm{reg}}).$$ 
    \end{prop}

    \begin{proof}
    We consider the residue morphism in cohomology,
    $$\Res_Y: \H^n(X\setminus Y) \to \H^{n-1}(Y^{\mathrm{reg}})(-1),$$
    which can be  computed by the map \eqref{eq:residue-for-simple-poles} on classes of $n$-forms which have at most simple poles. It is a morphism of mixed Hodge structures, which implies in particular that
    $$\Res_Y(\F^n\H^n(X\setminus Y)_\CC) \subset \F^{n-1}\H^{n-1}(Y^{\mathrm{reg}})_\CC.$$
    Thanks to \eqref{eq:log-forms-to-cohomology}, this proves the implication $\Longrightarrow$. For the reverse implication, we use the fact that morphisms of mixed Hodge structures are \emph{strictly} compatible with the Hodge filtration \cite[Th\'{e}or\`{e}me 2.3.5 (iii)]{delignehodge2} (this is essentially equivalent to the exactness of the functors $\gr^p_F$). This implies that if $\Res_Y(\omega)\in \logforms{n-1}(Y^{\mathrm{reg}})\simeq \F^{n-1}\H^{n-1}(Y^{\mathrm{reg}})_\CC$, then there exists a form $\eta\in \logforms{n}(X\setminus Y)$ such that $\Res_Y([\omega])=\Res_Y([\eta])$. Since $\Res_Y(\omega)$ and $\Res_Y(\eta)$ are two elements of $\logforms{n-1}(Y^{\mathrm{reg}})$ with the same cohomology class,  \eqref{eq:log-forms-to-cohomology} implies that they are equal. Thus, the difference $\omega-\eta$ is a holomorphic $n$-form on $X\setminus Y$ with at most simple poles along $Y$, and with zero residue along $Y^{\mathrm{reg}}$, which means that it extends to a holomorphic $n$-form on $X\setminus Y^{\mathrm{sing}}$. Since $X$ is smooth and $Y^{\mathrm{sing}}$ has codimension $\geq 2$ in $X$, Hartogs' theorem implies that $\omega-\eta$ extends to a holomorphic $n$-form on $X$, and the claim follows.
    \end{proof}

\section{Canonical forms}\label{sec:canonical-forms}

In this section we define canonical forms and prove our main theorem.

\subsection{Genus, rank, and the \texorpdfstring{$R$}{R}-map}

    Let $X$ be a complex variety of dimension $n$, and $Y\subset X$ be a closed subvariety. By \eqref{eq:bound-on-hodge-numbers}, the only non-vanishing Hodge numbers of the relative homology group $\H_n(X,Y)$ are of the form $h^{-p,-q}$ for $0\leq p,q\leq n$ and thus lie in a square.  Among those, the ones of relevance to the study of canonical forms lie along two edges of this square: 
    \begin{enumerate}
    \item The sum of the  Hodge numbers $h^{-p,0}=h^{0,-p}$ for $p>0$ will be called  the \emph{genus}.  It measures the  obstruction to the uniqueness of canonical forms: canonical forms are therefore only well-defined in the genus zero case.
    \item We shall call the  Hodge number $h^{0,0}$  the \emph{combinatorial rank}. In the genus zero case, it equals the dimension of the space of canonical forms.
    \end{enumerate}

    \subsubsection{Genus}
    
    The following terminology is non-standard.
    
    \begin{defi}\label{defi:genus}
    The \emph{genus} of $(X,Y)$ is
    $$g(X,Y):=\sum_{p>0}h^{-p,0}(\H_n(X,Y)).$$
    \end{defi}

    We will study the notion of genus in more detail in \S\ref{sec:genus} below. Until then, let us simply mention the bound
    $$g(X,Y)\geq g(X):=g(X,\varnothing),$$
    and the fact that, in the case when $X$ is smooth and compact of dimension $n\geq 1$, then
    $$g(X)=\dim \H^0(X,\Omega^n_X)$$
    coincides with the classical \emph{geometric genus}.
    
    \subsubsection{The combinatorial rank}

    Consider the weight $0$ quotient 
    $$\gr_0^\W\H_n(X,Y) = \H_n(X,Y)/\W_{-1}\H_n(X,Y).$$
    It underlies a pure Hodge structure of weight $0$ and type $(0,0)$.

    \begin{defi}\label{defi:rank}
    The \emph{combinatorial rank} of $(X,Y)$ is the dimension of the weight $0$ quotient of $\H_n(X,Y)$, or equivalently the Hodge number $h^{0,0}$ of $\H_n(X,Y)$, denoted by
    $$\rk(X,Y):=\dim \gr_0^\W\H_n(X,Y) = h^{0,0}(\H_n(X,Y)).$$
    \end{defi}

    We will study the notion of combinatorial rank in more detail in \S\ref{sec:combinatorial-rank} below. We note the equality
    \begin{equation} \label{eq: g-plus-rk-equals-dimF0} g(X,Y) + \rk(X,Y) = \dim\F^0\H_n(X,Y).
    \end{equation}

\subsubsection{The $R$-map}

    The genus and the combinatorial rank are related as follows. 

    \begin{lem}\label{lem:map-from-F0-to-H00}
    The natural linear map
    \begin{equation}\label{eq:map-from-F0-to-H00}
    \F^0\H_n(X,Y)_\CC \to  \H_n(X,Y)_\CC/\W_{-1}\H_n(X,Y)_\CC = \gr_0^\W\H_n(X,Y)_\CC
    \end{equation}
    is surjective and its kernel has dimension equal to the genus of $(X,Y)$. 
    \end{lem}

    \begin{proof}
    Since the pure Hodge structure $H=\gr^W_0 \H_n(X,Y)$ has  weight $0$ and type $(0,0)$, we have  $H_\CC=\F^0H_\CC$, which is equivalent to \eqref{eq:map-from-F0-to-H00} being surjective. Since the dimension of $\F^0\H_n(X,Y)_\CC$ is the sum of the Hodge numbers $h^{0,-p}=h^{-p,0}$ for $p\geq 0$, the second claim follows. 
    \end{proof}

    Using mixed Hodge theory and Poincar\'{e} duality, we can now recast \eqref{eq:map-from-F0-to-H00} as a map involving logarithmic forms, as follows.
    \begin{defi}
    Let $X$ be a compact complex variety of dimension $n$, and $Y\subset X$ be a closed subvariety such that $X\setminus Y$ is smooth. We define a surjective linear map
    \begin{equation}\label{eq:R-map}
R\,:\, \logforms{n}(X\setminus Y) \twoheadrightarrow \gr_0^\W\H_n(X,Y)_\CC
    \end{equation}
    by composing the following three maps:
    \begin{enumerate}[a)]
    \item the isomorphism \eqref{eq:log-forms-to-cohomology},
    $$\logforms{n}(X\setminus Y)\simeq \F^n\H^n(X\setminus Y)_\CC;$$
    \item the (inverse of the) Poincar\'{e} duality isomorphism \eqref{eq:poincare-duality} in degree $k=n$,
    $$\F^n\H^n(X\setminus Y)_\CC \simeq \F^0\H_n(X,Y)_\CC;$$
    \item the surjection \eqref{eq:map-from-F0-to-H00},
    $$\F^0\H_n(X,Y)_\CC\twoheadrightarrow \gr_0^\W\H_n(X,Y)_\CC.$$
    \end{enumerate}
    \end{defi}

 If $Y=D$ is a simple normal crossing divisor, the $R$-map \eqref{eq:R-map} is computed by taking $n$-fold iterated residues along the corners (zero-dimensional strata) of $D$, see Proposition \ref{prop:corners}. The symbol $R$ is intended to  suggest  an iterated residue. 

The following lemma is a restatement of Lemma \ref{lem:map-from-F0-to-H00}.

\begin{lem}\label{lem:R-map-and-genus}
The kernel of the $R$-map \eqref{eq:R-map} has dimension equal to the genus of $(X,Y)$.
\end{lem}

\subsection{Definition of canonical forms}

By  Lemma \ref{lem:R-map-and-genus}, the $R$-map is an isomorphism if (and only if) $(X,Y)$ has genus zero, which allows us to make the following definition.

\begin{defi}\label{defi:can-compact}
Let $X$ be a compact complex variety of dimension $n$, and $Y\subset X$ be a closed subvariety such that $X\setminus Y$ is smooth. Assume that $(X,Y)$ has genus zero. We define the morphism 
\begin{equation}\label{eq:can-defi}
\can\colon \H_n(X,Y)\to \logforms{n}(X\setminus Y)
\end{equation}
to be the composition of the quotient map
$$\H_n(X,Y)\twoheadrightarrow \gr_0^\W\H_n(X,Y)$$
with the inverse of the $R$-map
$$\logforms{n}(X\setminus Y)\stackrel{\sim}{\to} \gr_0^\W\H_n(X,Y)_\CC.$$
For a class $\sigma\in \H_n(X,Y)$ we write
$$\omegacan_\sigma:=\can(\sigma)$$
and call it the \emph{canonical form} of $\sigma$.
\end{defi}

\begin{rem} \label{rem: every-log-form-is-canonical}
In the setting of the previous definition, extending scalars to $\CC$ gives rise to a surjective $\CC$-linear map
\begin{equation}\label{eq:can-complexified}
\can_\CC:\H_n(X,Y)_\CC \twoheadrightarrow \logforms{n}(X\setminus Y),
\end{equation}
and therefore \emph{every} holomorphic $n$-form on $X\setminus Y$ with logarithmic poles at infinity is the canonical form of a relative $n$-cycle of $(X,Y)$ with coefficients in $\CC$.
\end{rem}

\begin{rem}\label{rem:2-pure-canonical-forms}
Let $X$ be a compact complex variety of dimension $n$, and $Y\subset X$ be a closed subvariety such that $X\setminus Y$ is smooth. If $\H^n(X\setminus Y)$ is pure of weight $2n$, then Poincar\'{e} duality \eqref{eq:poincare-duality} implies that $\H_n(X,Y)$ is pure of weight $0$ and type $(0,0)$, i.e.,  $\H_n(X,Y)\simeq \gr_0^\W\H_n(X,Y)$. Therefore, $(X,Y)$ has genus zero and combinatorial rank equal to the dimension of $\H_n(X,Y)$. In this case, the morphism \eqref{eq:can-complexified} is an isomorphism.
\end{rem}

\subsection{Basic properties of canonical forms}

In this paragraph, $X$ is a compact complex variety of dimension $n$ and $Y\subset X$ is a closed subvariety such that $X\setminus Y$ is smooth and $(X,Y)$ has genus zero. We state and prove the compatibilities between the morphisms \eqref{eq:can-defi} promised in the introduction, only providing a proof when a result is not an easy consequence of the definition.

\subsubsection{Invariance under modification}

If $f:(X',Y')\to (X,Y)$ is a modification, then $(X',Y')$ has genus zero because of the isomorphism of mixed Hodge structures \eqref{eq:iso-relative-cohomology-excision} for $\bullet=n$. The pairs $(X,Y)$ and $(X',Y')$ have ``the same canonical forms'' because we have the following commutative diagram whose horizontal maps are isomorphisms:
\begin{equation}\label{eq:can-and-modification}
\begin{gathered}
\diagram{
\H_n(X',Y') \ar[r]^-{\can} \ar[d]_-{f_*}^{\sim} & \logforms{n}(X'\setminus Y') \\
\H_n(X,Y) \ar[r]^-{\can} & \logforms{n} (X\setminus Y) \ar[u]_-{f^*}^-{\sim}
}
\end{gathered}
\end{equation}

\begin{rem}
Invariance under modification is immediate in our formalism. By contrast, in the setting of positive geometries in the sense of \cite{arkanihamedbailampositive}, invariance under modifications is more subtle \cite{wondertopes}.
\end{rem}

\subsubsection{Spurious boundary, spurious poles}

Sometimes one is presented with a domain $\sigma$ in $X$ whose boundary is contained in a smaller subvariety $Y'\subset Y$ of $X$. In this case, the corresponding canonical form $\omegacan_\sigma$ is the same whether one views $\sigma$ as a class in $\H_n(X,Y)$ or $\H_n(X,Y')$, and therefore only has poles along the subvariety $Y'\subset Y$.

In other words, we have the following commutative diagram:
\begin{equation}\label{eq:irrelevant-boundary}
\begin{gathered}
\diagram{
\H_n(X,Y') \ar[d]\ar[r]^-{\can} & \logforms{n}(X\setminus Y') \ar@{^(->}[d] \\
\H_n(X,Y) \ar[r]^-{\can} & \logforms{n}(X\setminus Y)
}
\end{gathered}
\end{equation}
\subsubsection{Linearity}

By definition, the morphism \eqref{eq:can-defi} is $\QQ$-linear, which amounts to the following identities on canonical forms:
\begin{equation}\label{eq:can-linearity} 
\omegacan_{\sigma+\sigma'} = \omegacan_\sigma + \omegacan_{\sigma'}\quad \mbox{ and } \quad \omegacan_{a\sigma}=a\omegacan_\sigma, 
\end{equation}
for $\sigma,\sigma'\in\H_n(X,Y)$ and $a\in\QQ$. In particular, reversing the orientation  of $\sigma$ multiplies the canonical form by $-1$:
$$\omegacan_{-\sigma} = -\omegacan_\sigma.$$
This linearity property allows one to compute canonical forms by ``triangulating'', i.e., writing a chain representing $\sigma$ as a linear combination of chains.
Note that new boundary components may appear in the individual chains constituting the triangulation (e.g., when cutting a square along a diagonal to form two triangles), and one implicitly uses \eqref{eq:irrelevant-boundary} to justify the computation.

\subsubsection{Functoriality for pushforwards}

Let $X'$ be another compact complex variety of the same dimension $n$, and $Y'\subset X'$ be a closed subvariety such that $X'\setminus Y'$ is smooth and $(X',Y')$  also has genus zero. Let $f:X'\to X$ be a morphism such that $f(Y')\subset Y$.  We have a pushforward map in relative homology: 
$$f_*:\H_\bullet(X',Y')\to \H_\bullet(X,Y),$$
which, owing to Poincar\'{e} duality \eqref{eq:poincare-duality},  gives rise to a pushforward map:
$$f_*:\H^\bullet(X'\setminus Y')\to \H^\bullet(X\setminus Y).$$
This is a morphism of mixed Hodge structures, and therefore after applying the Hodge filtration functor $\F^\bullet$ (more precisely, $\F^k$ on $\H^k$, see \eqref{eq:log-forms-to-cohomology}) we deduce  a pushforward map:
$$f_*:\logforms{\bullet}(X'\setminus Y')\to \logforms{\bullet}(X\setminus Y).$$
This can be thought of as  a trace, or ``summation over  fibers''   when the fibers of $f$ are finite sets of points. We have the following commutative diagram: 
$$\diagram{
\H_n(X',Y') \ar^-{\can}[r] \ar[d]_-{f_*} & \logforms{n}(X'\setminus Y') \ar[d]^-{f_*}\\
\H_n(X,Y) \ar^-{\can}[r] & \logforms{n}(X\setminus Y)
}$$
In other words, for every $\sigma'\in\H_n(X',Y')$ we have the equality
\begin{equation*}
\omegacan_{f_*\sigma'} = f_*\omegacan_{\sigma'}.
\end{equation*}

\begin{ex}
Let $N\geq 1$ be an integer, and let $\mu_N$ denote the set of complex $N$-th roots of unity. Let $X=X'=\PP^1_\CC$, $Y=\{0,1\}$, $Y'=\{0\}\cup\mu_N$, and $f$ be  given by $w\mapsto z= w^N$. Let $\sigma'$ be (the relative homology class of) a path from $0$ to $\xi\in\mu_N$, for which $f_*\sigma'$ is (the relative homology class of) a path from $0$ to $1$. We check that:
$$f_*\omegacan_{\sigma'} = f_*\dlog\left(\frac{w-\xi}{w}\right) = \sum_{w^N=z}\dlog\left(\frac{w-\xi}{w}\right) = \dlog\left(\frac{z-1}{z}\right) = \omegacan_{f_*\sigma'}.$$
\end{ex}

\subsubsection{Functoriality for pullbacks}

Let $X'$ be another compact complex variety of the same dimension $n$, and $Y'\subset X'$ be a closed subvariety such that $X'\setminus Y'$ is smooth and $(X',Y')$ has genus zero. Let $g:X'\setminus Y'\to X\setminus Y$ be a covering map, i.e., a smooth morphism. 
Then we have a pullback map on the level of locally finite homology: 
$$g^*:\H_\bullet^{\mathrm{lf}}(X\setminus Y)\to \H_\bullet^{\mathrm{lf}}(X'\setminus Y'),$$
which by \eqref{eq:relative-homology-is-lf-homology} we view as a pullback map in relative homology:
$$g^*:\H_\bullet(X,Y)\to \H_\bullet(X',Y').$$
We have the following commutative diagram: 
$$\diagram{
\H_n(X',Y') \ar^-{\can}[r]  & \logforms{n}(X'\setminus Y') \\
\H_n(X,Y) \ar^-{\can}[r] \ar[u]^-{g^*} & \logforms{n}(X\setminus Y) \ar[u]_-{g^*}
}$$
In other words, for every $\sigma\in\H_n(X,Y)$ we have the equality
\begin{equation*}
\omegacan_{g^*\sigma} = g^*\omegacan_{\sigma}.
\end{equation*}

\begin{ex}
Let $N\geq 1$ be an integer, and let $\mu_N$ denote the set of complex $N$-th roots of unity. Let $X=X'=\PP^1_\CC$, $Y=\{0,\infty,1\}$, $Y'=\{0,\infty\}\cup\mu_N$, and $g$ be given by $w\mapsto z=w^N$. Let $\sigma$ be (the relative homology class of) a path from $0$ to $1$, for which $g^*\sigma$ is the sum of the (relative homology classes of) paths from $0$ to $\xi$, for $\xi\in\mu_N$. We check:
$$g^*\omegacan_{\sigma} = g^*\dlog\left(\frac{z-1}{z}\right) = \dlog\left(\frac{w^N-1}{w^N}\right) = \sum_{\xi\in \mu_N}\dlog\left(\frac{w-\xi}{w}\right) = \omegacan_{g^*\sigma},$$
where in the last equality we have used the linearity \eqref{eq:can-linearity}.
\end{ex}

\subsubsection{Multiplicativity}

For $i=1,2$, let $X_i$ be a compact complex variety of dimension $n_i$, and $Y_i\subset X_i$ be a closed subvariety such that $X_i\setminus Y_i$ is smooth.
We let
$$(X,Y)=(X_1\times X_2, (Y_1\times X_2)\cup (X_1\times Y_2)), $$
with projection maps $p_i:X\setminus Y\to X_i\setminus Y_i$. The cross product map in relative homology reads
$$\H_{n_1}(X_1,Y_1)\otimes \H_{n_2}(X_2,Y_2) \to \H_{n_1+n_2}(X,Y) \;\; , \;\; \sigma_1\otimes \sigma_2\mapsto \sigma_1\times \sigma_2.$$

\begin{prop}\label{prop:multiplicativity}
If $(X_1,Y_1)$ and $(X_2,Y_2)$ have genus zero, then so does $(X,Y)$, and we have the following equality, for $\sigma_i\in \H^{n_i}(X_i,Y_i)$,
    $$\omegacan_{\sigma_1\times\sigma_2} = (-1)^{n_1n_2} p_1^*\omegacan_{\sigma_1}\wedge p_2^*\omegacan_{\sigma_2} = p_2^*\omegacan_{\sigma_2}\wedge p_1^*\omegacan_{\sigma_1}.$$
\end{prop}

\begin{proof}
The first statement is Proposition \ref{prop:product-genus-zero} below. 
We consider the following diagram, whose rows are the Poincar\'{e} duality isomorphisms and whose columns are the cross product and cup product map respectively (recall that the latter is given on the level of differential forms by the formula $\omega_1\otimes\omega_2\mapsto p_1^*\omega_1\wedge p_2^*\omega_2$): 
$$\diagram{
\H_{n_1+n_2}(X, Y) \ar[r]^-{\sim} & \H^{n_1+n_2}(X\setminus Y)(n_1+n_2) \\
\H_{n_1}(X_1,Y_1)\otimes\H_{n_2}(X_2,Y_2) \ar[r]^-{\sim}\ar[u] & \H^{n_1}(X_1 \setminus Y_1)(n_1)\otimes \H^{n_2}(X_2 \setminus Y_2)(n_2) \ar[u]
}
$$
This square commutes  up to the sign $(-1)^{n_1n_2}$. Indeed, for $\sigma_i\in \H_{n_i}(X_i,Y_i)$ corresponding via Poincar\'{e} duality to classes of smooth forms $\varphi_i$ on $X_i\setminus Y_i$ of degree $n_i$, and for closed smooth forms with compact support $\eta_i$ on $X_i\setminus Y_i$, also of degree $n_i$, we have
\begin{align*}
    (-1)^{n_1n_2}\int_{X_1\times X_2}p_1^*\varphi_{\sigma_1}\wedge p_2^*\varphi_{\sigma_2}\wedge p_1^*\eta_1\wedge p_2^*\eta_2 & = \int_{X_1}\varphi_{\sigma_1}\wedge\eta_1 \cdot \int_{X_2}\varphi_{\sigma_2}\wedge \eta_2 \\
    & = (2\pi\mathrm{i})^{n_1} \int_{\sigma_1}\eta_1 \,\cdot\, (2\pi\mathrm{i})^{n_2}\int_{\sigma_2}\eta_2\\
    & = (2\pi\mathrm{i})^{n_1+n_2}  \int_{\sigma_1\times \sigma_2}p_1^*\eta_1\wedge p_2^*\eta_2.
\end{align*}
The claim follows.
\end{proof}

\begin{ex}
Consider $X_1=X_2=\PP^1_\CC$ with coordinates $(z_1,z_2)$, and $Y_1=Y_2=\{0,1\}$. For $i=1,2$, the canonical form of the interval $\sigma_i=[0,1]\subset X_i$ is 
$$\omegacan_{\sigma_i}=\mathrm{dlog}\left(\frac{z_i-1}{z_i}\right).$$
Since $n_1=n_2=1$, Proposition \ref{prop:multiplicativity} states that the canonical form of the square $\sigma_1\times\sigma_2=[0,1]^2$ is
$$\omegacan_{\sigma_1\times\sigma_2}=-\mathrm{dlog}\left(\frac{z_1-1}{z_1}\right)\wedge \mathrm{dlog}\left(\frac{z_2-1}{z_2}\right).$$
One way to check that the sign is correct is to compare it with the recursion formula from Proposition \ref{prop:recursion} below on all four sides of the square. For instance, taking the residue along the line $\{z_2=1\}$, we get
$$\Res_{\{z_2=1\}}(\omegacan_{\sigma_1\times\sigma_2}) = -\mathrm{dlog}\left(\frac{z_1-1}{z_1}\right).$$
This is indeed the canonical form of the partial boundary $\partial_{\{z_2=1\}}[0,1]^2$, which is the interval $[0,1]$ oriented from $1$ to $0$.
\end{ex}

\subsection{Canonical forms, boundaries, and residues}

\subsubsection{Recursion}

For a closed subvariety $Z\subset Y$, we have a \emph{boundary map}
$$\partial:\H_\bullet(X,Y)\to \H_{\bullet-1}(Y,Z)$$
which appears in the long exact sequence in relative homology \eqref{eq:long-exact-sequence-relative-homology} for the triple $(X,Y,Z)$. If $X\setminus Z$ and $Y\setminus Z$ are smooth, we also have a residue map
$$\Res:\H^\bullet(X\setminus Y)\to \H^{\bullet-1}(Y\setminus Z)(-1),$$
and applying the Hodge filtration functor $\F^\bullet$ produces a map
$$\Res:\logforms{\bullet}(X\setminus Y)\to \logforms{\bullet-1}(Y\setminus Z).$$
Note that we do indeed  require  $X\setminus Z$ to be smooth for this to make sense, even though we only need $X\setminus Y$ to be smooth for the canonical  map \eqref{eq:can-defi} to be defined.

\begin{prop}\label{prop:recursion}
Let $X$ be a compact complex variety of dimension $n\geq 1$, $Y\subset X$ be a closed subvariety of codimension $1$, and $Z\subset Y$ be a closed subvariety of codimension $\geq 1$, such that both $X\setminus Z$ and $Y\setminus Z$ are smooth.
\begin{enumerate}[1)]
\item Assume that $(Y,Z)$ has genus zero. Then $(X,Y)$ has genus zero if and only if $X$ does.
\item Assume that $(Y,Z)$ and $(X,Y)$ have genus zero. Then,
    \begin{enumerate}[a)]
    \item The following diagram commutes: 
    \begin{equation}\label{eq:can-recursion}
    \begin{gathered}
    \diagram{
    \H_n(X,Y) \ar[d]_-{\partial}\ar[r]^-{\can} & \logforms{n}(X\setminus Y) \ar[d]^{\Res} \\
    \H_{n-1}(Y,Z) \ar[r]^-{\can} & \logforms{n-1}(Y\setminus Z)
    }
    \end{gathered}
    \end{equation}
    In other words, for $\sigma\in \H_n(X,Y)$, we have the equality
    \begin{equation}\label{eq:recursion-formula}
    \mathrm{Res}(\omegacan_{\sigma}) = \omegacan_{\partial(\sigma)}.
    \end{equation}
    \item The morphism $\Res$ appearing in \eqref{eq:can-recursion} is injective if and only if $X$ has combinatorial rank zero (e.g., if $X$ is smooth). In this case, \eqref{eq:recursion-formula} completely characterizes $\omegacan_\sigma\in\logforms{n}(X\setminus Y)$.
    \end{enumerate}
    \end{enumerate}
\end{prop}

\begin{proof}
\begin{enumerate}[1)]
\item By Proposition \ref{prop:genus-inequalities-new} below, if $g(Y,Z)=0$ then $g(X,Y)=g(X,Z)$. Since $Z$ has codimension $\geq 2$ in $X$, applying Proposition \ref{prop:hartogs} below yields $g(X,Z)=g(X)$.
\item  
\begin{enumerate}[a)]
\item By embedded resolution of singularities, there exists a proper morphism $f:\widetilde{X}\to X$ with $\widetilde{X}$ smooth, such that both $f^{-1}(Y)$ and $f^{-1}(Z)$ are simple normal crossing divisors. Since $X\setminus Z$ and $Y\setminus Z$ are smooth, we can assume that $f$ is an isomorphism over $X\setminus Z$. We let $\widetilde{Y}$ denote the strict transform of $Y$ with respect to  $f$,  and $E:=f^{-1}(Z)$ denote the total inverse image of $Z$. We therefore have modifications $(\widetilde{X},\widetilde{Y}\cup E)\to (X,Y)$ and $(\widetilde{Y},\widetilde{Y}\cap E)\to (Y,Z)$. By \cite[Proposition 4.12]{BD1}, the following diagram commutes, where the rows are given by the Poincar\'{e} duality isomorphisms \eqref{eq:poincare-duality}:
\begin{equation*}
\begin{gathered}
\diagram{
\H_n(\widetilde{X},\widetilde{Y}\cup E) \ar[d]_-{\partial_{\widetilde{Y}}}\ar[r]^-{\sim}  & \H^n(\widetilde{X}\setminus \widetilde{Y}\cup E)(n) \ar[d]^{\Res_{\widetilde{Y}}} \\
\H_{n-1}(\widetilde{Y},\widetilde{Y}\cap E) \ar[r]^-{\sim}  & \H^{n-1}(\widetilde{Y}\setminus \widetilde{Y}\cap E)(n-1) \\
}
\end{gathered}
\end{equation*}
Note the disappearance of the $(-1)^{n-1}$ sign from \emph{loc.~cit.}, which is consistent with the sign convention for the residue that we are adopting  in the present paper (Remark \ref{rem:sign-convention-residue}). The claim then follows from the invariance of the  canonical morphism  by modifications \eqref{eq:can-and-modification}.
\item By the commutativity of \eqref{eq:can-recursion}, and the fact that 
$R:  \gr_0^\W\H_n(X,Y)_\CC   \simeq \logforms{n}(X\setminus Y)$ is an isomorphism, 
we see that $\Res$ is injective if and only if the $(0,0)$-component of the boundary morphism $\partial^{0,0}:\H_n(X,Y)^{0,0}\to \H_{n-1}(Y,Z)^{0,0}$ is injective. Thanks to \eqref{eq:long-exact-sequence-relative-homology}, this morphism fits into a long exact sequence
$$0 \To \H_n(X,Z)^{0,0}\To \H_n(X,Y)^{0,0}\stackrel{\partial^{0,0}}{\To} \H_{n-1}(Y,Z)^{0,0}\to \cdots,$$
where we have used the vanishing $\H_n(Y,Z)^{0,0}=0$, which follows from \eqref{eq:bound-on-hodge-numbers}. Therefore, $\Res$ is injective if and only if $(X,Z)$ has combinatorial rank zero. Since $Z$ has codimension $\geq 2$ in $X$,
this is equivalent to $X$ having combinatorial rank zero by Proposition \ref{prop:hartogs-rank} (below). The fact that the combinatorial rank of $X$ is zero if $X$ is smooth follows from \eqref{eq:bound-on-hodge-numbers-smooth}.
\end{enumerate}
\end{enumerate}
\end{proof}

\subsubsection{A special case}

A useful special case of Proposition \ref{prop:recursion} is as follows.

\begin{prop}\label{prop:recursion-irreducible-components}
Let $X$ be a compact complex variety of dimension $n\geq 1$, and $Y\subset X$ be a closed subvariety of codimension $1$, decomposed into irreducible components as $Y=Y_1\cup\cdots \cup Y_N$. Set $Z=Z_1\cup \cdots \cup Z_N$ with $Z_i=Y_i\cap\bigcup_{j\neq i}Y_j$, and assume that $X\setminus Z$ and all the $Y_i\setminus Z_i$ are smooth.
\begin{enumerate}[1)]
\item Assume that all the $(Y_i,Z_i)$ have genus zero. Then $(X,Y)$ has genus zero if and only if $X$ does.
\item Assume that all the $(Y_i,Z_i)$ and $(X,Y)$ have genus zero.
    \begin{enumerate}[a)]
    \item The following diagram commutes:
    \begin{equation}\label{eq:can-recursion-special-case}
    \begin{gathered}
    \diagram{
    \H_n(X,Y) \ar[d]_-{\bigoplus_i\partial_{Y_i}}\ar[r]^-{\can} & \logforms{n}(X\setminus Y) \ar[d]^{\bigoplus_i\Res_{Y_i}} \\
    \bigoplus_i \H_{n-1}(Y_i,Z_i) \ar[r]^-{\bigoplus_i\can} & \bigoplus_i \logforms{n-1}(Y_i\setminus Z_i)
    }
    \end{gathered}
    \end{equation}
    In other words, for $\sigma\in \H_n(X,Y)$ we have the equalities
    \begin{equation}\label{eq:recursion-formula-special-case}
    \Res_{Y_i}(\omegacan_\sigma)=\omegacan_{\partial_{Y_i}(\sigma)} \;\; \mbox{ for all } i=1,\ldots,N.
    \end{equation}
    \item The morphism $\bigoplus_i\Res_{Y_i}$ appearing in \eqref{eq:can-recursion-special-case} is injective if and only if $X$ has combinatorial rank zero (e.g., if $X$ is smooth). In this case, \eqref{eq:recursion-formula-special-case} completely characterizes $\omegacan_\sigma\in\logforms{n}(X\setminus Y)$.
    \end{enumerate}
    \end{enumerate}
\end{prop}
\begin{proof}
This follows from Proposition \ref{prop:recursion} because $(\bigsqcup_i Y_i, \bigsqcup_i Z_i) \to (Y,Z)$ is a modification.
\end{proof}

\subsubsection{Having genus zero is not a recursive condition}\label{sec:genus-zero-not-recursive}

In the setting of Proposition \ref{prop:recursion} or Proposition \ref{prop:recursion-irreducible-components}, the fact that $(X,Y)$ has genus zero does not imply that the same holds for $(Y,Z)$,  nor for the $(Y_i,Z_i)$. This is in contrast with the classical definition of a positive geometry \cite{arkanihamedbailampositive} according to which every face of a positive geometry is required to be a positive geometry. 

The simplest example of this  phenomenon is
$$(X,Y)=(\PP^1_\CC\times C,\{0\}\times C)$$
for $C$ a connected smooth projective curve of genus $g>0$. It has genus zero because by the K\"{u}nneth formula,
$$\H_2(X,Y) \simeq \H_2(\PP^1_\CC,\{0\}) \simeq \QQ(1)$$
 only has a single non-zero Hodge number $h^{-1,-1}=1$. However, its face $(Y,\varnothing)=(C,\varnothing)$ has genus $g>0$. Note that this example is uninteresting from the point of view of canonical forms because it has combinatorial rank $0$. In \S\ref{sec:example-non-recursive-positive-geometry} we provide an example of a similar non-recursive phenomenon, but which has  positive combinatorial rank.

\subsection{The normal crossing case: corners}

Let $X$ be a smooth  and compact complex variety of dimension $n$ and $D\subset X$ be a (not necessarily simple) normal crossing divisor. 

\begin{defi}
A \emph{corner} of $D$ is a point of $X$ which is the intersection of $n$ (necessarily distinct) local irreducible components of $D$.
\end{defi}

For a corner $c$ of $D$ we have an iterated boundary map
\begin{equation}\label{eq:corner-boundary}
\partial_c:\H_n(X,D)\to \H_0(\{c\})\simeq \QQ,
\end{equation}
which is well-defined up to a sign. We will not need to make a choice of sign.  If $D$ is a simple normal crossing divisor written as a union $D_1\cup\cdots \cup D_N$ of smooth irreducible components, then $c$ is a point of an intersection $D_{i_1}\cap\cdots \cap D_{i_n}$ for some indices $1\leq i_1<\cdots <i_n\leq N$, and $\partial_c$ is defined as the composition  $\partial_{D_{i_1}\cap\cdots \cap D_{i_n}} \cdots  \partial_{D_{i_{n-1}} \cap D_{i_n}}\circ \partial_{D_{i_n}}$ of partial boundary maps. Changing the order on the set of irreducible components of $D$ only changes $\partial_c$ by a sign since we have, for instance, $\partial_{D_i\cap D_j}\partial_{D_i}=-\partial_{D_i\cap D_j}\partial_{D_j}$. In the general case when $D$ is not necessarily simple, $\partial_c$ is computed by applying the same recipe in a local chart around $c$.

In a similar way,  we may consider the iterated residue map
\begin{equation}\label{eq:corner-residue}
\Res_c:\logforms{n}(X\setminus D)\to \Omega^0(\{c\})\simeq \CC,
\end{equation}
which is also well-defined up to a sign (since, for instance, $\Res_{D_i\cap D_j}\Res_{D_i}=-\Res_{D_i\cap D_j}\Res_{D_j}$ in the simple normal crossing case). Precise sign conventions will be irrelevant in what follows, as long as one uses consistent conventions for both $\partial_c$ and $\Res_c$.

\begin{prop}\label{prop:corners}
The iterated boundary maps along corners \eqref{eq:corner-boundary} induce an injective map
$$\gr_0^\W\H_n(X,D) \hookrightarrow \bigoplus_c \QQ.$$
where the  direct sum is  over the set of  all corners $c$ of $D$. The composition 
$$
\logforms{n}(X\setminus D) \stackrel{R}{\To} \gr_0^\W\H_n(X,D)_\CC \hookrightarrow \bigoplus_c \CC,
$$
where the $R$-map is \eqref{eq:R-map}, is given by the iterated residue maps along corners \eqref{eq:corner-residue}.
\end{prop}

In \S\ref{sec:combinatorial-rank} we will be more precise about the space $\gr_0^\W\H_n(X,D)$ and how to compute it.

\begin{proof}
In the case of a simple normal crossing divisor, this is \cite[Propositions 4.7, 4.8, 4.13]{BD1}, without the signs because of our sign convention (Remark \ref{rem:sign-convention-residue}). In the general case, repeatedly blowing up strata of $D$ produces a modification $(X',D')\to (X,D)$ such that $D'\subset X'$ is a simple normal crossing divisor and the morphism $X'\to X$ maps a corner of $D'$ to a corner of $D$. Let $c'$ be a corner of $D'$ mapping to a corner $c$ of $D$. As one easily checks in a local chart, the iterated boundary maps along corners \eqref{eq:corner-boundary} are ``the same'' for $c'$ and $c$, and likewise for the iterated residue maps along corners \eqref{eq:corner-residue}. More precisely, the following diagrams commute up to sign.
$$\diagram{
\H_n(X',D')\ar[r]^-{\partial_{c'}} \ar[d]_-{\sim}& \QQ\ar@{=}[d] \\
\H_n(X,D)\ar[r]_-{\partial_c} & \QQ
}
\qquad\qquad\qquad 
\diagram{
\logforms{n}(X',D')\ar[r]^-{\Res_{c'}} & \CC\ar@{=}[d] \\
\logforms{n}(X,D)\ar[r]_-{\Res_c}\ar[u]^-{\sim} & \CC
}
$$
Therefore, we have commutative diagrams as follows, where the horizontal maps are the iterated boundary maps along corners \eqref{eq:corner-boundary} on the left, and the iterated residue maps along corners \eqref{eq:corner-residue} on the right.
$$\diagram{
\gr_0^\W\H_n(X',D')\ar[r] \ar[d]_-{\sim}& \displaystyle\bigoplus_{c'}\QQ \\
\gr_0^\W\H_n(X,D)\ar[r] & \displaystyle\bigoplus_c \QQ \ar[u]
}
\qquad\qquad\qquad 
\diagram{
\logforms{n}(X',D')\ar[r] & \displaystyle\bigoplus_{c'}\CC\\
\logforms{n}(X,D)\ar[r]\ar[u]^-{\sim} & \displaystyle\bigoplus_c\CC \ar[u]
}$$
The claim in the general case then follows from the simple normal crossing case.
\end{proof}

\begin{coro}\label{coro:corners}
Assume that $(X,D)$ has genus zero. For a class $\sigma\in \H_n(X,D)$, the canonical form $\omegacan_\sigma$ is the unique element of $\,\logforms{n}(X\setminus D)$ such that for every corner $c$ of $D$ we have the equality
$$\Res_c(\omegacan_\sigma)=\partial_c(\sigma).$$
\end{coro}

This corollary allows one to compute canonical forms in cases where the recursive approach fails (\S\ref{sec:genus-zero-not-recursive}).

\subsection{Adjoint hypersurfaces}\label{sec:adjoint}

Let $X$ be a compact complex variety of dimension $n$, and $Y\subset X$ be a closed subvariety such that $X\setminus Y$ is smooth. Assume that $(X,Y)$ has genus zero, and let $\sigma\in \H_n(X,Y)$. If the canonical form $\omegacan_\sigma$ is non-zero, its vanishing locus $Z(\omegacan_\sigma)$ is a hypersurface in $X$ which goes under the name \emph{adjoint hypersurface} in the literature. If $X$ is smooth and $Y=D$ is a normal crossing divisor, then $\omegacan_\sigma$ is simply a section of the line bundle $\Omega^n_X(\log D)$. This can sometimes help to describe the adjoint hypersurface, as in the next proposition,  some versions of which already appeared in different guises \cite{arkanihamedbailampositive,polypol,laminvitation}.

\begin{prop}\label{prop:adjoint-hypersurface-degree}
Let $Y\subset \PP^n_\CC$ be a projective hypersurface of degree $d\geq n+1$. Assume that $(\PP^n_\CC, Y)$ has genus zero. Then for every class $\sigma\in \H_n(\PP^n_\CC, Y)$ such that the canonical form $\omegacan_\sigma$ is non-zero, the adjoint hypersurface $Z(\omegacan_\sigma)\subset \PP^n_\CC$ is a projective hypersurface of degree $d-n-1$.
\end{prop}

\begin{proof}
This follows from Proposition \ref{prop:log-forms-have-simple-poles} and the description of forms with at most simple poles on projective space (Remark \ref{rem:log-forms-on-projective-space}).
\end{proof}

\begin{rem}\label{rem:simplex-like}
In the case $d=n+1$, we get $Z(\omegacan_\sigma)=\varnothing$. Such a situation is referred to in the literature as a ``simplex-like positive geometry'' \cite{arkanihamedbailampositive} because a first example is given by  a simplex in projective space, with $D$  a union of $n+1$ independent hyperplanes.
\end{rem}

\subsection{Beyond genus zero}\label{subsec:beyond-genus-zero}

Let $X$ be  a compact complex variety of dimension $n$, and $Y\subset X$ be a closed subvariety such that $X\setminus Y$ is smooth. If the genus $g$ of $(X,Y)$ is not assumed to be $0$, then the $R$-map \eqref{eq:R-map} has a kernel of dimension $g$, and one can adapt Definition \ref{defi:can-compact} to get a morphism
$$\can\colon \H_n(X,Y)\to \logforms{n}(X\setminus Y) / \ker(R).$$
In other words, canonical forms are only well-defined over a $g$-dimensional space of ambiguities.

\begin{rem}\label{rem:ker-of-R}
Assume that $X$ is smooth and that $Y=D$ is a normal crossing divisor. Then Proposition \ref{prop:corners} implies that the $g$-dimensional space $\ker(R)$ consists of those elements of $\logforms{n}(X\setminus D)$ whose corner residues all vanish.
If $n\geq 1$, it contains the space of global holomorphic forms on $X$, which account for the Hodge number $h^{-n,0}$ of $\H_n(X,D)$. However, it contains more forms in general because some Hodge number $h^{-p,0}$ may be $\neq 0$ for $0<p<n$ (e.g., $X=\PP^2_\CC$ and $Y=$ a smooth curve of genus $g>0$).
\end{rem}

In certain situations, extra structure may lead to a preferred splitting of $R$, i.e., of the projection $\logforms{n}(X\setminus Y) \twoheadrightarrow \logforms{n}(X\setminus Y)/\ker(R)$, and therefore to preferred choices of canonical forms. To fix a unique choice of ``canonical'' form, one imposes as many linear conditions as the dimension of $\mathrm{ker}(R)$, which equals $g$.
In the remainder of this section, we discuss the case of curves.

\subsubsection{The case of curves}
Let $X$ be a smooth compact complex curve (also known as a compact Riemann surface) of genus $g$ and let $Y=\{a,b\}$ consist of two distinct points of $X$. We seek  a ``canonical form'' $\omega=\omegacan_{\gamma_{a,b}}$ for (the relative homology class of) a path $\gamma_{a,b}$ from $a$ to $b$ in $X$. The residue long exact sequence reads:
\begin{equation}\label{eq:residue-long-exact-sequence-curves}
0 \To \H^1(X)\To \H^1(X\setminus \{a,b\}) \stackrel{\mathrm{Res}}{\To} \H^0(\{a\})(-1) \oplus \H^0(\{b\})(-1) \To \H^2(X)\To 0.
\end{equation}
It follows that the non-vanishing Hodge numbers of $\H^1(X\backslash \{a,b\})$ are $h^{1,0}=h^{0,1}=g$ and $h^{1,1}=1$. Using Poincar\'{e} duality \eqref{eq:poincare-duality}, one sees that the non-vanishing Hodge numbers of $\H_1(X,\{a,b\})$ are $h^{-1,0}=h^{0,-1}=g$ and $h^{0,0}=1$. In particular, the genus of $(X,\{a,b\})$ is $g$. By applying the (exact) Hodge filtration functor $\F^1$ to \eqref{eq:residue-long-exact-sequence-curves} one gets the exact sequence
\[   0 \To \H^0(X,\Omega^1_X) \To \logforms{1}(X\setminus \{a,b\})  \stackrel{\Res}{\To} \CC\oplus \CC\To \CC \To 0,\]
where the last map is the sum $\CC\oplus \CC\to \CC$. It states that given two complex numbers $r_a,r_b$ which sum to $0$, there exists a form $\omega\in \logforms{1}(X\setminus\{a,b\})$ such that $\Res_a(\omega)=r_a$ and $\Res_b(\omega)=r_b$ (such a form is  called  a \emph{differential of the third kind}). It is only well-defined modulo a global holomorphic $1$-form on $X$ (called  a \emph{differential of the first kind}). The space $\H^0(X,\Omega^1_X)$ of global holomorphic $1$-forms on $X$ has dimension $g$.

The literature provides a natural choice of differential of the third kind: there exists a unique $\omega\in\logforms{1}(X\setminus \{a,b\}$ with prescribed residues and whose periods are imaginary, i.e., such that
\[ \mathrm{Re} \, \int_{\delta} \omega= 0  \]
for $\delta$ any closed path in $X\setminus \{a,b\}$. This choice is related to the existence of Green's functions, also known as archimedean height pairings,  on curves. See \cite[\S 6.4]{BD1} for a discussion and an interpretation in purely algebraic terms using a ``determinant trick''.

\subsubsection{The case of elliptic curves} Let $X=E$ be an elliptic curve, and assume that $a,b\in X$ are distinct from the origin $0\in E$. Then we may require the differential of the third kind $\omega$ to vanish at $0$. Since a holomorphic differential on $E$ (i.e., a generator of $\H^0(E,\Omega^1_E)$) is nowhere vanishing, this condition uniquely determines $\omega$.  Furthermore, this choice of $\omega$ is functorial with respect to morphisms between elliptic curves. 

Explicitly, if $E$ is given in Weierstrass form by the affine equation $ y^2 =x^3+ux +v$, then 
\[ \omega = \frac{1}{2} \left( \frac{y+y(b)}{x-x(b)} -  \frac{y+y(a)}{x-x(a)}  \right) \frac{\dd x}{y}\ , \]
has residues $-1$ at $a$ and $1$ at $b$, and vanishes at the origin (the point at infinity $(0:0:1)$).

\section{The genus of a pair of varieties}\label{sec:genus}

We study the  genus of a pair of complex varieties, which generalizes the classical genus of a curve, and the geometric genus of a smooth compact variety. For the sake of convenience, we switch from homology to cohomology and restate Definition \ref{defi:genus} as follows. The terminology introduced here is non-standard.

\begin{defi} \label{defi:genus-with-cohomology}
Let $X$ be a complex variety of dimension $n$, and $Y\subset X$ be a closed suvariety. The \emph{genus} of the pair $(X,Y)$, denoted by $g(X,Y)$, is the sum of the Hodge numbers $h^{p,0}$ of $\H^n(X,Y)$, for $p>0$:
$$g(X,Y)=\sum_{p>0}h^{p,0}(\H^n(X,Y)).$$
\end{defi}

\subsection{The genus of a variety}

We first focus on the case of a single variety, which already presents some interesting subtleties in the singular case. We let $g(X):=g(X,\varnothing)$.

\subsubsection{The smooth case}

In the \emph{smooth and compact} case (which  is the standard context for defining the genus), 
the genus as defined in Definition \ref{defi:genus-with-cohomology} coincides with what is
classically called the \emph{geometric genus}, i.e., the dimension of the space of global holomorphic forms of maximal degree:

\begin{prop}
Let $X$ be a smooth and compact complex variety of dimension $n\geq 1$. Then 
$$g(X) = \dim\H^0(X,\Omega_X^n).$$
\end{prop}

\begin{proof}
Since $\H^n(X)$ is a pure Hodge structure of weight $n$ we have  $g(X)=\dim\H^{n,0}(X)$, which equals the dimension of the space $\H^0(X,\Omega^n_X)$ by  \eqref{eq:hodge-p-q-as-sheaf-cohomology}.
\end{proof}

The general case of smooth varieties easily reduces to the smooth compact case, as the next result shows (it also proves the classical fact that the geometric genus is a birational invariant).

\begin{prop}\label{prop:genus-smooth-compactification}
Let $X$ be a smooth complex variety and $X\subset \overline{X}$ be a smooth compactification. Then we have the equality
$$g(X)=g(\overline{X}).$$
\end{prop}

\begin{proof}
The natural map $\H^n(\overline{X})\to \H^n(X)$ is a morphism of mixed Hodge structures, and it is enough to prove that it induces an isomorphism on the $(p,0)$ components for all $p>0$. Let $Z=\overline{X}\setminus X$. By \eqref{eq:poincare-duality} and classical Poincar\'{e} duality for $\overline{X}$ this amounts to proving that the natural map $\H^n(\overline{X},Z)\to \H^n(\overline{X})$ induces an isomorphism on the $(n-p,n)$ components for all $p>0$. Consider the long exact sequence in relative cohomology
$$\cdots \To \H^{n-1}(Z) \To \H^n(\overline{X},Z)\To \H^n(\overline{X}) \To \H^n(Z)\To \cdots \ . $$
Since $Z$ has  dimension $<n$, by \eqref{eq:bound-on-hodge-numbers} we get that
$\H^r(Z)^{n-p,n}=0$ for all $r$ and $p$, hence the claim.
\end{proof}

\begin{ex}\label{ex:genus-smooth-rational}
The genus of projective space is
\begin{equation}\label{eq:genus-projective-space}
g(\PP^n_\CC)=0
\end{equation}
and therefore every smooth rational variety has genus zero. This includes open subschemes of projective spaces or Grassmannians, and blow-ups of such varieties along smooth subvarieties.
\end{ex}

\subsubsection{The singular case}

For singular \emph{curves}, our notion of genus agrees with the standard one:

\begin{prop} \label{prop: genusofcurveandresolution}
Let $C$ be a complex curve, and let $\widetilde{C}$ be its canonical resolution of singularities. Then we have the equality
$$g(C)=g(\widetilde{C}).$$
\end{prop}

\begin{proof}
The resolution $\pi\colon\widetilde{C}\to C$ is an isomorphism over a Zariski open $C_0$ of $C$, and the loci $S=C\setminus C_0$ and $\widetilde{S}=\pi^{-1}(S)$ are finite sets of points. By \eqref{eq:iso-relative-cohomology-excision}, the following natural morphism of mixed Hodge structures induced by $\pi$ is an isomorphism: 
$$\H^1(C,S) \stackrel{\sim}{\To} \H^1(\widetilde{C},\widetilde{S}). $$
 The relative cohomology  long exact sequences  of the pairs $(C,S)$ and $(\widetilde{C}, \widetilde{S})$ imply that the natural morphisms of mixed Hodge structures $\H^1(C,S)\to \H^1(C)$ and $\H^1(\widetilde{C},\widetilde{S})\to \H^1(\widetilde{C})$ are isomorphisms on the $(1,0)$ components, and the claim follows.
\end{proof}

In higher dimension, the genus of a singular variety may differ from the genus of a resolution of singularities, as illustrated in the next example borrowed from \cite[Example 4.2]{chataurcirici}.

\begin{ex}
Let $d\geq 2$, let $\varphi$ and $\psi$ be non-zero homogeneous polynomials in  three variables $x,y,z$ of respective degrees $d$ and $d+1$, and let $C=\{\varphi=0\}$ and $D=\{\psi=0\}$ denote the corresponding curves in $\mathbb{P}^2_\CC$. We assume that $C$ and $D$ are smooth and intersect transversely in  $d(d+1)$ points. Let $\widetilde{X}$ denote the blow-up of $\mathbb{P}^2_\CC$ along $C\cap D$, and  $\widetilde{C}\subset \widetilde{X}$ the strict transform of $C$. Since $\widetilde{X}$ is a smooth rational surface we have $g(\widetilde{X})=0$ by Example \ref{ex:genus-smooth-rational}, and more precisely,
\begin{equation}\label{eq:H2-of-Xtilde-in-example}
\H^2(\widetilde{X})\simeq \QQ(-1)\oplus \QQ(-1)^{d(d+1)}.
\end{equation}
Denote coordinates on $\mathbb{P}^3_\CC$ by  $(x:y:z:w)$ and let  $X
\subset\mathbb{P}^3_\CC$  be the surface defined by the equation $w\,\varphi(x,y,z)=\psi(x,y,z)$. Its only singular point is $o=(0:0:0:1)$. Consider the rational map
$$\mathbb{P}^2_\CC \dashrightarrow X\; , \; (x:y:z) \mapsto (x\varphi(x,y,z):y\varphi(x,y,z):z\varphi(x,y,z):\psi(x,y,z)),$$
which sends a line through $o$ in $\PP^3_\CC$ to its other intersection point with $X$. It induces a regular morphism  
$$\pi\colon \widetilde{X}\to X,$$
which is an isomorphism over $X\setminus\{o\}$, with $\pi^{-1}(\{o\})=\widetilde{C}$. By \eqref{eq:iso-relative-cohomology-excision}, the natural morphism of mixed Hodge structures induced by $\pi$:
$$\H^2(X)\simeq \H^2(X,\{o\})\To \H^2(\widetilde{X},\widetilde{C}),$$
is an isomorphism. Consider the long exact sequence in relative cohomology for the pair $(\widetilde{X},\widetilde{C})$:
$$\cdots \To \H^1(\widetilde{X})\To \H^1(\widetilde{C}) \To \H^2(\widetilde{X},\widetilde{C}) \To \H^2(\widetilde{X})\To \H^2(\widetilde{C})\To \cdots .$$
We have $\H^1(\widetilde{X})\simeq \H^1(\mathbb{P}^2_\CC)=0$ and the map $\H^2(\widetilde{X})\to \H^2(\widetilde{C})\simeq \QQ(-1)$ is easily seen to be surjective, with kernel $\QQ(-1)^{d(d+1)}$ by \eqref{eq:H2-of-Xtilde-in-example}. Since $\widetilde{C}\simeq C$, we obtain the short exact sequence
$$0\To \H^1(C)\To \H^2(X)\To \QQ(-1)^{d(d+1)}\To 0,$$
which shows that the only non-zero Hodge numbers of $\H^2(X)$ are $h^{1,0}=h^{0,1}=g(C)=\frac{(d-1)(d-2)}{2}$, and $h^{1,1}=d(d+1)$. Therefore the genus of $X$ is $g(X)=g(C)$, and if $d\geq 3$,
$$g(X) \neq   g(\widetilde{X}) =0 $$
which implies that the genus of $X$ is different from the genus of its resolution of singularities $\widetilde{X}$. 
\end{ex}

\subsection{Properties of the genus} 

\subsubsection{Invariance under modification} 

If $f:(X',Y')\to (X,Y)$ is a modification then we have
\begin{equation}\label{eq:genus-modification}
g(X,Y)=g(X',Y')
\end{equation}
because of the isomorphism of mixed Hodge structures \eqref{eq:iso-relative-cohomology-excision} for $\bullet=\dim(X)$.

\subsubsection{Products}

\begin{prop}\label{prop:product-genus-zero}
If $(X,Y)$ and $(X',Y')$ have genus zero, then so does $(X\times X',(Y\times X'\cup X\times Y'))$.
\end{prop}

\begin{proof}
Let $n=\dim(X)$ and $n'=\dim(X')$. We have the K\"{u}nneth formula for relative cohomology:
$$\H^{n+n'}(X\times X', Y\times X'\cup X\times Y') \simeq \bigoplus_{i+i'=n+n'}\H^i(X,Y)\otimes \H^{i'}(X',Y').$$
By \eqref{eq:bound-on-hodge-numbers} the terms with $i>n$ or $i'>n'$ have vanishing Hodge numbers $h^{p,0}$ for all $p$, and by assumption the term with $(i,i')=(n,n')$ has vanishing $h^{p,0}$ for all $p>0$. The claim follows.
\end{proof}

\begin{rem}
The proof of Proposition \ref{prop:product-genus-zero} shows that if $(X,Y)$ and $(X',Y')$ have combinatorial rank zero, then $g(X\times X',(Y\times X'\cup X\times Y'))=g(X,Y)g(X',Y')$.  Alternatively, this follows from the fact that $g(X,Y)= \dim \F^0\H_n(X,Y)$ and  $g(X',Y')= \dim \F^0\H_n(X',Y')$, which follows from \eqref{eq: g-plus-rk-equals-dimF0}.
\end{rem}

\subsubsection{Hartogs phenomenon}

The genus is insensitive to subvarieties of codimension $\geq 2$, as follows.

\begin{prop}\label{prop:hartogs}
Let $X$ be a complex variety, and $Y,Z\subset X$ be closed subvarieties. If $Z$ has codimension $\geq 2$, then 
$$g(X,Y\cup Z)=g(X,Y).$$
\end{prop}

\begin{proof}
Let $n=\dim(X)$. We have the long exact sequence in relative cohomology \eqref{eq:long-exact-sequence-relative-cohomology}:
$$\cdots \to \H^{n-1}(Z,Y\cap Z) \to \H^n(X,Y\cup Z)\to \H^n(X,Y)\to \H^n(Z,Y\cap Z)\to \cdots.$$
We have used the isomorphism $\H^\bullet(Y\cup Z, Z)\simeq \H^\bullet(Z,Y\cap Z)$, which follows from \eqref{eq:iso-relative-cohomology-excision} applied to the modification $(Z,Y\cap Z)\hookrightarrow (Y\cup Z,Y)$. Since $\dim(Z)\leq n-2$, we have by \eqref{eq:bound-on-hodge-numbers} that $\H^{n-1}(Z,Y\cap Z)$ and $\H^n(Z,Y\cap Z)$ have vanishing Hodge numbers $h^{p,0}$ for all $p>0$, and the claim follows.
\end{proof}

\subsubsection{Recursion}

One can estimate and sometimes compute the genus in a recursive manner, as follows.

\begin{prop}\label{prop:genus-inequalities-new}
Let $X$ be a complex variety of dimension $n$ and $Z\subset Y\subset X$ be closed subvarieties, with $Y$ of codimension $1$. We have the inequalities:
$$g(X,Z)\leq g(X,Y)\leq g(X,Z)+g(Y,Z).$$
If $\H^{n-1}(X,Z)^{p,0}=0$ for all $p>0$ (e.g., if $X$ is compact and $X\setminus Z$ is smooth and affine), then
$$g(X,Y)=g(X,Z)+g(Y,Z).$$
\end{prop}

Note that by Proposition \ref{prop:hartogs}, if $Z$ has codimension $\geq 2$ in $X$ then $g(X,Z)=g(X)$. We state the case $Z=\varnothing$ for future reference.

\begin{coro}\label{coro:genus-inequalities-simple}
Let $X$ be a complex variety of dimension $n$ and $Y\subset X$ be a closed subvariety of codimension $1$. We have the inequalities:
\begin{equation}\label{eq:genus-two-inequalities}
g(X)\leq g(X,Y)\leq g(X)+g(Y).
\end{equation}
If $\H^{n-1}(X)^{p,0}=0$ for all $p>0$, then
$$g(X,Y)=g(X)+g(Y).$$
\end{coro}

Note that we can have $g(X,Y)<g(X)+g(Y)$, see \S\ref{sec:genus-zero-not-recursive}.

\begin{proof}[Proof of Proposition \ref{prop:genus-inequalities-new}]
Consider the long exact sequence in relative cohomology \eqref{eq:long-exact-sequence-relative-cohomology}:
$$\cdots\to \H^{n-1}(X,Z)\to \H^{n-1}(Y,Z)\to \H^n(X,Y)\to \H^n(X,Z)\to \H^n(Y,Z)\to \cdots.$$
For $p>0$, by \eqref{eq:bound-on-hodge-numbers} we get that $\H^n(Y,Z)^{p,0}=0$ because $Y$ has dimension $n-1$, and therefore we get a long exact sequence
$$\cdots\to \H^{n-1}(X,Z)^{p,0}\to \H^{n-1}(Y,Z)^{p,0}\to \H^n(X,Y)^{p,0}\to \H^n(X,Z)^{p,0}\to 0.$$
The result follows. If $X$ is compact and $X\setminus Z$ is smooth, then Poincar\'{e} duality \eqref{eq:poincare-duality} gives an isomorphism between $\H^{n-1}(X,Z)$ and $\H^{n+1}(X\setminus Z)^\vee(n)$, which vanishes if $X\setminus Z$ is affine (``Artin vanishing'').
\end{proof}

We also have the following useful special cases of Proposition \ref{prop:genus-inequalities-new}.

\begin{coro}\label{coro:genus-inequalities-two-hypersurfaces}
Let $X$ be a complex variety of dimension $n$, and $Y,Y'\subset X$ be two closed subvarieties of codimension $1$. We have the inequalities:
\begin{equation}\label{eq:genus-inequality-intersection}
g(X,Y')\leq g(X,Y\cup Y')\leq g(X,Y')+g(Y,Y\cap Y').
\end{equation}
If $\H^{n-1}(X,Y')^{p,0}=0$ for all $p>0$ (e.g., if $X$ is compact and $X\setminus Y'$ is smooth and affine), then
$$g(X,Y\cup Y')=g(X,Y')+g(Y,Y\cap Y').$$
\end{coro}

\begin{proof}
By applying \eqref{eq:genus-modification} to the modification $(Y,Y\cap Y')\hookrightarrow (Y\cup Y',Y')$ we get $g(Y\cup Y',Y')=g(Y, Y\cap Y')$, and the result follows from Proposition \ref{prop:genus-inequalities-new}.
\end{proof}

\begin{coro}\label{coro:genus-inequalities-special-case}
Let $X$ be a complex variety of dimension $n$, and $Y\subset X$ be a closed subvariety of codimension $1$, decomposed into irreducible components as $Y=Y_1\cup\cdots \cup Y_N$. Set $Z=Z_1\cup\cdots \cup Z_N$ with $Z_i=Y_i\cap \bigcup_{j\neq i}Y_j$. We have the inequalities:
$$g(X) \leq g(X,Y) \leq g(X) + \sum_{i=1}^N g(Y_i,Z_i).$$
If $\H^{n-1}(X,Z)^{p,0}=0$ for all $p>0$ then
$$g(X,Y)=g(X)+\sum_{i=1}^N g(Y_i,Z_i).$$
\end{coro}

\begin{proof}
Note that $Z$ has codimension $\geq 2$ in $X$, hence $g(X,Z)=g(X)$ by Proposition \ref{prop:hartogs}. By applying \eqref{eq:genus-modification} to the modification $(\bigsqcup_i Y_i,\bigsqcup Z_i)\to (Y,Z)$, we get $g(Y,Z)=\sum_i g(Y_i,Z_i)$, and the result follows from Proposition \ref{prop:genus-inequalities-new}.   
\end{proof}

One can use Corollary \ref{coro:genus-inequalities-special-case} inductively to estimate the genus, as in the next result.

\begin{coro}
Let $X$ be a complex variety, and $D\subset X$ be a simple normal crossing divisor, written as a union $D_1\cup\cdots \cup D_N$ of smooth irreducible components. Then we have the inequality:
$$g(X,D)\leq \sum_{I\subset \{1,\ldots,N\}}g(D_I).$$
\end{coro}

\subsubsection{Genus and compactification}

A subtelty of the notion of the genus of a pair is that in general it can change after  passing to a compactification, contrary to the case of single variety (Proposition \ref{prop:genus-smooth-compactification}), see Example \ref{ex:triple-point} below. One is naturally led to the following constraints.

    \begin{defi}
    Let $X$ be a complex variety, and let $Y,Z\subset X$ be two closed subvarieties. The triple $(X,Y,Z)$ is \emph{locally of product type} if it is locally (in the analytic topology) isomorphic to a triple $(U\times V,Y_U\times V,U\times Z_V)$ for pairs $(U,Y_U)$ and $(V,Z_V)$ of complex analytic varieties. 
    \end{defi}

    The prototypical example of such a situation is a triple $(X,A,B)$ where $X$ is a smooth complex variety $X$ and $A,B$ are two hypersurfaces in $X$ that do not share an irreducible component and such that $A\cup B$ is a normal crossing divisor in $X$.

    \begin{defi}\label{defi:good-compactification}
    Let $(X,Y)$ be a pair of complex varieties such that $X\setminus Y$ is smooth. A \emph{good compactification} of $(X,Y)$ is a pair $(\overline{X},\overline{Y})$ where $\overline{X}$ is a compactification of $X$ and $\overline{Y}$ is the Zariski closure of $Y$ in $\overline{X}$, such that $\overline{X}\setminus \overline{Y}$ is smooth and such that if $Z$ denotes the complement $\overline{X}\setminus X$, the triple $(\overline{X},\overline{Y},Z)$ is locally of product type.
    \end{defi}

    Note that a good compactification of $(X,\varnothing)$ is simply a pair $(\overline{X},\varnothing)$ with $\overline{X}$ a smooth compactification of $X$. If $X$ is compact, then $X$ itself is a good compactification of $(X,Y)$. 

    \begin{ex}\label{ex:not-a-good-compactification}
    Take $X=\CC^2$ with coordinates $x,y$, and $Y$ the union of the two parallel lines $L=\{x=0\}$ and $L'=\{x=1\}$. Then $\overline{X}=\PP^2_\CC$ does not give rise to a good compactification because near the point at infinity $(1:0:0)$ where $\overline{L}$ and $\overline{L}'$ meet the line at infinity $Z$, the triple $(\overline{X},\overline{Y}, Z)$ is not of product type. An example of a good compactification in this case is $\PP^1_\CC\times \PP^1_\CC$.
    \end{ex}
    
    We do not know whether every pair $(X,Y)$ admits a good compactification. However, one can easily prove, using embedded resolution of singularities, that every pair $(X,Y)$ has a modification $(X',Y')$ which admits a good compactification, which is enough for most purposes in view of \eqref{eq:iso-relative-cohomology-excision}.

    \begin{prop}\label{prop:poincare-duality-good-compactification-appendix}
    Let $(X,Y)$ be a pair of complex varieties such that $X\setminus Y$ is smooth. Let $(\overline{X},\overline{Y})$ be a good compactification of $(X,Y)$, and let $Z=\overline{X}\setminus X$. For all $k$ we have an isomorphism of mixed Hodge structures:
    \begin{equation}\label{eq:poincare-duality-good-compactification}
    \H_k(X,Y)=\H_k(\overline{X}\setminus Z, \overline{Y}\setminus \overline{Y}\cap Z) \simeq \H^{2n-k}(\overline{X}\setminus \overline{Y}, Z\setminus \overline{Y}\cap Z)(n).
    \end{equation}
    \end{prop}

    \begin{proof}
    This follows from the same argument as in the proof of \cite[Theorem 3.1]{BD1} (see Remark 3.1 of \emph{op.~cit.}). 
    \end{proof}

\begin{prop}\label{prop:genus-good-compactification}
Let $(X,Y)$ be a pair of complex varieties such that $X\setminus Y$ is smooth, and let $(\overline{X},\overline{Y})$ be a good compactification of $(X,Y)$. Then we have the equality
$$g(X,Y)=g(\overline{X},\overline{Y}).$$
\end{prop}

\begin{proof}
It is enough to prove that the natural map $\H^n(\overline{X},\overline{Y})\to \H^n(X,Y)$ induces an isomorphism after applying  $\gr^0_\F$. Let $Z=\overline{X}\setminus X$. 
By Proposition \ref{prop:poincare-duality-good-compactification-appendix} we have the Poincar\'{e} duality isomorphism
$$\H^n(X,Y)=\H^n(\overline{X}\setminus Z,\overline{Y}\setminus \overline{Y}\cap Z) \simeq \H^n(\overline{X}\setminus \overline{Y},Z\setminus \overline{Y}\cap Z)^\vee (n) \;\; \mbox{ and }\;\; \H^n(\overline{X},\overline{Y}) \simeq \H^n(\overline{X}\setminus\overline{Y})^\vee(n).$$
We are therefore reduced to proving that the natural map $\H^n(\overline{X}\setminus \overline{Y},Z\setminus\overline{Y}\cap Z) \to \H^n(\overline{X}\setminus\overline{Y}) $ induces an isomorphism on $\gr^n_\F$. This follows from the same argument as in the proof of Proposition \ref{prop:genus-smooth-compactification},  which uses the fact that $\dim Z<n$. 
\end{proof}

\begin{ex}\label{ex:triple-point}

The following situation, which builds upon Example \ref{ex:not-a-good-compactification}, illustrates the relevance of the ``good compactification'' assumption in Proposition \ref{prop:genus-good-compactification} (this is inspired by \cite[Remark 3.3]{BD1}). Start with $X_0=\CC^2$ with coordinates $x,y$ and $Y_0=L\cup L'\cup L''$ the union of the three affine lines $L=\{x=0\}$, $L'=\{x=1\}$, and $L''=\{y=0\}$. Since both $X_0$ and $Y_0$ are contractible, we have $\H^i(X_0,Y_0)=0$ for all $i$. Now consider the compactification $\overline{X}_0=\mathbb{P}^2_\CC$, $\overline{Y}_0=\overline{L}\cup\overline{L'}\cup\overline{L''}$ the union of three projective lines forming a triangle with one vertex $\overline{L}\cap\overline{L'}$ on the line at infinity $Z_0=\overline{X}_0\setminus X_0$. It is \emph{not} a good compactification of $(X_0,Y_0)$ because it is not of product type near $\overline{L}\cap\overline{L'}$. Using Poincar\'{e} duality \eqref{eq:poincare-duality} one sees that the only non-zero cohomology groups of the pair $(\overline{X}_0,\overline{Y}_0)$ are $\H^2\simeq \QQ(0)$, $\H^3\simeq \QQ(-1)^{\oplus 2}$, and $\H^4\simeq \QQ(-2)$. 

Now fix a smooth compact connected complex curve $C$ of genus $g>0$, and consider the pair $(X,Y)=(X_0\times C,Y_0\times C)$ and its compactification $(\overline{X},\overline{Y})=(\overline{X}_0\times C,\overline{Y}_0\times C)$. It is \emph{not} a good compactification. By the K\"{u}nneth formula for relative cohomology, $\H^i(X,Y)$ vanishes for all $i$, and $\H^3(\overline{X},\overline{Y})\simeq \H^1(C)\oplus \QQ(-1)^{\oplus 2}$. Therefore $(X,Y)$ has genus $0$ while $(\overline{X},\overline{Y})$ has genus $g$.
\end{ex}

\subsection{Examples}

\subsubsection{The $0$-dimensional case}
A point has genus zero because $\H^0(\mathrm{pt})$ is pure of weight $0$.

\subsubsection{The $1$-dimensional case}

Let $C$ be a curve and $S\subset C$ be a finite set. By Corollary \ref{coro:genus-inequalities-simple} and the fact that $g(S)=0$, we have
$$g(C,S)=g(C),$$
the classical genus of $C$ (Proposition \ref{prop: genusofcurveandresolution}). Therefore, $(C,S)$ has genus zero if and only if $C$ does.

\subsubsection{Projective plane curves}

We start with a general statement about projective hypersurfaces, which follows from Corollary \ref{coro:genus-inequalities-simple}.

\begin{prop}\label{prop:genus-projective-hypersurface}
For a hypersurface $Y\subset \PP^n_\CC$ we have
$$g(\PP^n_\CC,Y)=g(Y).$$
\end{prop}

If $C\subset \PP^2_\CC$ is a curve, we therefore have
\begin{equation}\label{eq:genus-degree-curves}
g(\PP^2_\CC,C)=g(C),
\end{equation}
which by Proposition \ref{prop: genusofcurveandresolution} equals the classical genus of $C$. Recall that this is computed, for an irreducible curve $C$ of degree $d$ in $\PP^2_\CC$, by the formula
$$g(C)=\frac{(d-1)(d-2)}{2} - \sum_{P\in C} \delta_P(C),$$
where $\delta_P(C)$ is the \emph{$\delta$-invariant} of a point $P$. (It equals $0$ if and only if $P$ is a smooth point, and $1$ if $P$ is a node.) In particular, if $C$ is smooth then its genus is simply $\frac{(d-1)(d-2)}{2}$. Examples of projective plane curves of genus zero include all lines ($d=1$) and quadrics ($d=2$), along with cubics ($d=3$) with at least one singular point. 

For curves with many irreducible components, we have the following result.

\begin{prop}\label{prop:genus-curve-many-irreducible components}
Let $C_1,\ldots,C_r$ be curves in $\PP^2_\CC$ without common irreducible components. Then 
$$g(C_1\cup\cdots \cup C_r) = g(C_1)+\cdots +g(C_r).$$
In particular, any union of genus zero curves (e.g., lines, quadrics, singular cubics) has genus zero.
\end{prop}

\begin{proof}
This follows from Proposition \ref{prop: genusofcurveandresolution} and the fact that the resolution of singularities of $C_1\cup\cdots \cup C_r$ is the disjoint union of the resolutions of singularities of the $C_i$.
\end{proof}

\subsubsection{Projective hypersurfaces}

Let $Y$ be a hypersurface in $\PP^n_\CC$ defined by a homogeneous polynomial of degree $d$. By Proposition \ref{prop:genus-projective-hypersurface} we have $g(\PP^n_\CC, Y) = g(Y)$. If $Y$ is smooth and $n\geq 1$, Proposition \ref{prop:log-forms-have-simple-poles} and Remark \ref{rem:log-forms-on-projective-space} yield the genus-degree formula $g(Y)= \binom{d-1}{n}$, and in particular $g(Y)=0$ if and only if $d\leq n$. In general, we still have the following useful result \cite[Th\'{e}or\`{e}me 1]{delignedimca}, which is a Hodge-theoretic version of the Chevalley--Warning theorem.

\begin{thm}\label{thm:projective-hypersurface-of-small-degree}
A hypersurface in $\PP^n_\CC$ of degree $\leq n$ has genus zero.
\end{thm}

\begin{rem}
In general, the middle degree  cohomology $\H^{n-1}(Y)$ of a hypersurface $Y$ in $\PP^n_\CC$ of degree $d\leq n$ is not of Tate type, even though its $h^{0,n-1}$ vanishes. The first example of such a phenomenon is a smooth cubic $3$-fold ($n=4$ and $d=3$) for which $h^{1,2}=h^{2,1}=5$. By the work of Hirzebruch \cite[\S 22.1]{hirzebruch}, all Hodge numbers of smooth projective hypersurfaces have explicit expressions in terms of $d$ and $n$. 
\end{rem}

If $Y$ has many irreducible components, then one can bound its genus using Corollaries \ref{coro:genus-inequalities-two-hypersurfaces} and \ref{coro:genus-inequalities-special-case} and the following generalization of Theorem \ref{thm:projective-hypersurface-of-small-degree}.

\begin{thm}\label{thm:projective-suvariety-of-small-degree}
Let $Y_1,\ldots, Y_r$ be (possibly singular) hypersurfaces in $\PP^n_\CC$ of degrees $d_1,\ldots,d_r$. If $d_1+\cdots +d_r\leq n$ then $Y_1\cap\cdots \cap Y_r$ has genus zero.
\end{thm}

\begin{proof}
By induction on $r$, using inclusion-exclusion and applying \cite[Th\'{e}or\`{e}me 1]{delignedimca} to the hypersurfaces $Y_{i_1}\cup\cdots \cup Y_{i_k}$ which all have degree $\leq n$.
\end{proof}

In low degree one has the following result.

\begin{prop}
The following projective hypersurfaces have genus zero:
\begin{enumerate}[1)]
\item a union of hyperplanes;
\item a union of hyperplanes and one quadric.
\end{enumerate}
\end{prop}

\begin{proof}
This follows by induction on the number of hyperplanes, applying Corollary \ref{coro:genus-inequalities-two-hypersurfaces} to $X=$ projective space and $Y=$ a hyperplane and using the fact that projective spaces and quadrics have genus zero.
\end{proof}

\subsubsection{Linear fibrations}

In $\PP^n_\CC$ with homogeneous coordinates $(x_0:x_1:\cdots :x_n)$, let us consider a hypersurface defined by a homogeneous polynomial
$$f(x_0,x_1,\ldots,x_n)=x_0f^0(x_1,\ldots,x_n)+f_0(x_1,\ldots,x_n),$$
with $f^0$ and $f_0$ homogeneous polynomials in $x_1,\ldots,x_n$, and $f^0\neq 0$.

\begin{prop}\label{prop:linear-fibration-genus}
We have the equality
$$g(\PP^n_\CC, V(f)) = g(\PP^{n-1}_\CC, V(f^0)).$$
\end{prop}

\begin{proof}
Let $o$ denote the point $(1:0:\cdots:0)\in \PP^n_\CC$, and consider the projection from $o$, 
$$\PP^n_\CC\setminus \{o\}\to \PP^{n-1}_\CC \; ,\; (x_0:x_1:\cdots:x_n) \mapsto (x_1:\cdots :x_n).$$ 
It restricts to an isomorphism 
\begin{equation}\label{eq:projection-cone-in-proof-linear-fibration}
V(f) \setminus V(f,f^0) \overset{\sim}{\To}  \PP^{n-1}_\CC \setminus V(f^0)
\end{equation}
whose inverse is given by the rational map $(x_1:\cdots: x_n) \mapsto (-f_0: f^0 x_1:\cdots : f^0 x_n)$. The latter admits the following interpretation: identify the pencil of lines $\ell$ through $o$ with $\PP^{n-1}_\CC$, and consider the map $\ell \mapsto \ell \cap V(f)$ which is well-defined outside $V(f^0).$
The isomorphism \eqref{eq:projection-cone-in-proof-linear-fibration} induces an isomorphism in compactly supported cohomology, which, since $V(f)$ and $\PP^{n-1}_\CC$ are compact, gives an isomorphism of mixed Hodge structures
\begin{equation}\label{eq:intermediate-iso-cone-in-proof-linear-fibration}
\H^{n-1}( V(f), V(f,f^0)) \simeq \H^{n-1} (\PP^{n-1}_\CC, V(f^0)).
\end{equation}

The long exact sequence in relative cohomology \eqref{eq:long-exact-sequence-relative-cohomology} for the triple $(\PP^n_\CC, V(f), V(f,f^0))$ reads:
\begin{equation}\label{eq:long-exact-sequence-triple-in-proof-linear-fibration}
\cdots \to \H^{n-1} (\PP^{n}_\CC, V(f,f^0)) \to \H^{n-1} (V(f), V(f,f^0)) \to \H^{n} (\PP^n_\CC, V(f)) \to \H^n (\PP^n_\CC, V(f,f^0)) \to \cdots \end{equation}
By definition, $V(f,f^0)\subset \PP^n_\CC$ is a cone over $V(f_0,f^0)\subset \PP^{n-1}_\CC$ with cone point $o$, and the projection
$$V(f,f^0)\setminus\{o\} \to V(f_0,f^0)$$
is a locally trivial $\mathbb{A}^1_\CC$-bundle. Using the Leray spectral sequence in compactly supported cohomology, we therefore get isomorphisms of mixed Hodge structures
\begin{equation}\label{eq:something-is-a-tate-twist-in-proof-linear-fibration}
\H^\bullet_{\mathrm{c}}(V(f,f^0)\setminus\{o\}) \simeq \H^{\bullet-2}_{\mathrm{c}}(V(f_0,f^0))(-1),
\end{equation}
and therefore
$$\H^k_{\mathrm{c}}(V(f,f^0)\setminus\{0\})^{0,p}=0$$
for all $k$ and $p\geq 0$. By considering  the localization long exact sequence 
$$\cdots \to \H^{k-1}_{\mathrm{c}}(\{o\})\to \H^k_{\mathrm{c}}(V(f,f^0)\setminus \{o\}) \to \H^k_{\mathrm{c}}(V(f,f^0)) \to \H^k_{\mathrm{c}}(\{o\})\to \cdots $$
we deduce that $\H^k(V(f,f^0))^{0,p}= \H^k_{\mathrm{c}}(V(f,f^0))^{0,p}=0$ for all $k\geq 1$ and $p\geq 0$. The long exact sequence in relative cohomology \eqref{eq:long-exact-sequence-relative-cohomology} for the pair $(\PP^n_\CC, V(f,f^0))$ implies that $\H^k(\PP^n_\CC, V(f,f^0))^{0,p}=0$ for all $k$ and $p\geq 0$. Thus by \eqref{eq:long-exact-sequence-triple-in-proof-linear-fibration} we conclude that  
\[   \H^n (\PP^n_\CC, V(f))^{0,p} \simeq  \H^{n-1} (V(f), V(f,f^0))^{0,p}   \quad\mbox{ for all }p\geq 0,  \]
and the claim follows on applying the isomorphism  \eqref{eq:intermediate-iso-cone-in-proof-linear-fibration}. 
\end{proof}

\section{The combinatorial rank of a pair of varieties}\label{sec:combinatorial-rank}

    We study the notion of combinatorial rank for a pair of complex varieties. For the sake of convenience, we switch from homology to cohomology and restate Definition \ref{defi:rank} as follows.

    \begin{defi}
    Let $X$ be a complex variety of dimension $n$, and $Y\subset X$ be a closed subvariety. The \emph{combinatorial rank} of $(X,Y)$, denoted by $\rk(X,Y)$, is the Hodge number $h^{0,0}$ of $\H^n(X,Y)$, or equivalently the dimension of the weight zero subspace $\gr_0^\W\H^n(X,Y)=\W_0\H^n(X,Y)$.
    $$\rk(X,Y)=\dim\H^n(X,Y)^{0,0} = \dim\gr_0^\W\H^n(X,Y).$$
    \end{defi}

    Note that
    \begin{equation}\label{eq:cr-as-dim-of-canonical-forms}
    \mbox{if } (X,Y) \mbox{ has genus zero,} \qquad \rk(X,Y)=\dim\logforms{n}(X\setminus Y)
    \end{equation}
    is the dimension of the space of forms on $X\setminus Y$ with logarithmic poles at infinity, or, in other words, the maximum number of linearly independent canonical forms for $(X,Y)$.

\subsection{The case of a smooth variety}

    We let $\rk(X):=\rk(X,\varnothing)$. Since the $\H^0$ of a point is pure of weight $0$ and dimension $1$, the combinatorial rank of a point is $1$:
    $$\rk(\mathrm{pt})=1.$$
    For smooth varieties of positive dimension, the combinatorial rank vanishes:

    \begin{prop}\label{prop:rank-smooth}
    If $X$ is smooth of dimension $\geq 1$, then $\rk(X)=0$.
    \end{prop}

    \begin{proof}
    By \eqref{eq:bound-on-hodge-numbers-smooth} we have $h^{0,0}(\H^n(X))=0$ if $n=\dim(X)\geq 1$.
    \end{proof}

\subsection{The combinatorial rank as a combinatorial invariant}

    Let $X$ be a connected complex variety of dimension $n$, and $Y\subset X$ be a non-empty closed subvariety, decomposed into distinct irreducible components as $Y_1\cup\cdots \cup Y_N$. For a set $I\subset \{1,\ldots,N\}$ we let $Y_I=\bigcap_{i\in I} Y_i$ denote the corresponding intersection. This includes the special case $Y_\varnothing=X$. The following discussion is standard in the case of simple normal crossing divisors.

    \begin{defi}
    Assume that all the multiple intersections $Y_I$ are smooth, for $I\subset \{1,\ldots,N\}$. The \emph{dual complex} of $Y$, denoted by $\Delta(Y)$, is the abstract $\Delta$-complex with one $k$-simplex for each connected component of some multiple intersection $Y_I$ for $|I|=k+1$, with attaching maps prescribed by the inclusions of such intersections.
    \end{defi}

    Concretely, $\Delta(Y)$ has $N$ vertices, an edge between vertices $i,j$ for each connected component of $Y_i\cap Y_j$, a triangle with vertices $i,j,k$ for each connected component of $Y_i\cap Y_j\cap Y_k$, etc.

    \begin{prop}\label{prop:rank-as-combinatorial-invariant}
    If all the multiple intersections $Y_I$ are smooth, for $I\subset \{1,\ldots,N\}$, then
    $$\rk(X,Y) = \dim \widetilde{\H}_{n-1}(\Delta(Y)).$$
    \end{prop}

    \begin{proof}
    Consider the spectral sequence in relative homology \eqref{eq:relative-homology-spectral-sequence}, and note that $\gr_0^\W\H_q(Y_I)=0$ for all $I$ and $q>0$ as a consequence of 
    \eqref{eq:bound-on-hodge-numbers-smooth} and the smoothness assumption. It follows that $\gr_0^\W\H_\bullet(X,Y)$ is the homology of the complex
    $$\cdots \to \bigoplus_{|I|=n}\H_0(Y_I)\to \bigoplus_{|J|=n-1}\H_0(Y_J)\to \cdots \to\bigoplus_{i<j}\H_0(Y_i\cap Y_j)\to   \bigoplus_i \H_0(Y_i) \to \H_0(X)\to 0,$$
    where $\H_0(X)$ sits in degree $0$. By definition, this is, up to a shift, the complex which computes the reduced homology of $\Delta(Y)$, and the claim follows.
    \end{proof}

    The proposition applies in particular if $X$ is smooth and $Y=D$ is a simple normal crossing divisor, in which case $\Delta(D)$ has dimension $\leq n-1$ and $\gr_0^\W\H_n(X,D)$ is a subspace of $\bigoplus_{|I|=n}\H_0(Y_I)$ as explained in Proposition \ref{prop:corners}.

    \begin{ex}
    In $X=\PP^2_\CC$, consider $Y=C\cup L_1\cup L_2$ where $C$ is a smooth conic and $L_1$, $L_2$ are two distinct lines which intersect on $C$ (see Figure \ref{figureTwoLinesAndAConic}). The dual complex has three vertices $C$, $L_1$, $L_2$, five edges (two for $C\cap L_1$, two for $C\cap L_2$, and one for $L_1\cap L_2$), and one triangle for $C\cap L_1\cap L_2$. It is homotopy equivalent to a wedge of two circles (one can contract the triangle to a point), hence $\rk(X,Y)= \dim\H_1(\Delta(Y)) = 2$. The  shaded region (left) has a relative homology class in $\H_2(X,Y)$ whose image in $\gr_0^\W\H_2(X,Y)\simeq \H_1(\Delta(Y))$ can be represented by the outer cycle of the dual complex $\Delta(Y)$. (Note that in this case $\H_2(X,Y)$ is in fact isomorphic to $\gr_0^\W\H_2(X,Y)\simeq \QQ(0)^2$.)

    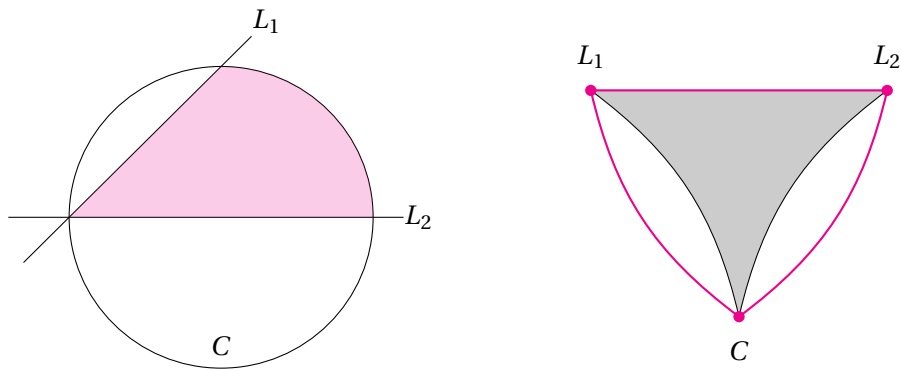
\begin{figure}[h]
    \begin{center}
    \begin{tikzpicture}[scale=2]
    \draw (1,0) circle (1);
    \draw (-.4,0) -- (2.2,0);
    \draw (-.3,-.3) -- (1.2,1.2);
    \fill[fill=magenta,opacity=0.2] (2,0) arc (0:90:1) -- (0,0) -- (2,0);
    \node at (1.3, 1.3) {$L_1$};
    \node at (2.3, 0) {$L_2$};
    \node at (1,-.85) {$C$};
    \end{tikzpicture}
    \hspace{1.5cm}
    \begin{tikzpicture}[scale=1.5]
    \draw[draw=black, fill=black, fill opacity=0.2] (-1.3,1) -- (1.3,1) to[bend right = 20] (0,-1) to[bend right = 20] (-1.3,1);
    
     \draw (-1.3,1) node[circle, fill=magenta, inner sep=1.5pt] {};
    \node at (-1.3,1.3) {$L_1$};
    \draw (1.3,1) node[circle, fill=magenta, inner sep=1.5pt] {};
    \node at (1.3,1.3) {$L_2$};
    \draw (0,-1) node[circle, fill=magenta, inner sep=1.5pt] {};
    \node at (0,-1.3) {$C$};
    \draw[magenta, thick] (-1.3,1) -- (1.3,1);

    \draw[magenta, thick] (-1.3,1) to[bend right = 20] (0,-1);
    \draw[magenta, thick] (1.3,1) to[bend left = 20] (0,-1);
    
    \end{tikzpicture}

    \end{center}
    \caption{The union of a conic and two lines in the projective plane (left), and its dual complex (right). The shaded region (left) corresponds to the outer cycle (right).}
    \label{figureTwoLinesAndAConic}
    \end{figure}
    \end{ex}

    If the multiple intersections $Y_I$ are not smooth, then it is not as easy to compute the combinatorial rank. In the case of a non-simple normal crossing divisor (e.g., the nodal cubic, see \S\ref{example: NodalCubic}), there is a generalization of Proposition \ref{prop:rank-as-combinatorial-invariant} which uses the notion of symmetric $\Delta$-complexes, see \cite[\S 3]{changalatiuspayne}. 

\subsection{Properties of the combinatorial rank}

\subsubsection{Invariance under modification} 

If $f:(X',Y')\to (X,Y)$ is a modification then we have
\begin{equation}\label{eq:rank-modification}
\rk(X,Y)=\rk(X',Y')
\end{equation}
because of the isomorphism of mixed Hodge structures \eqref{eq:iso-relative-cohomology-excision} for $\bullet=\dim(X)$.

\subsubsection{Products}

The combinatorial rank is multiplicative:

\begin{prop}\label{prop:product-rank}
For pairs $(X,Y)$ and $(X',Y')$ we have
$$\rk(X\times X', Y\times X'\cup X\times Y')=\rk(X,Y)\rk(X',Y').$$
\end{prop}

\begin{proof}
Let $n=\dim(X)$ and $n'=\dim(X')$. We have the K\"{u}nneth formula for relative cohomology:
$$\H^{n+n'}(X\times X', Y\times X'\cup X\times Y') \simeq \bigoplus_{i+i'=n+n'}\H^i(X,Y)\otimes \H^{i'}(X',Y').$$
By \eqref{eq:bound-on-hodge-numbers} the terms with $i>n$ or $i'>n'$ have vanishing $\gr_0^\W$, hence the claim.
\end{proof}

\subsubsection{Hartogs phenomenon}

The following proposition is proved in the same way as Proposition \ref{prop:hartogs}. 

\begin{prop}\label{prop:hartogs-rank}
Let $X$ be a complex variety, and $Y,Z\subset X$ be closed subvarieties. If $Z$ has codimension $\geq 2$, then 
$$\rk(X,Y\cup Z)=\rk(X,Y).$$
\end{prop}

\subsubsection{Recursion}

 The following proposition is proved in the same way as Proposition \ref{prop:genus-inequalities-new}.

\begin{prop}\label{prop:rank-inequalities-new}
Let $X$ be a complex variety of dimension $n$ and $Z\subset Y\subset X$ be closed subvarieties, with $Y$ of codimension $1$. We have the inequalities:
$$\rk(X,Z)\leq \rk(X,Y)\leq \rk(X,Z)+\rk(Y,Z).$$
If $\gr_0^\W\H^{n-1}(X,Z)=0$ (e.g., if $X$ is compact and $X\setminus Z$ is smooth and affine), then
$$\rk(X,Y)=\rk(X,Z)+\rk(Y,Z).$$
\end{prop}

Note that by Proposition \ref{prop:hartogs-rank}, if $Z$ has codimension $\geq 2$ in $X$ then $\rk(X,Z)=\rk(X)$. 

\begin{coro}\label{coro:rank-inequalities-simple}
Let $X$ be a complex variety of dimension $n$ and $Y\subset X$ be a closed subvariety of codimension $1$. We have the inequalities:
\begin{equation}\label{eq:rank-two-inequalities}
\rk(X)\leq \rk(X,Y)\leq \rk(X)+\rk(Y).
\end{equation}
If $\gr_0^\W\H^{n-1}(X)=0$, then
$$\rk(X,Y)=\rk(X)+\rk(Y).$$
In particular, if $X$ is smooth of dimension $n\geq 2$ then
$$\rk(X,Y)=\rk(Y).$$
\end{coro}

The fact that $\gr_0^\W\H^{n-1}(X)=0$ if $X$ is smooth and $n\geq 2$ follows from \eqref{eq:bound-on-hodge-numbers-smooth}. For an example where $\rk(X,Y)<\rk(X)+\rk(Y)$, take $X=\PP^1_\CC$, $Y=\mathrm{pt}$, and note that $\rk(X)=\rk(X,Y)=0$ and $\rk(Y)=1$.

We also have the following useful special cases of Proposition \ref{prop:rank-inequalities-new}, which are proved in the same way as Corollaries \ref{coro:genus-inequalities-two-hypersurfaces} and \ref{coro:genus-inequalities-special-case}.

\begin{coro}\label{coro:rank-inequalities-two-hypersurfaces}
Let $X$ be a complex variety of dimension $n$, and $Y,Y'\subset X$ be two closed subvarieties of codimension $1$. We have the inequalities:
\begin{equation}\label{eq:rank-inequality-intersection}
\rk(X,Y')\leq \rk(X,Y\cup Y')\leq \rk(X,Y')+\rk(Y,Y\cap Y').
\end{equation}
If $\gr_0^\W\H^{n-1}(X,Y')=0$ (e.g., if $X$ is compact and $X\setminus Y'$ is smooth and affine), then
$$\rk(X,Y\cup Y')=\rk(X,Y')+\rk(Y,Y\cap Y').$$
\end{coro}

\begin{coro}\label{coro:rank-inequalities-special-case}
Let $X$ be a complex variety of dimension $n$, and $Y\subset X$ be a closed subvariety of codimension $1$, decomposed into irreducible components as $Y=Y_1\cup\cdots \cup Y_N$. Set $Z=Z_1\cup\cdots \cup Z_N$ with $Z_i=Y_i\cap \bigcup_{j\neq i}Y_j$. We have the inequalities:
$$\rk(X) \leq \rk(X,Y) \leq \rk(X) + \sum_{i=1}^N \rk(Y_i,Z_i).$$
If $\gr_0^\W\H^{n-1}(X,Z)=0$, then
$$\rk(X,Y)=\rk(X)+\sum_{i=1}^N \rk(Y_i,Z_i).$$
\end{coro}

\subsubsection{Combinatorial rank and compactification}

The following proposition is proved in the same way as Proposition \ref{prop:genus-good-compactification}.

\begin{prop}\label{prop:rank-good-compactification}
Let $(X,Y)$ be a pair of complex varieties such that $X\setminus Y$ is smooth, and let $(\overline{X},\overline{Y})$ be a good compactification of $(X,Y)$. Then we have the equality
$$\rk(X,Y)=\rk(\overline{X},\overline{Y}).$$
\end{prop}

\begin{ex}\label{ex:triple-point-rank}
With the notation of Example \ref{ex:triple-point}, we see that $(X_0,Y_0)$ has combinatorial rank $0$, and $(\overline{X}_0,\overline{Y}_0)$ has combinatorial rank $1$, illustrating the necessity for the ``good compactification'' assumption in Proposition \ref{prop:rank-good-compactification}.
\end{ex}

\subsection{Examples}

\subsubsection{The $0$-dimensional case}
A point has combinatorial rank $\rk(\mathrm{pt})=1$.

\subsubsection{The $1$-dimensional case}

Let $C$ be a complex curve. For a point $P\in C$, we let $r_P(C)$ denote the number of local branches of $C$ at $P$ (it is usually called the \emph{branching number}). If $P$ is a smooth point then $r_P(C)=1$.  

\begin{prop}\label{prop:rank-curve}
Let $C$ be a connected complex curve.
\begin{enumerate}[1)]
\item Let $k$ denote the number of irreducible components of $C$. Then we have the equality:
$$\rk(C) = \left(\sum_{P\in C} (r_P(C)-1)\right) - (k-1).$$
\item For a non-empty finite subset $S\subset C$, we have the equality:
$$\rk(C,S)=\rk(C)+|S|-1.$$
\end{enumerate}
\end{prop}

\begin{proof}
The second statement follows from the long exact sequence in relative cohomology \eqref{eq:long-exact-sequence-relative-cohomology}: 
\begin{equation} \label{eq:zeroth-equality-in-proof-rank-curves}
0\to \H^0(C)\to \H^0(S)\to \H^1(C,S)\to \H^1(C)\to 0.
\end{equation}
For the first statement, we assume without loss of generality that $C$ has at least one singular point. Let $Z\subset C$ denote its singular locus, and note the equality
\begin{equation}\label{eq:first-equality-in-proof-rank-curves}
\rk(C)=\rk(C,Z)-|Z|+1,
\end{equation}
which follows from the second statement. The canonical resolution of singularities $\pi: \widetilde{C}\to C$ induces a modification $(\widetilde{C},\pi^{-1}(Z))\to (C,Z)$, and therefore $\rk(C,Z)=\rk(\widetilde{C},\pi^{-1}(Z))$.
Note that $\widetilde{C}$ is smooth and has $k$ connected components, each of which contains at least one point of $\pi^{-1}(Z)$.
By applying  \eqref{eq:zeroth-equality-in-proof-rank-curves} to $(\widetilde{C},Z)$ we deduce that:
\begin{equation}\label{eq:second-equality-in-proof-rank-curves}
\rk(\widetilde{C},\pi^{-1}(Z)) =|\pi^{-1}(Z)|-k.
\end{equation}
By combining \eqref{eq:first-equality-in-proof-rank-curves} and \eqref{eq:second-equality-in-proof-rank-curves} we get the equality
$$\rk(C)=|\pi^{-1}(Z)|-|Z| - (k-1) = \left(\sum_{P\in C}(|\pi^{-1}(P)|-1)\right)-(k-1),$$
and the claim follows from the fact that $|\pi^{-1}(P)|=r_P(C)$ for every point $P$.
\end{proof}

\subsubsection{Projective hypersurfaces}

The following proposition follows from Corollary \ref{coro:rank-inequalities-simple}.

\begin{prop}\label{prop:rank-projective-hypersurface}
Let $n\geq 2$. If $Y\subset \PP^n_\CC$  is a hypersurface, we have
$$\rk(\PP^n_\CC,Y)=\rk(Y).$$
\end{prop}

Similarly to the genus (Theorem \ref{thm:projective-hypersurface-of-small-degree}), projective hypersurfaces of small degree have vanishing combinatorial rank, by \cite[Th\'{e}or\`{e}me 1]{delignedimca}.

\begin{thm}\label{thm:projective-hypersurface-of-small-degree-rank}
A hypersurface in $\PP^n_\CC$ of degree $\leq n$ has combinatorial rank zero.
\end{thm}

We also have the following useful result.

\begin{prop}\label{prop:combinatorial-rank-projective-space}
Let $D\subset \PP^n_\CC$ be a normal crossing divisor of degree $d$, and assume that the pair $(\PP^n_\CC,D)$ has genus zero. Then its combinatorial rank is given by the formula
$$\rk(\PP^n_\CC, D) = \binom{d-1}{n}.$$
\end{prop}

\begin{proof}
This follows from Proposition \ref{prop:log-forms-have-simple-poles} and the description of differential forms with at most a simple pole on projective space (Remark \ref{rem:log-forms-on-projective-space}).
\end{proof}

\begin{rem}
Proposition \ref{prop:combinatorial-rank-projective-space} implies that $(\PP^n_\CC, D)$ has combinatorial rank $1$ if and only if $d=n+1$. By Proposition \ref{prop:adjoint-hypersurface-degree} this is equivalent to $(\PP^n_\CC, D)$ underlying a ``simplex-like positive geometry'' in the sense of Remark \ref{rem:simplex-like}.
\end{rem}

\subsubsection{Linear fibrations}

In $\PP^n_\CC$ with homogeneous coordinates $(x_0:x_1:\cdots :x_n)$, let us consider a hypersurface defined by a homogeneous polynomial
$$f(x_0,x_1,\ldots,x_n)=x_0f^0(x_1,\ldots,x_n)+f_0(x_1,\ldots,x_n),$$
with $f^0$ and $f_0$ homogeneous polynomials in $x_1,\ldots,x_n$, and $f^0\neq 0$. The following proposition is proved in the same way as Proposition \ref{prop:linear-fibration-genus}.

\begin{prop}\label{prop:linear-fibration-rank}
We have the equality
$$\rk(\PP^n_\CC, V(f)) = \rk(\PP^{n-1}_\CC, V(f^0)).$$
\end{prop}

\section{Examples}\label{sec:examples}

We give some examples of genus zero pairs and canonical forms, some of which are standard in the context of positive geometries \cite{arkanihamedbailampositive}.

\subsection{The point}

A point $X=\mathrm{pt}$ has genus zero and combinatorial rank $1$, and the map
$$\can:\H_0(\mathrm{pt}) \to \logforms{0}(\mathrm{pt})$$
is the natural inclusion of $\QQ$ inside $\CC$, i.e., sends the class of the $0$-cycle $\{\mathrm{pt}\}$ to the $0$-form $1$:
$$\omegacan_{\{\mathrm{pt}\}}=1.$$

\subsection{The projective line}

Let $a,b\in\PP^1_\CC$ be two distincts points. The pair $(\PP^1_\CC,\{a,b\})$ has genus zero and combinatorial rank $1$, and a basis of $\H_1(\PP^1_\CC,\{a,b\})$ is given by (the class of) a continuous path $\gamma_{a,b}$ from $a$ to $b$ in $\PP^1_\CC$. Its canonical form is
$$\omegacan_{\gamma_{a,b}}=\dlog\left(\frac{z-b}{z-a}\right),$$
with the convention that $z-\infty$ should be replaced with $1$ in this formula.

\subsection{Toric varieties}

Let $X$ be a compact toric variety over $\CC$, with its open dense torus $(\CC^*)^n\subset X$, and let $Y=X\setminus (\CC^*)^n$. Then by Poincar\'{e} duality \eqref{eq:poincare-duality},
$$\H_n(X,Y)\simeq \H^n((\CC^*)^n)(n) \simeq \QQ(0)$$
is $1$-dimensional, with a basis element $\sigma$, which is well-defined up to a sign because Poincar\'{e} duality holds over 
$\ZZ$. The pair $(X,Y)$ therefore has genus zero and combinatorial rank $1$, and the canonical form of $\sigma$ is, up to a sign, the standard $n$-form on $X\setminus Y=(\CC^*)^n$,
$$\omegacan_\sigma = \pm \frac{\dd z_1}{z_1}\wedge\cdots\wedge \frac{\dd z_n}{z_n}.$$

\begin{rem}
It was already proved in \cite[\S 5.6]{arkanihamedbailampositive} that toric varieties give rise to positive geometries in the sense of \emph{op.~cit.}.
\end{rem}

\subsection{The nodal cubic} \label{example: NodalCubic}
(Compare with \cite[\S 5.3.1]{arkanihamedbailampositive} in the context of positive geometries.) We work in the projective plane $\PP^2_\CC$ with homogeneous coordinates $(x_0:x_1:x_2)$ and the affine chart $\CC^2$ with affine coordinates $(x,y):=(x_1/x_0,x_2/x_0)$. Consider the \emph{nodal cubic}
$$Y=\{(x_0:x_1:x_2)\in\PP^2_\CC \,|\, x_0x_2^2=x_1^2(x_0-x_1)\}$$
with affine equation $y^2=x^2(1-x)$. Note that it is a (non-simple) normal crossing divisor. Indeed, near the singular point $(0,0)$, one can choose a holomorphic branch of $\sqrt{1-x}$ and consider the holomorphic coordinates $(x'=x\sqrt{1-x},y)$ for which the equation of $Y$ reads $(y-x')(y+x')=0$. 

The ``teardrop''
$$\sigma = \{(x,y)\in\RR^2 \, | \, x\geq 0,  y^2\leq x^2-x^3\},$$
shown in Figure \ref{figureNodalCubic} has boundary along $Y$. 

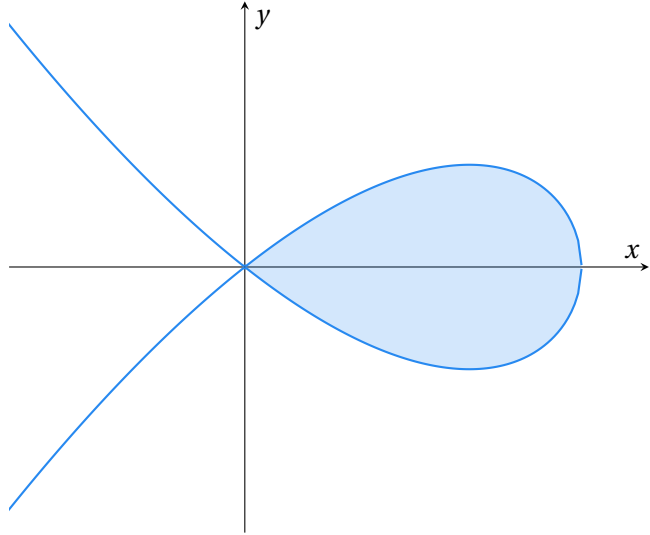
\begin{figure}[h]
    \centering
    \begin{tikzpicture}
    \begin{axis}[
    width=10cm, 
    domain=-1:1, 
    xmin=-.7, xmax=1.2, 
    ymin=-1, ymax=1,   
    axis x line = center, 
    axis y line = center,
    tick style={draw=none}, 
    ymajorticks=false, 
    xmajorticks=false,
    xlabel = {$x$},
    ylabel = {$y$},]
    \addplot [samples=200, thick, name path=f, color= MyBlue] {sqrt(x^2-x^3)};
    \addplot [samples=200, thick, name path=g, color=MyBlue] {-sqrt(x^2-x^3)};
    \addplot[MyBlue, opacity=0.2] fill between[of=f and g, soft clip={domain=0:1}];
    \end{axis}
    \end{tikzpicture}
    \caption{The nodal cubic $y^2=x^2(1-x)$.}
\label{figureNodalCubic}
\end{figure}

The long exact sequence in relative homology \eqref{eq:long-exact-sequence-relative-homology} shows that
$$\H_2(\PP^2,Y)\simeq \QQ(0),$$
with basis the class of $\sigma$. Therefore $(\PP^2,Y)$ has genus zero and combinatorial rank $1$.

\begin{prop}\label{prop:canonical-form-of-nodal-cubic}
The canonical form of $\sigma$ is given by the formula
$$\omegacan_\sigma = 2 \, \frac{x_0\dd x_1\wedge \dd x_2 - x_1\dd x_0\wedge \dd x_2 + x_2\dd x_0\wedge \dd x_1}{x_0x_2^2-x_1^2(x_0-x_1)} = 2\, \frac{\dd x\wedge \dd y}{y^2-x^2(1-x)}.$$
\end{prop}

We give several different proofs of this proposition.

\subsubsection{By computing iterated residues at corners}

Since $Y$ is a (non-simple) normal crossing divisor in $\PP^2_\CC$, we can compute canonical forms by using Corollary \ref{coro:corners}.
Let $\omega$ denote the right-hand side of the claimed equality, and let us note that it has logarithmic poles along $Y$. Indeed, this is clearly  the case away from $(0,0)$ (by proposition \ref{prop:log-forms-have-simple-poles}), and in local holomorphic coordinates $(x'=x\sqrt{1-x},y)$ near $(0,0)$ we can write
$$\omega= 2\,\varphi \frac{\dd x'\wedge \dd y}{(y-x')(y+x')} = \varphi \,\dlog(y+x')\wedge \dlog(y-x'),$$
where $\varphi$ is holomorphic and verifies $\varphi(0,0)=1$.
Locally around $(0,0)$, let $Y_+=\{y+x'=0\}$ and $Y_-=\{y-x'=0\}$ denote the two local branches of $Y$. We have
$$\Res_{Y_+\cap Y_-}\Res_{Y_-}(\omega)=1 \quad \mbox{ and } \quad \partial_{Y_+\cap Y_-}\partial_{Y_-}(\sigma)=1,$$
which proves the claim.

\subsubsection{By using the recursion}

Having noted that $\omega$ has logarithmic poles along $Y$ as in the first proof, we want to apply Proposition \ref{prop:recursion} for $Z=\{(0,0)\}$. Note that the normalisation of the nodal cubic produces a modification (Definition \ref{defi:modification})
$$p: (\PP^1_\CC,\{-1,1\})\to (Y,Z)$$
defined in affine coordinates by $t\mapsto (1-t^2, t(1-t^2))$. (Concretely, if $t\notin\{-1,1\}$, then $p(t)$ is the intersection of $Y\setminus Z$ with the line $y=tx$.) This implies that $(Y,Z)$ has genus zero. Via the isomorphism $\H_1(\PP^1_\CC,\{-1,1\})\simeq \H_1(Y,Z)$, the boundary of $\sigma$ along $Y$ corresponds to  the real interval $[-1,1]$ oriented from $-1$ to $1$. By the invariance under modification \eqref{eq:can-and-modification}, it is enough to check the equality
$$\Res_Y(\omega) \stackrel{?}{=} \omegacan_{\partial\sigma} = \omegacan_{[-1,1]}$$
in $\logforms{1}(Y\setminus Z)\simeq \logforms{1}(\PP^1_\CC\setminus \{-1,1\})$. We easily compute
$$\Res_Y(\omega) = \Res_Y\left(\frac{\dd x}{y}\wedge \dlog(y^2-x^2(1-x))\right) = \left.\frac{\dd x}{y}\right\vert_Y = \frac{-2\, \dd t}{1-t^2} = \dlog\left(\frac{t-1}{t+1}\right),$$
which is the canonical form of the interval $[-1,1]$, and the claim follows.

\subsubsection{By blowing up}

One replaces $Y$ by a \emph{simple} normal crossing divisor by performing the blow-up $\pi: P\rightarrow \PP^2$ along the point $(0,0)$. By the invariance by modification \eqref{eq:can-and-modification}, we have
$$\pi^*\omegacan_\sigma = \omegacan_{\widetilde{\sigma}},$$
where $\widetilde{\sigma}$ is the strict transform of $\sigma$, shown in Figure \ref{figureNodalCubicBlowup}.

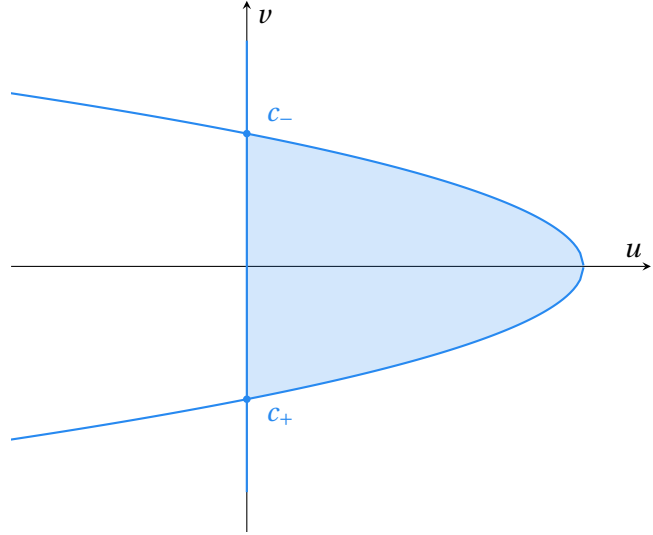
\begin{figure}[h]
    \centering
    \begin{tikzpicture}
    \begin{axis}[
    width=10cm, 
    domain=-1:1, 
    xmin=-.7, xmax=1.2, 
    ymin=-2, ymax=2,   
    axis x line = center, 
    axis y line = center,
    tick style={draw=none}, 
    ymajorticks=false, 
    xmajorticks=false,
    xlabel = {$u$},
    ylabel = {$v$},]
    \addplot [samples=200, thick, name path=f, color= MyBlue] {sqrt(1-x)};
    \addplot [samples=200, thick, name path=g, color=MyBlue] {-sqrt(1-x)};
    \addplot[MyBlue, opacity=0.2] fill between[of=f and g, soft clip={domain=0:1}];
    \addplot[thick, MyBlue] coordinates {(0, -1.7) (0, 1.7)};
    \node[color=MyBlue, circle,fill, color=MyBlue, inner sep=1pt] at (axis cs:0,1) {};
    \node[color=MyBlue, circle,fill, color=MyBlue, inner sep=1pt] at (axis cs:0,-1) {};
    \node[] at (axis cs:0.1,1.13) {$\textcolor{MyBlue}{c_-}$};
    \node[] at (axis cs:0.1,-1.13) {$\textcolor{MyBlue}{c_+}$};
    \end{axis}
    \end{tikzpicture}
    \caption{The nodal cubic $y^2=x^2(1-x)$ after blow-up of the node $(0,0)$, in the coordinate chart $(u,v)=(x,y/x)$. The vertical line $v=0$ is the exceptional divisor.}
\label{figureNodalCubicBlowup}
\end{figure}

It is contained in the affine chart of $P$ with coordinates $(u,v)=(x,y/x)$, where $\pi(u,v)=(u,uv)$. The exceptional divisor $E$ is given by the equation $u=0$, and the strict transform $\widetilde{Y}$ of $Y$, by the equation $v^2=1-u$.
One easily computes:
\[ \pi^* \omega =   2\,\frac{\dd u \wedge \dd v}{u(v^2-(1-u))}  = 2\,\frac{\dd u}{u}\wedge\frac{\dd v}{v^2-(1-u)}  .\]
From this expression, one easily verifies that $\pi^*\omega$ is indeed the canonical form of $\widetilde{\sigma}$, for instance by using Corollary \ref{coro:corners}.
Taking first the residue along the exceptional divisor yields
\[  \Res_E (\pi^*\omega) =  - 2 \, \frac{\dd v}{v^2-1} = \dlog\left(\frac{v+1}{v-1}\right)\ , \]
while $\partial_E(\widetilde{\sigma})$ is the interval $[-1,1]$ in the $v$-coordinate, oriented from $c_-=\{v-1=0\}$ to $c_{+}=\{v+1=0\}$. By defining the double residues along $c_\pm$ in this order (first along $E$ then along $v=\pm 1$),  we  obtain
$$\Res_{c_-}(\pi^*\omega)=-1 \quad \mbox{ and } \quad \Res_{c_+}(\pi^*\omega)=1,$$ 
which is consistent with  taking the double boundary (in the same order):
$$\partial_{c_-}(\widetilde{\sigma}) = -1 \quad \mbox{ and } \quad \partial_{c_+}(\widetilde{\sigma})= 1.$$

\subsection{The cuspidal cubic and  a line}\label{sec:cuspidal-cubic-and-line}

We work in the projective plane $\PP^2_\CC$ with homogeneous coordinates $(x_0:x_1:x_2)$ and the affine chart $\CC^2$ with affine coordinates $(x,y):=(x_1/x_0,x_2/x_0)$. Consider the \emph{cuspidal cubic}
$$C=\{(x_0:x_1:x_2)\in\PP^2_\CC \,|\, x_0x_2^2=x_1^3\}$$
with affine equation $y^2=x^3$. Its only singular point is $(0,0)$ and it is not a normal crossing divisor. It has genus zero because its resolution of singularities is 
$$\PP^1_\CC\to C \;  , \; (t_0:t_1)\mapsto (t_0^3:t_0t_1^2:t_1^3),$$
which in the affine coordinate $t=t_1/t_0$ corresponds to $t\mapsto (t^2,t^3)$. For any line $L\subset \PP^2_\CC$, Propositions \ref{prop:genus-projective-hypersurface} and \ref{prop:genus-curve-many-irreducible components} then implies that the pair $(\PP^2_\CC, C\cup L)$ has genus zero.

As for the combinatorial rank, we have, by Corollary \ref{coro:rank-inequalities-two-hypersurfaces} and Proposition \ref{prop:rank-projective-hypersurface}, $\rk(\PP^2_\CC, C\cup L)=\rk(\PP^2_\CC, L) + \rk(C,C\cap L) = \rk(L)+\rk(C,C\cap L)$. Using Proposition \ref{prop:rank-curve}, we see that $\rk(L)=0$ and $\rk(C,C\cap L)=|C\cap L|-1$. Hence,
$$\rk(\PP^2_\CC,C\cup L)=|C\cap L|-1.$$

\subsubsection{A line in general position}

Let $L\subset \PP^2_\CC$ denote the line defined by the projective equation $x_1=x_0$, i.e., the affine equation $x=1$.
Then $|C\cap L|=3$ (including a point at infinity) and the pair $(\PP^2_\CC, C\cup L)$ has combinatorial rank $2$. The domain
\[  \sigma = \{(x,y) \in \RR^2 \,|\, 0\leq x \leq 1, y^2\leq x^3\}  \ \subset  \  \PP^2(\RR), \]
shown in Figure \ref{figureCuspidalCubicAndLineGeneralPosition}, has  boundary contained in $C \cup L$, and therefore defines a class in $\H_2(\PP^2_\CC, C\cup L)$. We compute its canonical form after determining the space of forms on $\PP^2_\CC\setminus C\cup L$ with logarithmic poles at infinity, which we already know has dimension $2$ by \eqref{eq:cr-as-dim-of-canonical-forms}.

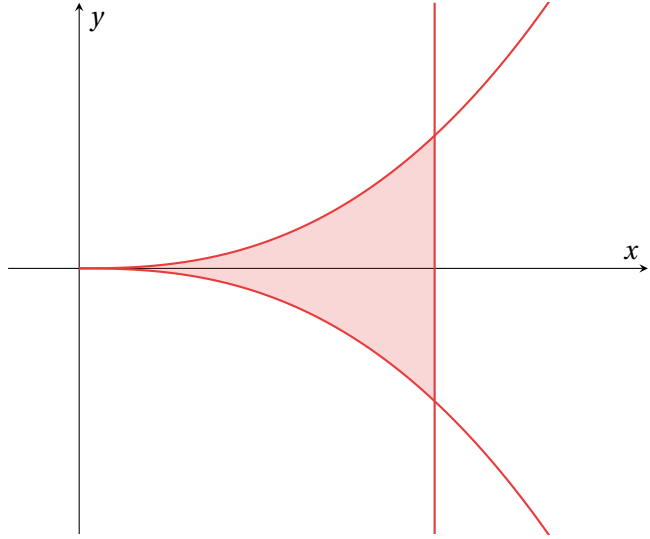
\begin{figure}[h]
    \centering
    \begin{tikzpicture}
    \begin{axis}[
    width=10cm, 
    domain=0:1.5, 
    xmin=-.2, xmax=1.6, 
    ymin=-2, ymax=2,   
    axis x line = center, 
    axis y line = center,
    tick style={draw=none}, 
    ymajorticks=false, 
    xmajorticks=false,
    xlabel = {$x$},
    ylabel = {$y$},]
    \addplot [samples=200, thick, name path=f, color= MyRed] {sqrt(x^5)};
    \addplot [samples=200, thick, name path=g, color=MyRed] {-sqrt(x^5)};
    \addplot[MyRed, opacity=0.2] fill between[of=f and g, soft clip={domain=0:1}];
    \addplot[MyRed, thick] coordinates {(1,-2) (1,2)};
    \end{axis}
    \end{tikzpicture}
    \caption{The cuspidal cubic $y^2=x^3$ and the line $x=1$.}
\label{figureCuspidalCubicAndLineGeneralPosition}
\end{figure}

\begin{prop}\label{prop:log-forms-cuspidal-cubic}
The space $\logforms{2}(\PP^2_\CC\setminus C\cup L)$ consists of the forms
$$(\alpha_1 x_1+\alpha_2 x_2) \frac{x_0\,\dd x_1\wedge \dd x_2 - x_1\,\dd x_0\wedge \dd x_2 + x_2\,\dd x_0\wedge \dd x_1}{(x_1-x_0)(x_0x_2^2-x_1^3)} = (\alpha_1 x + \alpha_2 y ) \frac{\dd x\wedge \dd y}{(x-1)(y^2-x^3)},$$
with $\alpha_1,\alpha_2\in \CC$. The canonical form of $\sigma$ is given by the formula
\[  \omegacan_{\sigma} =  -2x_1\,\frac{x_0\,\dd x_1\wedge \dd x_2 - x_1\,\dd x_0\wedge \dd x_2 + x_2\,\dd x_0\wedge \dd x_1}{(x_1-x_0)(x_0x_2^2-x_1^3)} = -2x\, \frac{\dd x\wedge \dd y}{(x-1)(y^2-x^3)} . \]
\end{prop}

\begin{proof}
By Proposition \ref{prop:criterion-logarithmic}, an element of the space $\logforms{2}(\PP^2_\CC\setminus C\cup L)$ is a two-form 
\begin{align}\label{eq:general-potential-log-form-cuspidal-cubic}
\omega & = (\alpha_0 x_0+\alpha_1 x_1+\alpha_2 x_2) \frac{x_0\,\dd x_1\wedge \dd x_2 - x_1\,\dd x_0\wedge \dd x_2 + x_2\,\dd x_0\wedge \dd x_1}{(x_1-x_0)(x_0x_2^2-x_1^3)} \\
&= (\alpha_0+\alpha_1 x + \alpha_2 y ) \frac{\dd x\wedge \dd y}{(x-1)(y^2-x^3)}, 
\end{align}
with $\alpha_0,\alpha_1,\alpha_2\in \CC$, such that $\Res_C(\omega)$ and $\Res_L(\omega)$ have logarithmic poles at infinity. The latter condition ends up being vacuous, and we consider the former by computing:
By noting that 
\[      \frac{\dd x\wedge \dd y}{(x-1)(y^2-x^3)}=   \frac{1}{2y} \frac{\dd x}{x-1} \wedge  \dd\log( y^2-x^3) \ , \]
we find that 
$$\Res_C(\omega) = (\alpha_0+\alpha_1x+\alpha_2 y)\frac{1}{2y}\frac{\dd x}{x-1}\Bigg|_C = (\alpha_0+\alpha_1t^2+\alpha_2 t^3)\frac{1}{2t^3}\frac{2t \,\dd t}{t^2-1} = \left(\frac{\alpha_0}{t^2}+\alpha_1+\alpha_2t\right)\frac{\dd t}{t^2-1},$$
where we work in the coordinates $(x,y)=(t^2,t^3)$ on $C^{\mathrm{reg}}\setminus C\cap L\simeq \PP^1_\CC\setminus \{-1,0,1,\infty\}$. Clearly, such a form always has logarithmic poles at $-1,1,\infty$, and has a logarithmic pole (and is in fact regular) at $0$ if and only if $\alpha_0=0$.

Now by Proposition \ref{prop:recursion}, the canonical form $\omega=\omegacan_\sigma$ must satisfy
$$\Res_C(\omega) = \omegacan_{\partial_C(\sigma)} = \dlog\left(\frac{t+1}{t-1}\right) = -2\frac{\dd t}{t^2-1}$$
because $\partial_C(\sigma)$ is the interval $[-1,1]$ oriented from $1$ to $-1$. Therefore, $\alpha_1=-2$ and $\alpha_2=0$.
\end{proof}

\begin{rem}
There are two other interesting relative cycles in this situation, namely the (topological closures of the) domains 
$$\sigma_+=\{(x,y)\in\RR^2\,|\, x\geq 1, y\geq 1, y^2-x^3\geq 0\} \quad \mbox{ and } \quad \sigma_-=\{x\geq 1, y\leq -1, y^2\geq x^3\},$$ 
which touch the point at infinity $(0:0:1)$. Their canonical forms equal
$$\omegacan_{\sigma_+} = -(x+y)\frac{\dd x\wedge \dd y}{(x-1)(y^2-x^3)} \quad \mbox{ and } \quad \omegacan_{\sigma_-} = (y-x)\frac{\dd x\wedge \dd y}{(x-1)(y^2-x^3)}.$$
Note the equality $\omegacan_\sigma = \omegacan_{\sigma_+}+\omegacan_{\sigma_-}$, which reflects the equality $\sigma = \sigma_+ + \sigma_-$ in $\H_2(\PP^2_\CC, C\cup L)$.
\end{rem}

\begin{rem}
Another strategy to determine the space $\logforms{2}(\PP^2_\CC\setminus C\cup L)$ is to resolve the singularities of $C\cup L$ by repeatedly blowing up points. Consider first the blow-up $\pi_1:P_1\to \PP^2_\CC$ along the cusp $(x,y)=(0,0)$, shown in Figure \ref{figureCuspidalCubicBlowup}. It is given in an affine chart by the formula $\pi_1(u,v)=(u,uv)$, for which the exceptional divisor $E_1$ is $\{u=0\}$ and the strict transforms of $C$ and $L$ are $C_1=\{u=v^2\}$ and $L_1=\{u=1\}$ respectively. 

\begin{figure}[h]
    \centering
    \begin{tikzpicture}
    \begin{axis}[
    width=10cm, 
    domain=0:1.4, 
    xmin=-.2, xmax=1.2, 
    ymin=-2, ymax=2,   
    axis x line = center, 
    axis y line = center,
    tick style={draw=none}, 
    ymajorticks=false, 
    xmajorticks=false,
    xlabel = {$u$},
    ylabel = {$v$},]
    \addplot [samples=200, thick, name path=f, color= MyRed] {sqrt(x)};
    \addplot [samples=200, thick, name path=g, color=MyRed] {-sqrt(x)};
    \addplot[MyRed, opacity=0.2] fill between[of=f and g, soft clip={domain=0:1}];
    \addplot[thick, MyRed] coordinates {(0, -1.7) (0, 1.7)};
    \addplot[thick, MyRed] coordinates {(1, -1.7) (1, 1.7)};
    \end{axis}
    \end{tikzpicture}
    \caption{The cuspidal cubic $y^2=x^3$ and the line $x=1$ after blow-up of the node $(0,0)$, in the coordinate chart $(u,v)=(x,y/x)$. The vertical line $v=0$ is the exceptional divisor.}
\label{figureCuspidalCubicBlowup}
\end{figure}
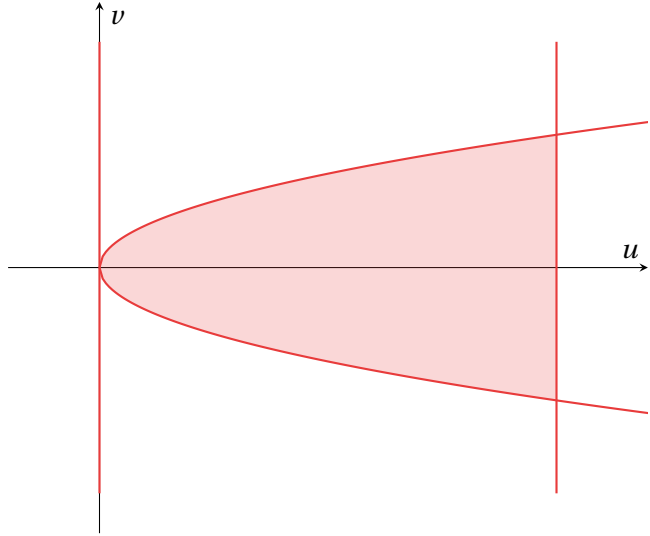

The pullback of a form \eqref{eq:general-potential-log-form-cuspidal-cubic} in this chart is 
$$\pi_1^*\omega = (\alpha_0+\alpha_1u+\alpha_2uv)\frac{\dd u\wedge\dd v}{u(u-1)(v^2-u)}.$$
Since the strict transform of $C$ is tangent to the exceptional divisor at the point $(u,v)=(0,0)$, we must perform a further blow up $\pi_2: P_2 \rightarrow P_1$ of the origin $(u,v)=(0,0)$.
It is given in an affine chart  by the formula $\pi_2(r,s)=(rs,s)$, for which the exceptional divisor is $F_2=\{s=0\}$ and the strict transforms of $E_1$, $C$, and $L$ are $E_2=\{r=0\}$, $C_2=\{r=s\}$ and $L_2=\{rs=1\}$ respectively. Computing the pullback we get:
$$\pi_2^*\pi_1^*\omega = (\alpha_0+\alpha_1rs+\alpha_2rs^2)\frac{\dd r \wedge \dd s}{rs(rs-1)(s-r)}.$$
Alas, the three divisors $\{r=0\}$, $\{r=s\}$, $\{s=0\}$ do not cross normally at the origin $(r,s)=(0,0)$ and it behoves us to perform a final blow-up $\pi_3: P_3 \rightarrow P_2$ of this point, 
given in an affine chart by $\pi_3(p,q)=(p,pq)$, for which the exceptional divisor is $G_3=\{p=0\}$ and the strict transforms of $F_2$, $C_2$, and $L_2$, are $F_3=\{q=0\}$, $C_3=\{q=1\}$, and $L_3=\{p^2q=1\}$ respectively. (The total transform $E_3$ of $E_2$ is not visible in this chart.) We get:
$$\pi_3^*\pi_2^*\pi_1^*\omega = (\alpha_0+\alpha_1p^2q+\alpha_2p^3q^2)\frac{\dd p\wedge \dd q}{p^2 q(p^2q-1)(q-1)}.$$
Because of the term $\frac{\dd p}{p^2}$, this form is logarithmic near $(p,q)=(0,0)$ if and only if $\alpha_0=0$. One can check that under this assumption, it does indeed have logarithmic poles along the normal crossing divisor $C_3\cup L_3\cup E_3\cup F_3\cup G_3$.
\end{rem}

\subsubsection{A line passing through the cusp}

We now consider the  variant where the line $L$ passes through the cusp while cutting the cuspidal cubic in another point. Let $L'\in\PP^2_\CC$ denote the line defined by the projective equation $x_2=x_1$, i.e., the affine equation $y=x$. Then $|C\cap L|=2$ and the pair $(\PP^2_\CC,C\cup L')$ has combinatorial rank $1$. The domain
$$\sigma'=\{(x,y)\in\RR^2 \, | \, 0\leq y\leq x \leq 1, y^2\geq x^3\} \subset \PP^2(\RR),$$
shown in Figure \ref{figureCuspidalCubicAndLineThroughCusp}, has  boundary contained in $C\cup L'$, and therefore defines a class in $\H_2(\PP^2_\CC, C\cup L')$. We compute its canonical form, after determining the space of forms on $\PP^2_\CC\setminus C\cup L'$ with logarithmic poles at infinity, which we already know has dimension $1$ by \eqref{eq:cr-as-dim-of-canonical-forms}.

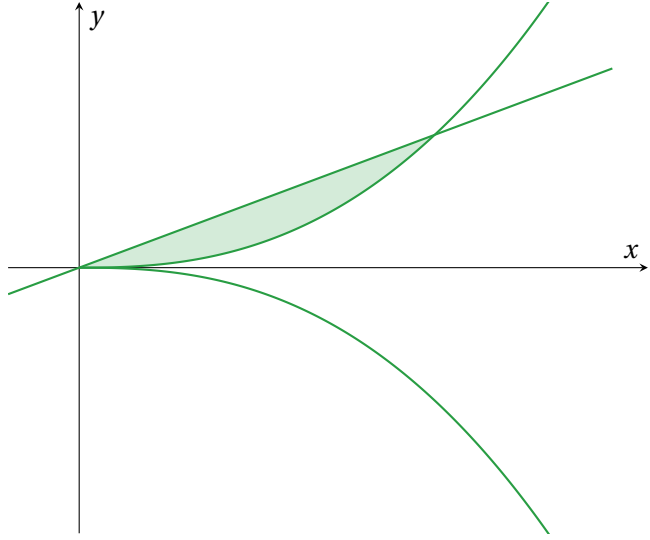
\begin{figure}[h]
    \centering
    \begin{tikzpicture}
    \begin{axis}[
    width=10cm, 
    domain=0:1.5, 
    xmin=-.2, xmax=1.6, 
    ymin=-2, ymax=2,   
    axis x line = center, 
    axis y line = center,
    tick style={draw=none}, 
    ymajorticks=false, 
    xmajorticks=false,
    xlabel = {$x$},
    ylabel = {$y$},]
    \addplot [samples=200, thick, name path=f, color= MyGreen] {sqrt(x^5)};
    \addplot [samples=200, thick, name path=g, color=MyGreen] {-sqrt(x^5)};
    \addplot [domain=-.2:1.5, samples=200, thick, name path=h, color=MyGreen] {x};
    \addplot[MyGreen, opacity=0.2] fill between[of=f and h, soft clip={domain=0:1}];
    \end{axis}
    \end{tikzpicture}
    \caption{The cuspidal cubic $y^2=x^3$ and the line $y=x$.}
\label{figureCuspidalCubicAndLineThroughCusp}
\end{figure}

\begin{prop}\label{prop:log-forms-cuspidal-cubic-degenerate}
The space $\logforms{2}(\PP^2_\CC\setminus C\cup L')$ consists of the forms
$$\alpha_2 x_2 \,\frac{x_0\,\dd x_1\wedge \dd x_2 - x_1\,\dd x_0\wedge \dd x_2 + x_2\,\dd x_0\wedge \dd x_1}{(x_2-x_1)(x_0x_2^2-x_1^3)} = \alpha_2 y \,\frac{\dd x\wedge \dd y}{(y-x)(y^2-x^3)},$$
with $\alpha_2\in \CC$. The canonical form of $\sigma'$ is given by the formula
\[  \omegacan_{\sigma'} =  x_2\,\frac{x_0\,\dd x_1\wedge \dd x_2 - x_1\,\dd x_0\wedge \dd x_2 + x_2\,\dd x_0\wedge \dd x_1}{(x_2-x_1)(x_0x_2^2-x_1^3)} = y\, \frac{\dd x\wedge \dd y}{(y-x)(y^2-x^3)} . \]
\end{prop}

\begin{proof}
By Proposition \ref{prop:criterion-logarithmic}, an element of the space $\logforms{2}(\PP^2_\CC\setminus C\cup L')$ is a form 
\begin{align*}
\omega  & = (\alpha_0 x_0+\alpha_1 x_1+\alpha_2 x_2) \frac{x_0\,\dd x_1\wedge \dd x_2 - x_1\,\dd x_0\wedge \dd x_2 + x_2\,\dd x_0\wedge \dd x_1}{(x_2-x_1)(x_0x_2^2-x_1^3)} \\
& = (\alpha_0+\alpha_1 x + \alpha_2 y ) \frac{\dd x\wedge \dd y}{(y-x)(y^2-x^3)}, 
\end{align*}
with $\alpha_0,\alpha_1,\alpha_2\in \CC$, which is such that $\Res_C(\omega)$ and $\Res_{L'}(\omega)$ have logarithmic poles at infinity. We consider the first condition by computing:
$$\Res_C(\omega) = (\alpha_0+\alpha_1x+\alpha_2 y)\frac{1}{2y}\frac{\dd x}{y-x}\Bigg|_C = (\alpha_0+\alpha_1t^2+\alpha_2 t^3)\frac{1}{2t^3}\frac{2t \,\dd t}{t^3-t^2} = \left(\frac{\alpha_0}{t^3}+\frac{\alpha_1}{t}+\alpha_2\right)\frac{\dd t}{t(t-1)},$$
where we  work in the coordinates $(x,y)=(t^2,t^3)$ on $C\setminus C\cup L'\simeq \PP^1_\CC\setminus \{0,1,\infty\}$. Clearly, such a form always has logarithmic poles at $1$ and $\infty$, and has a logarithmic pole at $0$ if and only if $\alpha_0=\alpha_1=0$. If these conditions are satisfied, one easily checks that $\Res_L(\omega)$ has a logarithmic pole at infinity (In fact, this must be the case since we know that $\logforms{2}(\PP^2_\CC\setminus C\cup L')$ has dimension $1$ by \eqref{eq:cr-as-dim-of-canonical-forms}).

Now, by Proposition \ref{prop:recursion}, the canonical form $\omega=\omegacan_{\sigma'}$ must satisfy 
$$\Res_C(\omega)=\omegacan_{\partial_C(\sigma')} = \dlog\left(\frac{t-1}{t}\right) = \frac{\dd t}{t(t-1)},$$
because $\partial_C(\sigma')$ is the interval $[0,1]$ oriented from $0$ to $1$. We conclude that $\alpha_2=1$.
\end{proof}

\subsection{A cubic surface and a plane}

In $\mathbb{P}^3_\CC$, let $C$ be a smooth cubic surface and  $P$ a plane  such that  $C\cap P$ is a nodal cubic in $P\simeq \PP^2_\CC$. 
\begin{figure}[h] 
\begin{center}
\includegraphics[width=9cm]{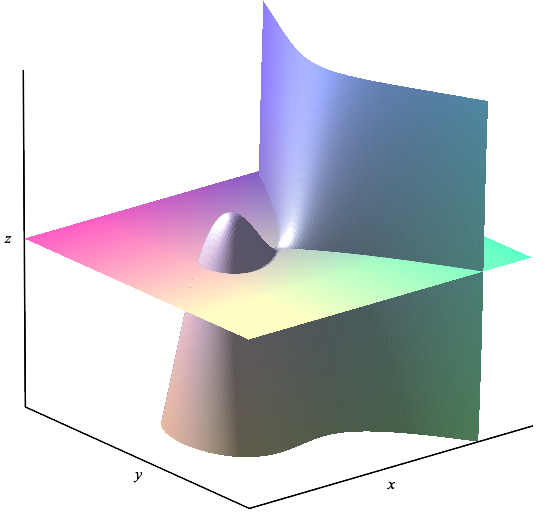}
\end{center}
\caption{(``Coastal sea stack''). The cubic surface $C: y^2=x^2(1-x)+z+z^3$ and the  hyperplane $P: z=0$, where $C\cap P: y^2=x^2(1-x)$ is the nodal cubic. The interior of the region bounded by the dome on the left (sea stack) and the hyperplane $z=0$ (the sea) defines a relative homology class in $\H_3(\PP^3, C \cup P)$. 
}\label{figureCubicSurfaceAndLine}
\end{figure}

Recall that $C$ has genus zero by Theorem \ref{thm:projective-hypersurface-of-small-degree} and that $C\cap P$ has genus zero because its resolution of singularities is $\PP^1_\CC$. Therefore, by Proposition \ref{prop:genus-inequalities-new}:
$$g(\PP^3_\CC,C\cup P)\leq g(\PP^3_\CC, C)+ g(P,C\cap P) \leq g(\PP_\CC^3)+g(C)+g(P)+ g(C\cap P)=0,$$
and hence $(\PP^3_\CC, C\cup P)$ has genus zero. We can also apply Proposition \ref{prop:rank-inequalities-new} to compute the combinatorial rank: since $\PP^3_\CC$ is compact and $\PP^3_\CC\setminus C$ is affine,
$$\rk(\PP_\CC^3,C\cup P) = \rk(\PP^3_\CC, C)+\rk(P,C\cap P)=0+1=1,$$
where we have used the fact that $\rk(\PP^3_\CC, C)=\rk(C)=0$ (Propositions \ref{prop:rank-projective-hypersurface} and \ref{prop:rank-smooth}), and the computation of the combinatorial rank of the nodal cubic (\S\ref{example: NodalCubic}). The next proposition gives a more detailed description of $\H^3(\PP^3_\CC, C\cup P)$ and is included for completeness, but is not required for the purposes of understanding the associated canonical form. 

\begin{prop}
We have
$$\gr_w^\W\H^3(\PP^3_\CC, C\cup P) \simeq \begin{cases} \QQ(0) & \mbox{ for } w=0;\\ \QQ(-1)^6 & \mbox{ for } w=2; \\ 0 & \mbox{ otherwise}.\end{cases}$$
\end{prop}

\begin{proof}
We consider the spectral sequence in relative cohomology \eqref{eq:relative-cohomology-spectral-sequence}. Recall \cite{HuybrechtsCubic} that $\H^1(C)$ vanishes and that $\H^2(C)\simeq \QQ(-1)^7$. One easily sees that $\gr_w^\W\H^3(\PP^3_\CC, C\cup P)$ vanishes for $w\in\{1,3\}$, and  $\gr_2^\W$ is the cohomology in the middle of the following complex:
$$0\to \H^2(\PP^3_\CC) \to \H^2(C)\oplus \H^2(P)\to \H^2(C\cap P)\to 0.$$
The result follows from the fact that $\H^2(\PP^3_\CC)\to \H^2(C)$ is injective and that $\H^2(P)\to \H^2(C\cap P)$ is an isomorphism.
\end{proof}

By way of example, consider the cubic surface defined in affine coordinates $(x,y,z)$ by the equation $y^2=x^2(1-x)+z+z^3$, and where $P$ is defined by $z=0$, depicted in Figure \ref{figureCubicSurfaceAndLine}.

The intersection $C\cap P$ is the nodal cubic $y^2=x^2(1-x)$. Consider the region  
\[ \sigma = \{ (x,y,z) \in \RR^3 \,|\,  z\geq 0\ , \  z+z^3 \leq  y^2-x^2(1-x)     \} ,\]
which defines a class in $\H_3(\PP^3_\CC, C\cup P)$.

\begin{prop}
The canonical form of $\sigma$ is given in affine coordinates by the formula
$$\omegacan_\sigma = -2\, \frac{\dd x\wedge \dd y\wedge \dd z}{z(y^2-x^2(1-x)-z-z^3)}$$
\end{prop}

\begin{proof}
We work in projective coordinates $(x_0:x_1:x_2:x_3)$ such that $(x,y,z)=(x_1/x_0, x_2/x_0, x_3/x_0)$ and consider the form 
\begin{align*}
    \omega & = \frac{x_0\,\dd x_1\wedge \dd x_2\wedge \dd x_3 - x_1\,\dd x_0\wedge\dd x_2\wedge \dd x_3 + x_2\,\dd x_0\wedge \dd x_1\wedge \dd x_3 - x_3\,\dd x_0\wedge \dd x_1\wedge \dd x_2}{x_3(x_0x_2^2-x_1^2(x_0-x_1)-x_0^2x_3-x_3^3)} \\
    & = \frac{\dd x\wedge \dd y\wedge \dd z}{z(y^2-x^2(1-x)-z-z^3)}.
\end{align*}
It generates the space of forms on $\PP^3_\CC\setminus C\cup P$ with at most a simple pole along $C\cup P$, i.e. the space of global sections of $\Omega^3_{\PP^3_\CC}(C\cup P)$. By Proposition \ref{prop:log-forms-have-simple-poles}, this space contains the subspace $\logforms{3}(\PP^3_\CC\setminus C\cup P)$, which has dimension $1$ by \eqref{eq:cr-as-dim-of-canonical-forms}, and hence the two spaces are equal: 
$$\logforms{3}(\PP^3_\CC\setminus C\cup P) = \Omega^3_{\PP^3_\CC}(C\cup P) = \CC \omega.$$
Therefore, we have $\omegacan_\sigma=\alpha\omega$ for some $\alpha\in \CC$, which one finds by using the recursion formula
$$\Res_P(\omegacan_\sigma) = \omegacan_{\partial_P(\sigma)}.$$
Since $\partial_P(\sigma)$ is the ``teardrop'' from \S\ref{example: NodalCubic} with reversed orientation, Proposition \ref{prop:canonical-form-of-nodal-cubic} yields $\alpha=-2$. 
\end{proof}

\subsection{The determinant locus and positive semi-definite matrices}
For $n\geq 2 $, let  $\PP^{d_n}_\CC$ denote the space of projective $n\times n$ symmetric matrices, where $d_n=\binom{n+1}{2}-1$,  and $\mathrm{Det} \subset \PP^{d_n}_\CC$  the locus defined by the vanishing of the determinant, which is homogeneous in matrix entries. The subspace $P\subset \PP^{d_n}(\RR)$ consisting of  projective classes of positive definite symmetric matrices satisfies $\partial(P) \subset \mathrm{Det}(\RR)$ and is orientable because $P$ is a homogeneous space for the action of $\mathrm{SL}_n(\RR)$, which is connected.   Therefore $[P]$ defines a relative homology class in $\H_{d_n}(\PP^{d_n}_\CC , \mathrm{Det})$.    A volume form may be given explicitly  by contracting  the form $\eta=  \det(X)^{-(n+1)/2} \, \bigwedge_{1\leq i \leq j \leq n} dX_{ij} $  (for some choice of ordering in the exterior product, and taking the positive square root of $\det(X)$) with the Euler vector field $\sum_{i,j} X_{ij} \frac{\partial}{\partial X_{ij}} $, where $X= (X_{ij})$ is a positive definite symmetric matrix.   In the case $n=3$, the locus $\mathrm{Det}$ is isomorphic to the graph hypersurface $X_{W_3}$ of the wheel  graph $W_3$ with $3$ spokes (or complete graph $K_4$), and the corresponding volume form is proportional to the Feynman differential form $\Omega_{W_3} / \Psi_{W_3}^2$.  It was shown in \cite{BlochEsnaultKreimer} that $\H^5(\PP^5_\CC \setminus X_{W_3})\simeq \QQ(-3)$ and it follows that  $[P]$ is non-vanishing in $\H_5(\PP^5_\CC, \mathrm{Det})$.  

By the following result, $P$ has a well-defined canonical form, which is $\omegacan_P=0$ (in particular, it is not the volume form!).

\begin{prop} 
The pair $ (\PP^{d_n},\mathrm{Det})  $ has genus $0$ and combinatorial rank $0$.
\end{prop}

\begin{proof}  
This follows from Theorems \ref{thm:projective-hypersurface-of-small-degree} and \ref{thm:projective-hypersurface-of-small-degree-rank} since the determinant hypersurface has degree $n\leq d_n$. It also follows from repeated application  of Propositions \ref{prop:linear-fibration-genus} and \ref{prop:linear-fibration-rank}, since the determinant of a symmetric matrix $A$ is linear in any entry  $a_{i,i}$ on the diagonal, the corresponding coefficient being the determinant of the symmetric matrix obtained from $A$ by deleting row $i$ and column $i$.
\end{proof}

\subsection{A class of hypersurfaces of genus zero and combinatorial rank \texorpdfstring{$1$}{1}}

Let $n\geq 1$. We  work in projective space $\PP^n_\CC$ with homogeneous coordinates $(x_0:\cdots :x_n)$ and consider a hypersurface
\[ Y =\{  x_0 x_1 \ldots x_n =  \psi(x_1,\ldots, x_n)\} \]
where $\psi$ is  a homogeneous polynomial of degree $n+1$ which  does not depend on $x_0$. By Propositions \ref{prop:linear-fibration-genus} and \ref{prop:linear-fibration-rank} we have that
$$g(\PP^n_\CC, Y)=g(\PP^{n-1}_\CC,V(x_1\cdots x_n))=0 \qquad \mbox{ and } \qquad \rk(\PP^n_\CC, Y)=\rk(\PP^{n-1}_\CC, V(x_1\cdots x_n))=1.$$

\begin{prop}\label{prop:canonical-form-rank-one-hypersurface}
For every class $\sigma\in \H_n(\PP^n_\CC,Y)$, the canonical form $\omegacan_\sigma$ is a complex multiple of
$$\omega = \frac{\sum_{i=0}^n (-1)^i x_i\, \dd x_0\wedge\cdots \wedge\widehat{\dd x_i}\wedge\cdots \wedge \dd x_n}{x_0x_1\cdots x_n-\psi(x_1,\ldots,x_n)}.$$
\end{prop}

\begin{proof}
This follows from Proposition \ref{prop:log-forms-have-simple-poles} and the fact that $\omega$ is a generator of the  space of forms on $\PP^n_\CC\setminus Y$ with at most a simple pole along $Y$, which has dimension one, since the degree $d=n+1$ (remark \ref{rem:log-forms-on-projective-space}). Since $(\PP^n_\CC,Y)$ has combinatorial rank $1$, we have the equality $\logforms{n}(\PP^n_\CC\setminus Y)=\CC\omega$.
\end{proof}

The multiplicative constant in Proposition \ref{prop:canonical-form-rank-one-hypersurface} is actually a rational number (and even an integer if $\sigma$ lives in relative homology with integer coefficients). It may be computed as follows.
   Consider the blow-up $\pi:P\to \PP^n_\CC$ along the point $(1:0:\cdots :0)$, and identify the exceptional divisor $E$ with $\PP^{n-1}_\CC$ with homogeneous coordinates $(x_1:\cdots:x_n)$. Let $\widetilde{Y}\subset P$ denote the strict transform of $Y$, and note that $\widetilde{Y}\cap E$ is identified with $\{x_1\cdots x_n=0\}\subset \PP^{n-1}_\CC$.
We consider the linear map (taking the  boundary along $E$, combined with excision):
$$\partial_E:\H_n(\PP^n_\CC, Y)\simeq \H_n(P,\widetilde{Y}\cup E) \To \H_{n-1}(E,\widetilde{Y}\cap E) = \H_{n-1}(\PP^{n-1}_\CC, \{x_1\cdots x_n=0\}).$$
By adapting the proof of Proposition \ref{prop:linear-fibration-genus}, one sees that it induces an isomorphism on the level of $\gr_0^\W$.  Note that $\H_{n-1}(\PP^{n-1}_\CC,\{x_1\cdots x_n=0\})$ is one-dimensional, with basis given by (the relative homology class of) the real $(n-1)$-simplex $\Delta^{n-1}:=\{x_1\geq 0,\ldots, x_n\geq 0\}\subset \PP^{n-1}(\RR)$. 
We note that
$$\omegacan_{\Delta^{n-1}} = \pm \dlog(x_1/x_n)\wedge\cdots \wedge\dlog(x_{n-1}/x_n),$$
where the sign depends on a chosen orientation of $\Delta^{n-1}$.
Furthermore, with $\omega$ as in Proposition \ref{prop:canonical-form-rank-one-hypersurface}, we easily compute
$$\Res_E(\pi^*\omega) = \dlog(x_1/x_n)\wedge\cdots\wedge \dlog(x_{n-1}/x_n).$$
For $\sigma\in \H_n(\PP^n_\CC,Y)$, there exists $a\in \QQ$ such that $\partial_E(\sigma)=\pm a\Delta^{n-1}$ in $\H_{n-1}(\PP^{n-1}_\CC,\{x_1\cdots x_n=0\})$, and we therefore have
$$\omegacan_\sigma=\pm a\omega,$$
where the sign depends on the chosen orientation on $\Delta^{n-1}.$

By way of example, let us put
\begin{equation} \label{psiexample} 
\psi= \sum_{i=1}^n x_i^{n+1} \ . 
\end{equation}
The case $n=2$ is isomorphic to the nodal cubic considered in \S\ref{example: NodalCubic} via the change of coordinates $(x_0,x_1,x_2)=(8x'_0-6x'_1,x'_1+x'_2,x'_1-x'_2)$; the case $n=3$ is pictured below (Figure \ref{figureSextic}).
Any of the four skittles bounds a domain $\sigma$ such that $\partial_E(\sigma)$ equals, up to a sign, the class of the triangle $\Delta^2$ inside the exceptional divisor $E\simeq \PP^2_\CC$. This shows that the canonical form of each skittle is
$$\omegacan_\sigma=\pm \frac{\dd x\wedge \dd y\wedge \dd z}{xyz-(x^4+y^4+z^4)}.$$

\begin{figure}[h] \begin{center}
\quad {\includegraphics[width=7cm]{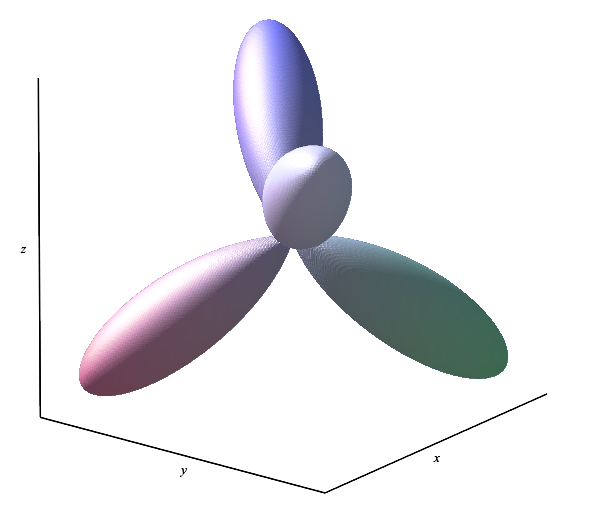}} 
{\includegraphics[width=7cm]{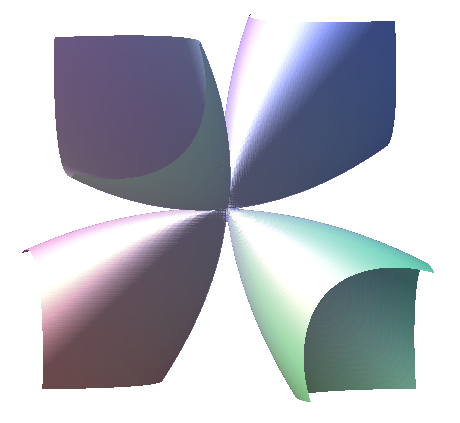}} 
\end{center}
\caption{The  surface $xyz = x^4+y^4+z^4 $ on the left (``helium''). The surface of any of the four skittles bounds a three-dimensional  domain. A zoom of the singular point at the origin (right) shows how this surface becomes tangent to the coordinate hyperplanes $xyz=0$ at the origin. 
} \label{figureSextic}
\end{figure}

\subsection{A genus zero pair of positive rank which is not recursive.}\label{sec:example-non-recursive-positive-geometry}
  We exhibit a genus zero pair $(X,Y)$ such that $X$ is a smooth compact surface of genus $0$ and $Y$ is a curve of genus $1$. This is an example for which the second inequality in Corollary \ref{coro:genus-inequalities-simple} is strict. Contrary to the easier example in \S\ref{sec:genus-zero-not-recursive} which has combinatorial rank $0$, this one has combinatorial rank $2$ and therefore two linearly independent canonical forms. One \emph{cannot} compute these canonical forms recursively using Proposition \ref{prop:recursion} because $Y$ has genus $1$, and we leave it to others to compute them.
  
  We work in the projective plane $\PP^2_\CC$ with homogeneous coordinates $(x_0:x_1:x_2)$ and the affine chart $\CC^2$ with affine coordinates $(x,y):=(x_1/x_0,x_2/x_0)$. Consider the plane elliptic curve $E\subset \PP^2_\CC$ with affine equation $y^2 = x^3 - 4x$ and  likewise $E'\subset \PP^2_\CC$ defined by the affine equation $y^2= x^3 + x$. There is an explicit isogeny of degree two between them, 
    \[ f : E' \to E   \ \mbox{ given in affine coordinates by  } \  f(x,y) = (x+ x^{-1},  y ( 1- x^{-2}) )  \ . \]
Let  $ X = \PP^1_\CC \times E$ and let $Y$ denote the image of the morphism
\[ (\pi_x, f)  : E' \To \PP^1_\CC \times E, \]
where $\pi_x:E'\rightarrow \PP^1_\CC$ is the projection onto the $x$ coordinate. We denote by 
$$\phi:E'\to Y$$
the resulting morphism of algebraic curves.

\begin{lem}\label{lem:fibers-isogeny}
For $P\in Y$, the fiber $\phi^{-1}(\{P\})$ consists of two points if $P= (1,(2,0))$ or $(-1,(-2,0))$, and one point otherwise. 
\end{lem}

\begin{proof}
The fiber of $\phi$ at the point at infinity consists of the point at infinity in $E'$, and for the other points we can work in affine coordinates. If two distinct points of $E'$ with affine coordinates $(x,y)$ and $(x',y')$ have the same image under $\phi$, then $x=x'$ and hence $y(1-x^{-2})=y'(1-x^{-2})$, which implies that $x^2=1$. The points of $E'$ with $x=1$ are $(1,\pm\sqrt{2})$ and their images under $\phi$ are $(1,(2,0))$. The points of $E'$ with $x=-1$ are $(-1,\pm\sqrt{-2})$, and their images under $\phi$ are $(-1,(-2,0))$.
\end{proof}

\begin{prop}
The pair $(X,Y)$ has genus $0$ and combinatorial rank $2$, while $Y$ has genus $1$.
\end{prop}

\begin{proof}
Lemma \ref{lem:fibers-isogeny} implies that $\phi:E'\to Y$ is the resolution of singularities of $Y$, and hence $g(Y)=1$. Furthermore, Proposition \ref{prop:rank-curve} yields $\rk(Y)=2$ because all the points of $Y$ have branching number $1$, except for $(1,(2,0))$ and $(-1,(-2,0))$ which have branching number $2$. Since $X$ is a smooth surface, Corollary \ref{coro:rank-inequalities-simple} implies that $\rk(X,Y)=2$. To prove that $(X,Y)$ has genus zero, consider the long exact sequence in relative cohomology \eqref{eq:long-exact-sequence-relative-cohomology}:
\[  \cdots \To \H^1(X) \stackrel{i^*}{\To} \H^1(Y) \To \H^2(X,Y) \To \H^2(X) \To \cdots,  \]
where   $i: Y \hookrightarrow X$ is the inclusion. 
Since $\H^2(X)\simeq \QQ(-1)^2$ by the K\"unneth formula, it suffices to show that 
  $i^*: \H^1(X) \rightarrow \H^1(Y)$ induces an isomorphism on  $(1,0)$ components.  To see this, note that the composition
\[  E' \overset{\phi}{\To} Y  \overset{i}{\hookrightarrow} X  \overset{\pi_2}{\To} E  \]
is the isogeny $f$,  where $\pi_2: X = \PP^1_\CC \times E \rightarrow E$ is the projection onto the second component. The maps $f^*$ and $\pi_2^*$ are isomorphisms on $\H^1$; since $\phi$ is the resolution of singularities of $Y$, the same argument as in the proof of Proposition \ref{prop: genusofcurveandresolution} shows that $\phi^*$ is an isomorphism on the $(1,0)$ component of $\H^1$. The same therefore holds  for $i^*$, and hence  $(X,Y)$ has genus zero.
\end{proof}

\section{The case of arrangements of hyperplanes, and convex polytopes}\label{sec:hyperplane-arrangements}

Arrangements of hyperplanes in projective space provide a simple class of genus zero pairs of algebraic varieties. In this setting, we derive general formulas for canonical forms.

\subsection{Genus and combinatorial rank}

We refer the reader to \cite{dimcaarrangements} for basic facts about arrangements of hyperplanes. Let $n\geq 1$ and let $\mathcal{A}\subset \PP^n_\CC$ be a non-empty finite union of projective hyperplanes. By a theorem due to Brieskorn \cite{brieskorn}, $\H^k(\PP^n_\CC\setminus\mathcal{A})$ is pure of weight $2k$ for all $k$. By Remark \ref{rem:2-pure-canonical-forms}, the pair $(\PP^n_\CC,\mathcal{A})$ has genus zero and we have an isomorphism
$$\can_\CC:\H_n(\PP^n_\CC,\mathcal{A})_\CC \stackrel{\sim}{\to} \logforms{n}(\PP^n_\CC\setminus \mathcal{A}).$$

The combinatorial rank of $(\PP^n_\CC,\mathcal{A})$, which is the dimension of $\H_n(\PP^n_\CC,\mathcal{A})$, is determined as follows. A \emph{flat} of $\mathcal{A}$ is a projective subspace of $\PP^n_\CC$ which is the intersection of some of the hyperplanes in $\mathcal{A}$, and we view the set of flats as a finite poset ordered by reverse inclusion. It has a least element $\hat{0}$, the ambient projective space $\PP^n_\CC$, and if $\mathcal{A}$ is \emph{essential} (meaning that the intersection of all the hyperplanes in $\mathcal{A}$ is empty) it also has a greatest element $\hat{1}$, the empty set. We let $\mu_{\mathcal{A}}$ denote the M\"{o}bius function of the poset of flats.

\begin{prop}\label{prop:rank-hyperplane-arrangement-moebius}
We have
$$\rk(\PP^n_\CC,\mathcal{A})= \begin{cases} 
(-1)^{n-1}\mu_{\mathcal{A}}(\hat{0},\hat{1}) & \mbox{ if }\mathcal{A} \mbox{ is essential;}\\
0 & \mbox{ otherwise.}
\end{cases}$$
\end{prop}

\begin{proof}
By Poincar\'{e} duality \eqref{eq:poincare-duality}, the combinatorial rank of $(\PP^n_\CC,\mathcal{A})$ equals the dimension of $\H^n(\PP^n_\CC\setminus\mathcal{A})$, which is given by \cite[Corollary 3.6]{dimcaarrangements}.
\end{proof}

\subsection{Bounded regions}

Assume that $\mathcal{A}$ is the complexification of a real arrangement of hyperplanes $\mathcal{A}_\RR$ in $\PP^n_\RR$. The topological closure of a connected component of $\PP^n_\RR\setminus \mathcal{A}_\RR$ is a projective polytope, and is called a \emph{region} of $\mathcal{A}_\RR$. Let $H_\infty$ be a hyperplane in $\PP^n_\RR$ which is generic with respect to $\mathcal{A}_\RR$ in the sense that the intersection of $H_\infty$ with any flat of $\mathcal{A}_\RR$ of codimension $r$ has codimension $r+1$. We identify $\PP^n_\RR\setminus H_\infty$ with $\RR^n$. A region of $\mathcal{A}_\RR$ which does not intersect $H_\infty$ is an affine polytope, and is called a \emph{bounded region} of $\mathcal{A}_\RR\cap \RR^n$. We fix an orientation of $\RR^n$ and give each bounded region the induced orientation.

The following result is probably well-known, but we could not find it in the literature.

\begin{prop}\label{prop:bounded-regions}
A basis of $\,\H_n(\PP^n_\CC,\mathcal{A})$ is given by the (classes of the) bounded regions of $\mathcal{A}_\RR\cap \RR^n$.
\end{prop}

Using Proposition \ref{prop:rank-hyperplane-arrangement-moebius}, one sees that this is consistent with (and in fact implies) Zaslavsky's formula \cite[Theorem C]{zaslavsky} for the number of bounded regions of a real arrangement of affine hyperplanes;  the genericity assumption on $H_\infty$ ensuring that the poset of flats of the affine arrangement $\mathcal{A}_\RR\cap \RR^n$ coincides with that of $\mathcal{A}$. Note that our proof uses Zaslavsky's result.

\begin{proof}
By induction on the number of hyperplanes in $\mathcal{A}$. The case of one hyperplane is clear (there is no bounded region). For the induction step, let $H$ be one of the hyperplanes in $\mathcal{A}$ and consider the ``deleted'' arrangement of hyperplanes $\mathcal{A}'$ defined as the union of the other hyperplanes. We will also work with the ``restricted'' arrangement of hyperplanes $\mathcal{A}'\cap H$ on $H\simeq \PP^{n-1}_\CC$. Let us fix an orientation of the real hyperplane $H_\RR\cap \RR^n$. 

Figure \ref{fig:bounded-regions} shows an example of an arrangement of real affine hyperplanes $\mathcal{A}_\RR\cap \RR^2$ with $5$ bounded regions. Note that the deleted arrangement $\mathcal{A}'_\RR\cap \RR^2$ has $3$ bounded regions and that the restricted arrangement $\mathcal{A}'\cap H_\RR\cap \RR^2$ has $2$ bounded regions.

\begin{figure}[h]
\begin{center}

\begin{tikzpicture}[scale=.7]
    \begin{axis}[
    width=12cm, 
    domain=-2:2.5, 
    xmin=-.7, xmax=2.3, 
    ymin=-.15, ymax=1.2,   
    axis x line = none, 
    axis y line = none,
    tick style={draw=none}, 
    ymajorticks=false, 
    xmajorticks=false,
    x post scale=1.8]
    \addplot[samples=200, name path = h] {0};
    \addplot[samples=200, dashed, name path = f] {1-x};
    \addplot[samples=200] {1+1.7*x};
    \addplot[samples=200, domain=-.5:2.3, name path = g] {.5-.25*x};
    \addplot[samples=200] coordinates{(0,-.5) (0,1.2)};
    \addplot[MyGreen, opacity=0.1] fill between[of=f and g, soft clip={domain=0:2/3}];
    \addplot[MyBlue, opacity=0.2] fill between[of=g and h, soft clip={domain=0:2/3}];
    \addplot[MyBlue, opacity=0.2] fill between[of=f and h, soft clip={domain=2/3:1}];
    \addplot[MyBlue, opacity=0.15] fill between[of=f and g, soft clip={domain=2/3:1}];
    \addplot[MyBlue, opacity=0.15] fill between[of=g and h, soft clip={domain=1:2}];
    \addplot[MyBlue, thick] coordinates { (1,0) (2/3,1/3) };
    \addplot[MyGreen, thick] coordinates { (0,1) (2/3,1/3) };
    \node at (axis cs:-.25,1.1) {$H$};
    \node[color=MyBlue] at (axis cs:.4,.2) {$P_+$};
    \node[color=MyBlue] at (axis cs:1.15,.1) {$P_-$};
    \node[color=MyGreen] at (axis cs:.17,.62) {$Q$};
    \end{axis}
\end{tikzpicture}

\end{center}
\caption{An arrangement of real affine hyperplanes $\mathcal{A}_\RR\cap \RR^2$, with a chosen hyperplane $H_\RR\cap \RR^2$ (dashed). In blue: a bounded region $P$ of $\mathcal{A}'_\RR\cap \RR^2$ which is cut into two bounded regions $P_+$ and $P_-$ of $\mathcal{A}_\RR\cap \RR^2$ by $H_\RR\cap \RR^2$. In green: a bounded region $Q$ of $\mathcal{A}_\RR\cap \RR^2$ with a facet lying on $H_\RR\cap \RR^2$.} \label{fig:bounded-regions}
\end{figure}
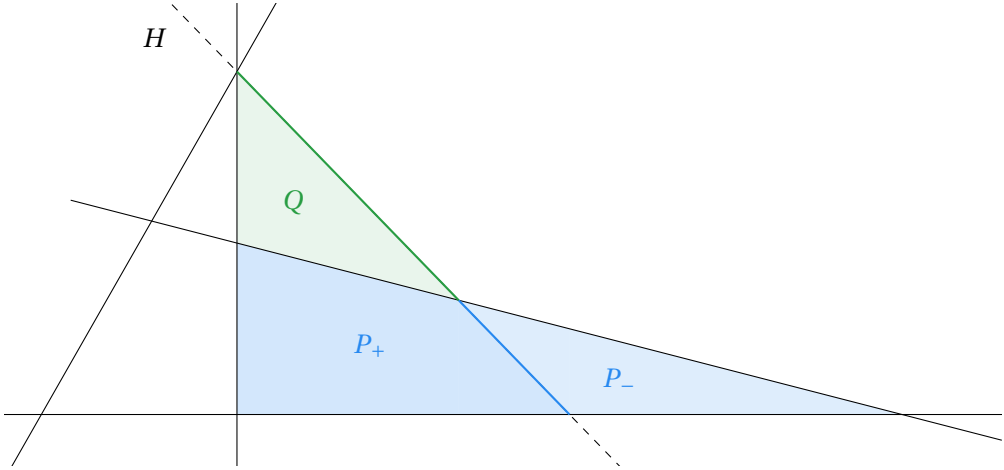

Let $\mathcal{B}$ (resp. $\mathcal{B}'$, $\mathcal{B}''$) denote the set of bounded regions of $\mathcal{A}_\RR\cap \RR^n$ (resp. of $\mathcal{A}'_\RR\cap \RR^n$, of $\mathcal{A}'_\RR\cap H_\RR\cap \RR^n$). We consider the following (yet to be defined) diagram.
\begin{equation}\label{eq:commutative-diagram-bounded-regions}
\begin{gathered}
\diagram{
0\ar[r] & \H_n(\PP^n_\CC,\mathcal{A}') \ar[r]^-{i} & \H_n(\PP^n_\CC,\mathcal{A})\ar[r]^-{\partial_H}&  \H_{n-1}(H,\mathcal{A}'\cap H)\ar[r] & 0 \\
0 \ar[r] & \QQ\mathcal{B}' \ar[u]^-{\sim}\ar[r]^{f} & \QQ\mathcal{B} \ar[u]\ar[r]^-{g} & \QQ\mathcal{B}''\ar[r]\ar[u]_-{\sim} & 0
}
\end{gathered}
\end{equation}
The first row is the ``deletion-restriction'' short exact sequence, which is a part of the long exact sequence in relative homology \eqref{eq:long-exact-sequence-relative-homology}, the map $\partial_H$ being the partial boundary map \eqref{eq:partial-boundary-map-defi}. (It splits into short exact sequences as shown in the diagram for weight reasons, because of the fact that $\H^k(\PP^n_\CC\setminus \mathcal{A})$ is pure of weight $2k$ for all $k$, and Poincar\'{e} duality \eqref{eq:poincare-duality}.) Each vertical arrow sends a bounded region to its relative homology class. By the induction hypothesis, the leftmost and rightmost vertical arrows are isomorphisms, and if we manage to define $f$ and $g$ such that the diagram commutes and such that its second row is a short exact sequence, then the middle vertical arrow is an isomorphism and we have completed the induction step. \vspace{2pt}

To define $f$, let $P$ be a bounded region of $\mathcal{A}'_\RR\cap \RR^n$.  There are two cases:\vspace{-1.5pt}
\begin{enumerate}[1)]
\item If the interior of $P$ does not meet $H_\RR$, we set $f(P)=P$;
\item If $H_\RR\cap \RR^n$ cuts $P$ into two bounded regions $P_+$ and $P_-$ of $\mathcal{A}_\RR\cap \RR^n$, and we set $f(P)=P_++P_-$ (see Figure \ref{fig:bounded-regions}).
\end{enumerate}
\vspace{-2pt}
The first square of \eqref{eq:commutative-diagram-bounded-regions} then commutes.\vspace{2pt}

To define $g$, let $Q$ be a bounded region of $\mathcal{A}_\RR\cap \RR^n$. There are two cases:\vspace{-1.5pt}
\begin{enumerate}[1)]
\item If $Q$ does not have a facet which is contained  in $H_\RR$, we set $g(Q)=0$.
\item If $Q$ has a facet contained in $H_\RR$, we set $g(Q)=\pm Q\cap H_\RR$ where the sign depends on the choices of orientations on $\RR^n$ and $H_\RR\cap \RR^n$ (see Figure \ref{fig:bounded-regions}).
\end{enumerate}
\vspace{-2pt}
With the correct signs, the second square of \eqref{eq:commutative-diagram-bounded-regions} commutes.
\vspace{2pt}

One easily sees that $\mathrm{Im}(f)=\mathrm{Ker}(g)$, and this implies that the map $\QQ\mathcal{B}\to \H_n(\PP^n_\CC,\mathcal{A})$ is injective. By Proposition \ref{prop:rank-hyperplane-arrangement-moebius} and \cite[Theorem C]{zaslavsky}, the dimension of $\H_n(\PP^n_\CC,\mathcal{A})$ equals the cardinality of $\mathcal{B}$ (by the genericity assumption on $H_\infty$, the poset of flats of the affine arrangement $\mathcal{A}_\RR\cap \RR^n$ coincides with that of $\mathcal{A}$). A dimension argument gives the claim.
\end{proof}

\begin{rem}
Proposition \ref{prop:bounded-regions} does not hold in general without the assumption that the hyperplane at infinity $H_\infty$ is generic with respect to $\mathcal{A}_\RR$. For instance, let $\mathcal{A}_\RR$ consist of three distinct lines $L_1,L_2,L_3$ in $\mathbb{P}^2_\RR$. Then $\H_2(\PP^2_\CC,\mathcal{A})$ has dimension $1$ and a basis consisting of one of the four triangles bounded by $L_1, L_2, L_3$ in $\PP^2_\RR$. However, if $H_\infty$ is a fourth line that passes through $L_1\cap L_2$, then $\mathcal{A}_\RR\cap \RR^2$ consists of three distinct lines, two of which are parallel; it does not have any bounded region.
\end{rem}

\subsection{The Orlik--Solomon relations, and the nbc basis}

We now turn to the classical description of the algebra of logarithmic forms on the complement of an arrangement of hyperplanes. It will be convenient to make the following choices.
\begin{enumerate}[--]
\item We choose a hyperplane $H_0\in\mathcal{A}$ and treat it as the hyperplane at infinity, identifying $\PP^n_\CC\setminus H_0$ with $\CC^n$. (Alternatively, one can think of this hyperplane as being \emph{added} to an arrangement for the sole purpose of giving a clean description of logarithmic forms, see Remark \ref{rem:hyperplane-at-infinity-irrelevant} below.)
\item We choose a linear order $H_1,\ldots,H_N$ of the other hyperplanes in $\mathcal{A}$.
\end{enumerate}
Writing $H_i=\{f_i=0\}$ for $i=0,\ldots,N$, the algebra $\logforms{\bullet}(\PP^n_\CC\setminus \mathcal{A}) = \logforms{\bullet}(\CC^n\setminus (H_1\cup\cdots \cup H_N))$ is generated by the $1$-forms
\begin{equation*}
\omega_i:=\dlog(f_i/f_0),
\end{equation*}
and the algebraic relations that they satisfy are consequences of two types of relations, called \emph{Orlik--Solomon relations} \cite{orliksolomon}:
\begin{enumerate}[a)]
\item For indices $1\leq i_1<\cdots <i_r\leq N$ such that the hyperplanes $H_{i_1},\ldots, H_{i_r}$ do not intersect in $\CC^n$,
$$\omega_{i_1}\wedge\cdots \wedge \omega_{i_r}=0.$$
\item For indices $1\leq i_1<\cdots <i_r\leq N$ such that the hyperplanes $H_{i_1},\ldots, H_{i_r}$ do intersect in $\CC^n$ and are linearly dependent,
$$\sum_{j=1}^r (-1)^j \omega_{i_1}\wedge \cdots \wedge \widehat{\omega_{i_j}}\wedge\cdots \wedge \omega_{i_r}=0,$$
where the notation means that $\omega_{i_j}$ is omitted from the wedge product.
\end{enumerate}

We now introduce a standard basis, called \emph{nbc basis}, of the algebra of logarithmic forms.

\begin{defi}
A \emph{circuit} of $\mathcal{A}$ is a set $\{i_1<\cdots <i_k\}\subset\{1,\ldots,N\}$ such that the hyperplanes $H_{i_1},\ldots, H_{i_k}$ intersect in $\CC^n$ and are linearly dependent, and which is minimal for this property. The corresponding \emph{broken circuit} is the set $\{i_2<\cdots <i_k\}$ obtained by removing the first element of the circuit. A \emph{non-broken-circuit} set   of $\mathcal{A}$ (or \emph{nbc} for short) is a subset of $\{1,\ldots,N\}$ that does not contain a broken circuit.
\end{defi}

\begin{prop}
The $\omega_{i_1}\wedge\cdots \wedge\omega_{i_k}$, for $\{i_1<\cdots<i_k\}$ an nbc set, form a basis of $\,\logforms{k}(\PP^n_\CC\setminus\mathcal{A})$.
\end{prop}

\begin{proof}
This is \cite[Theorem 3.55]{orlikterao}.
\end{proof}

\subsection{Canonical forms}

For a set $I=\{i_1<\cdots<i_n\}\subset \{1,\ldots,N\}$ such that $H_{i_1},\ldots, H_{i_n}$ are linearly independent, we let $H_I:=H_{i_1}\cap \cdots \cap H_{i_n}$ be the corresponding point, and let 
$$\partial_I : \H_n(\PP^n_\CC\setminus\mathcal{A}) \to \H_0(H_I)=\QQ$$
denote the composition of the partial boundary maps
$$\partial_{H_{i_1}\cap\cdots \cap H_{i_n}}\circ\cdots \circ \partial_{H_{i_{n-1}}\cap H_{i_n}}\circ \partial_{H_{i_n}}.$$
If $\mathcal{A}$ is the complexification of a real arrangement of hyperplanes $\mathcal{A}_\RR$ and $\sigma=P$ is a region of $\mathcal{A}_\RR$, then $\partial_I(P)$ is in $\{-1,0,1\}$, and is $0$ if $P$ does not contain $H_I$ as a vertex. We also set
$$\omega_I:=\omega_{i_1}\wedge\cdots \wedge\omega_{i_n}.$$

\begin{rem}\label{rem:iterated-boundaries-hyperplanes}
In the case when the arrangement of hyperplanes is not normal crossing near $H_I$, the iterated boundary map $\partial_I$ depends on the linear order on $I$ (which is fixed by the chosen linear order on the set of hyperplanes) in a more subtle way than simply determining an overall sign. For instance, let $\mathcal{A}$ be the complexified arrangement of lines in $\PP^2_\CC$ consisting of $5$ lines, with $L_0$ at infinity, and $L_1, L_2, L_3, L_4$ as in Figure \ref{figureFourLines}. Let $\sigma$ be the shaded triangle, endowed with the orientation of the plane, giving rise to a relative homology class in $\H_2(\PP^2_\CC, \mathcal{A})$. Then 
$$\partial_{\{1,2\}}(\sigma) = \partial_{L_1\cap L_2}(\partial_{L_2}(\sigma)) = -1,$$
while
$$\partial_{L_1\cap L_2}(\partial_{L_1}(\sigma)) = 0$$
because $\partial_{L_1}(\sigma)=0$.

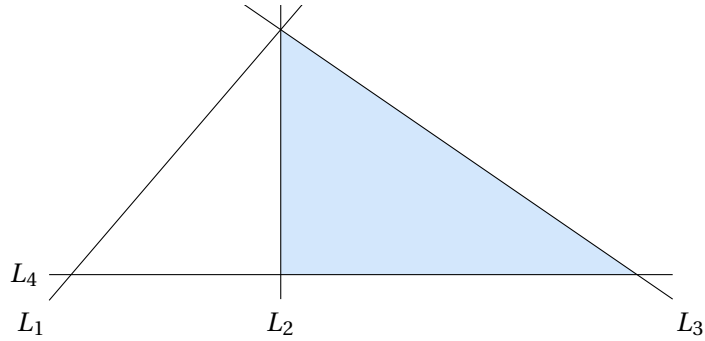
\begin{figure}[h]
\begin{center}
\begin{tikzpicture}[scale=.5]
    \begin{axis}[
    width=12cm, 
    domain=-.65:1.1, 
    xmin=-.8, xmax=1.2, 
    ymin=-.25, ymax=1.1,   
    axis x line = none, 
    axis y line = none,
    tick style={draw=none}, 
    ymajorticks=false, 
    xmajorticks=false,
    x post scale=1.8]
    \addplot[samples=200, name path = h] {0};
    \addplot[samples=200, name path = f] {1-x};
    \addplot[samples=200] {1+1.7*x};
    \addplot[samples=200] coordinates{(0,-.1) (0,1.2)};
    \addplot[MyBlue, opacity=0.2] fill between[of=f and h, soft clip={domain=0:1}];
    \node at (axis cs:-.7,-.2) {$L_1$};
    \node at (axis cs:0,-.2) {$L_2$};
    \node at (axis cs:1.15,-.2) {$L_3$};
    \node at (axis cs:-.72,0) {$L_4$};
    \end{axis}
\end{tikzpicture}
\end{center}
\caption{An arrangement of 4 lines (plus 1 line at infinity) in the projective plane.}
\label{figureFourLines}
\end{figure}
\end{rem}

We now give a general formula for canonical forms in the setting of arrangements of hyperplanes, which we believe to be new. It would be interesting to compare it with the ``dual volume'' formula of \cite[\S 7.4]{arkanihamedbailampositive}, valid in the case of complexified arrangements.

\begin{prop}\label{prop:canonical-form-nbc}
For $\sigma\in\H_n(\PP^n_\CC, \mathcal{A})$, we have
\begin{equation}\label{eq:canonical-form-nbc}
\omegacan_\sigma = \sum_{\substack{I \mbox{ \scriptsize{nbc set}} \\ |I|=n}} \partial_I(\sigma) \,\omega_I.
\end{equation}
\end{prop}

\begin{proof}
    For a set $I=\{i_1<\cdots <i_n\}\subset \{1,\ldots,N\}$ such that the hyperplanes $H_{i_1},\ldots, H_{i_n}$ are linearly independent, we denote by 
    $$\Res_I:\logforms{n}(\PP^n_\CC\setminus\mathcal{A})\to \Omega^0(H_{i_1}\cap\cdots \cap H_{i_n})=\CC$$ 
    the composition of the residue maps 
    $$\Res_{H_{i_1}\cap\cdots \cap H_{i_n}}\circ \cdots\circ \Res_{H_{i_{n-1}}\cap H_{i_n}}\circ \Res_{H_{i_n}}.$$
    Note the subtle dependence in the chosen linear order on the set of hyperplanes, analogous to that of iterated boundaries (Remark \ref{rem:iterated-boundaries-hyperplanes}).
    We clearly have  $\Res_I(\omega_I)=1$. As observed by Szenes \cite[Proposition 3.6]{szenes}, the nbc basis has the additional remarkable property that for two nbc sets $I$, $J$ of cardinality $n$,
    \begin{equation}\label{eq:nbc-residues}
    \Res_I(\omega_J)=\delta_{I,J}.
    \end{equation}
    By iterating Proposition \ref{prop:recursion-irreducible-components} we obtain the following commutative diagram.
    $$\diagram{
    \H_n(\PP^n_\CC,\mathcal{A})_\CC \ar[r]^-{\can_\CC} \ar[d]_-{\bigoplus \partial_I} & \logforms{n}(\PP^n_\CC\setminus\mathcal{A}) \ar[d]^-{\bigoplus\Res_I}\\
    \displaystyle\bigoplus_{\substack{I \mbox{ \scriptsize{nbc set}} \\ |I|=n}}\H_0(H_I)_\CC \ar[r]^-{\can_\CC} & \displaystyle\bigoplus_{\substack{I \mbox{ \scriptsize{nbc set}} \\ |I|=n}} \Omega^0(H_I)
    }$$
    By \eqref{eq:nbc-residues}, the  vertical map on the right is an isomorphism with inverse given by the natural maps $\Omega^0(H_I)\to \logforms{n}(\PP^n_\CC\setminus \mathcal{A})$ sending $1\in\Omega^0(H_I)$ to $\omega_I$. The result follows.
\end{proof}

\begin{rem}\label{rem:hyperplane-at-infinity-irrelevant}
For the sake of writing down simple formulas, we chose a hyperplane $H_0$ in $\mathcal{A}$ that we treated as the hyperplane at infinity. One can however interpret the formula for the canonical form independently of such a choice, as follows. Suppose that we start with a projective arrangement $H_1\cup\cdots \cup H_N$ and a class $\sigma\in \H_n(\PP^n_\CC,H_1\cup\cdots \cup H_N)$. Suppose that $H_0$ is in general position with respect to $H_1,\ldots,H_N$. Then setting $\omega_i=\dlog(f_i)$ instead of $\dlog(f_i/f_0)$, one sees that \eqref{eq:canonical-form-nbc} still holds, with the caveat that only the sum defines a projective form, and not any individual summand. This is because the canonical form $\omegacan_\sigma$ does not have any pole along $H_0$.
\end{rem}

\begin{rem}
A notable consequence of Proposition \ref{prop:canonical-form-nbc} is that a canonical form $\omegacan_\sigma$ is a sum of local contributions indexed by the \emph{corners} of $\mathcal{A}$, defined as the zero-dimensional intersections of hyperplanes in $\mathcal{A}$:
$$\omegacan_\sigma = \sum_c \omegacan_{\sigma,c},$$
where $\omegacan_{\sigma,c}$ is the sum of the terms corresponding to nbc sets $I$ such that $H_I$ is the corner $c$, and only depends on the (central) arrangement of hyperplanes near $c$.
This decomposition is in fact independent of the choices made, and simply reflects a canonical direct sum decomposition of the space $\logforms{n}(\PP^n_\CC\setminus \mathcal{A})$ indexed by the corners of $\mathcal{A}$ (it is called the ``Brieskorn decomposition'' in the literature, see \cite[Theorem 3.2]{dimcaarrangements}).
\end{rem}

\subsection{Convex polygons}

    In $\RR^2$, let $P$ be a convex polygon whose sides are bounded by lines $L_1,\ldots, L_n$ labeled counterclockwise. If $L_i$ is the vanishing locus of an affine function $f_i$, we let $\omega_i=\dlog(f_i)$. Then the canonical form of $P$ is
    \begin{equation}\label{eq:canonical-form-polygon}
    \omegacan_P = -(\omega_1\wedge \omega_2 + \omega_2\wedge \omega_3 + \cdots + \omega_{n-1}\wedge \omega_n + \omega_n\wedge \omega_1).
    \end{equation}
    Indeed, for every $i$ we have
    $$\Res_{L_i}(\omegacan_P) = (\omega_{i+1}-\omega_{i-1})|_{L_i},$$
    which is the canonical form of $\partial_{L_i}(P)$. Alternatively, one can apply Proposition \ref{prop:canonical-form-nbc} and note that $\partial_{L_i\cap L_{i+1}}\partial_{L_{i+1}}(P)=-1$.

\subsection{Simple polytopes}

    Let $P$ be a full dimensional simple polytope in $\RR^n$. For a vertex $v$ of $P$, the orientation of $\RR^n$ gives rise to a prescribed orientation (linear order up to even permutations) on the set of faces $F$ of $P$ containing $v$. Concretely, in the tangent space of $\RR^n$ at $v$ one can choose a positively oriented coordinate system $(x_1,\ldots,x_n)$ such that the tangent cone of $P$ is $\{x_1,\ldots, x_n\geq 0\}$, and the prescribed order is $F_1,\ldots ,F_n$ where $F_i:=\{x_i=0\}$. The following wedge product is therefore well-defined:
    $$\omegacan_{P,v}:=(-1)^{\frac{n(n+1)}{2}}\bigwedge_{F\ni v}\omega_F,$$
    where $\omega_F:=\dlog(f_F)$ with $f_F$ an affine function defining $F$. 
    
    \begin{prop}
    The canonical form of $P$ is given by the formula
    $$\omegacan_P = \sum_{v}\omegacan_{P,v}.$$
    \end{prop}
    
    \begin{proof}
    We apply Proposition \ref{prop:canonical-form-nbc} and note that there are no circuits in the hyperplane arrangement corresponding to the boundary of $P$ because $P$ is simple. If $Q_n=\{x_1,\ldots,x_n\geq 0\}$ denotes the standard $n$-quadrant with its usual orientation, its partial boundary $\partial_{\{x_n=0\}}Q_n$ is $Q_{n-1}$ with orientation multiplied by $(-1)^n$. Therefore, the iterated boundary $\partial_{1,\ldots,n}Q_n$ equals $(-1)^{n+(n-1)+\cdots +2 +1}$, and the claim follows.
    \end{proof}
  
    \begin{ex}
    One can check the consistency of the sign in the above formula with the formula for the canonical form of the hypercube $[0,1]^n$, which according to Proposition \ref{prop:multiplicativity} is
    $$\omegacan_{[0,1]^n} = (-1)^{\frac{n(n-1)}{2}} \bigwedge_{i=1}^n\dlog\left(\frac{z_i-1}{z_i}\right) = (-1)^{\frac{n(n+1)}{2}} \frac{\dd z_1}{z_1(1-z_1)}\wedge\cdots \wedge \frac{\dd z_n}{z_n(1-z_n)}.$$
    Indeed, the contribution of the vertex $(0,\ldots, 0)$ is $(-1)^{\frac{n(n+1)}{2}}\tfrac{\dd z_1}{z_1}\wedge\cdots \wedge \tfrac{\dd z_n}{z_n}$.
    \end{ex}

    \begin{ex}\label{ex:can-simplex}
    The canonical form of the $n$-simplex $\Delta^n=\{0\leq z_1\leq \cdots \leq z_n\leq 1\}$ is given by the formula
    $$\omegacan_{\Delta^n} = (-1)^{\frac{n(n+1)}{2}} \frac{\dd z_1\wedge\cdots \wedge \dd z_n}{z_1(z_2-z_1)\cdots (z_n-z_{n-1})(1-z_n)}.$$
    \end{ex}

\subsection{Convex polyhedra}

We start with an example. Let $P$ be the square pyramid in $\RR^3$ bounded by the $5$ hyperplanes
$$H_1=\{x=z\}, \; H_2=\{y=z\}, \; H_3=\{x=-z\}, \; H_4= \{y=-z\}, \mbox{ and } H_5=\{z=-1\}.$$
It is pictured in Figure \ref{figurePyramid}: looking from above, i.e., from the positive $z$-axis, one sees its triangular faces labeled $1, 2, 3, 4$ counterclockwise.

    \begin{figure}[h]
    \tdplotsetmaincoords{70}{120}
    \begin{center}
    \begin{tikzpicture}[scale=1.3, tdplot_main_coords,line cap=butt,line join=round,c/.style={circle,fill,color=MyDarkBlue,inner sep=1pt},
        declare function={a=4;h=3;}]
        \path
        (0,0,0) coordinate (A)
        (a,0,0) coordinate (B)
        (a,a,0) coordinate (C)
        (0,a,0) coordinate (D)     
    (a/2,a/2,h)  coordinate (S);
\draw[thick, color=MyDarkBlue] (S) -- (D) -- (C) -- (B) -- cycle (S) -- (C);
\draw[dashed, color=MyDarkBlue] (S) -- (A) --(D) (A) -- (B);
        
    \end{tikzpicture}
    \hspace{1cm}
    \begin{tikzpicture}[scale=1.3]
    \node at (0,0) {};
    \begin{scope}[yshift=2cm]
    \draw[thick, color=MyDarkBlue] (1.7,1.7)--(-1.7,1.7)--(-1.7,-1.7)--(1.7,-1.7)--(1.7,1.7);
    \draw[thick, color=MyDarkBlue] (1.7,1.7)--(-1.7,-1.7);
    \draw[thick, color=MyDarkBlue] (-1.7,1.7)--(1.7,-1.7);
    \node at (1,0) {$\textcolor{MyDarkBlue}{3}$};
    \node at (0,1) {$\textcolor{MyDarkBlue}{4}$};
    \node at (-1,0) {$\textcolor{MyDarkBlue}{1}$};
    \node at (0,-1) {$\textcolor{MyDarkBlue}{2}$};
    \node at (0,0.3) {};
    \end{scope}
    \end{tikzpicture}
    \end{center}
    \caption{The square pyramid (left), viewed from above (right).}
    \label{figurePyramid}
    \end{figure}
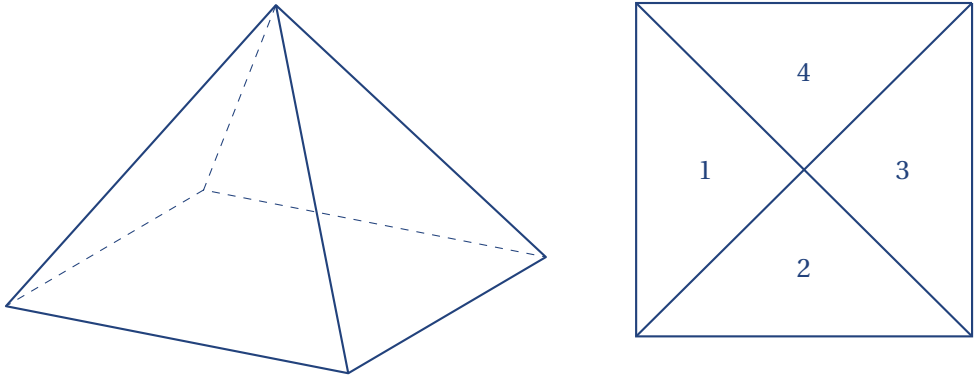

The only circuit is $\{1,2,3,4\}$, and we record the corresponding Orlik--Solomon relation
\begin{equation}\label{eq:OS-relation-pyramid}
\omega_1\wedge\omega_2\wedge\omega_3 - \omega_1\wedge\omega_2\wedge\omega_4+\omega_1\wedge\omega_3\wedge\omega_4-\omega_2\wedge\omega_3\wedge\omega_4=0.
\end{equation}
The nbc basis of $\Omega^3(\CC^3\setminus (H_1\cup H_2\cup H_3\cup H_4\cup H_5))$ consists of the seven monomials
\begin{center}$\omega_1\wedge\omega_2\wedge\omega_3 \;, \;\omega_1\wedge\omega_2\wedge\omega_4 \;,\;\omega_1\wedge\omega_3\wedge\omega_4 \;,$
\end{center}
\vspace{-3pt}
\begin{center}$\omega_1\wedge\omega_2\wedge\omega_5 \;, \;\omega_2\wedge\omega_3\wedge \omega_5 \;,\;\omega_3\wedge\omega_4\wedge \omega_5 \;, \;\omega_1\wedge\omega_4\wedge\omega_5 \; ,$
\end{center}
where the first three have residues along the top of the pyramid, and the last four along the base. 
Applying Proposition \ref{prop:canonical-form-nbc} gives the canonical form of $P$ written in this basis as 
\begin{align*}
\omegacan_P=& -\omega_1\wedge\omega_2\wedge\omega_3-\omega_1\wedge\omega_3\wedge\omega_4 \\
& + \omega_1\wedge\omega_2\wedge\omega_5 +\omega_2\wedge\omega_3\wedge\omega_5 +\omega_3\wedge\omega_4\wedge\omega_5 - \omega_1\wedge \omega_4\wedge \omega_5.
\end{align*}
Note that $\omega_1\wedge \omega_2\wedge \omega_4$ does not appear because $\partial_{\{1,2,4\}}(P)=0$. Indeed, $\partial_{H_4}(P)$ is a triangle on $H_4$ for which $H_2\cap H_4$ is not an edge but a vertex, and hence $\partial_{H_2\cap H_4}(\partial_{H_4}(P))=0$. 

One can use \eqref{eq:OS-relation-pyramid} to get an expression for $\omegacan_P$ where the cyclic symmetry $1\to 2\to 3\to 4\to 1$ is apparent:
\begin{align*}
    \omegacan_P  &= - \tfrac{1}{2}(\omega_1\wedge\omega_2\wedge \omega_3 + \omega_2\wedge\omega_3\wedge\omega_4 + \omega_3\wedge\omega_4\wedge\omega_1+\omega_4\wedge\omega_1\wedge\omega_2)\\
    & \qquad\quad + (\omega_1\wedge\omega_2+\omega_2\wedge\omega_3+\omega_3\wedge\omega_4+\omega_4\wedge\omega_1)\wedge \omega_5 \\
    & = -\tfrac{1}{2}\mathrm{Cyc}_4\left( \omega_1\wedge\omega_2\wedge\omega_3\right) + \mathrm{Cyc}_4(\omega_1\wedge\omega_2)\wedge\omega_5.
\end{align*}

The example of the square pyramid actually teaches us how to compute the canonical form of any convex polyhedron $P$ in $\RR^3$. It is a sum
$$\omegacan_P = \sum_v\omegacan_{P,v}$$
indexed by vertices $v$ of $P$, where $\omegacan_{P,v}$ can be described as follows. Let $F_1,\ldots,F_r$ denote the faces of $P$ at $v$, labeled counterclockwise when looking at $v$ from outside of $P$. Then 
$$\omegacan_{P,v} = -\sum_{i=2}^{r-1} \omega_1\wedge\omega_i\wedge\omega_{i+1}.$$
As above, averaging over the cyclic group $\ZZ/r\ZZ$ produces a more symmetric expression at the cost of introducing a denominator if $r\geq 4$.

\subsection{Moduli spaces of genus zero curves}

For $S$ a set with $n+3$ elements, we let 
$$\mathcal{M}_{0,S}\subset \overline{\mathcal{M}}_{0,S}$$
denote the moduli space of genus zero curves with marked points labeled by $S$, and its Deligne--Mumford compactification respectively. The latter is a smooth projective variety of dimension $n$, and the complement $\partial\overline{\mathcal{M}}_{0,S}:=\overline{\mathcal{M}}_{0,S}\setminus \mathcal{M}_{0,S}$ is a simple normal crossing divisor.

Fixing the genus zero curve to be $\PP^1_\CC$ and three of the marked points to be $0,1,\infty$, we have the following simple description:
$$\mathcal{M}_{0,S} = \{(z_1,\ldots,z_n)\in(\CC\setminus \{0,1\})^n \,|\, z_i\neq z_j \mbox{ if } i\neq j\}.$$
Consider $\PP^n_\CC$ with homogeneous coordinates $(z_0:z_1:\cdots :z_n)$, we see that $\mathcal{M}_{0,S}$ is the complement of the projective arrangement of hyperplanes
$$\mathcal{A}_n:=\bigcup_{0\leq i\leq n} \{z_i=0\} \cup \bigcup_{0\leq i<j\leq n} \{z_i=z_j\}.$$
By the work of Kapranov \cite{kapranov}, one can construct $\overline{\mathcal{M}}_{0,S}$ as a modification (Definition \ref{defi:modification})
$$\pi:(\overline{\mathcal{M}}_{0,S},\partial\overline{\mathcal{M}}_{0,S}) \to (\PP^n_\CC,\mathcal{A}_n),$$
explicitly given by an iterated blow-up along flats of $\mathcal{A}_n$. It follows that the pair $(\overline{\mathcal{M}}_{0,S},\partial\overline{\mathcal{M}}_{0,S})$ has genus zero, and by \eqref{eq:can-and-modification} one can compute canonical forms ``downstairs'' in projective space. For instance, there is a cell which is combinatorially an associahedron:
$$X_S\subset \overline{\mathcal{M}}_{0,S}(\RR),$$
with boundary along $\partial\overline{\mathcal{M}}_{0,S}$, which is mapped to the simplex $\Delta^n=\{0\leq z_1\leq \cdots \leq z_n\leq 1\}$ by $\pi$. Therefore, by Example \ref{ex:can-simplex}, its canonical form is given by the formula
$$\omegacan_{X_S} = (-1)^{\frac{n(n+1)}{2}} \frac{\dd z_1\wedge\cdots \wedge \dd z_n}{z_1(z_2-z_1)\cdots (z_n-z_{n-1})(1-z_n)} $$ 
in coordinates on $\mathcal{M}_{0,S}$.

\begin{rem}
The fact that moduli spaces of genus zero curves give rise to positive geometries in the sense of \cite{arkanihamedbailampositive} was first proved in \cite{arkanihamedhelam}, see also \cite{lammodulispaces}; the fact that $\omegacan_{X_S}$  only has logarithmic poles along $\partial X_S$ was known long before.  It would be interesting to study the more subtle case of (moduli spaces of) del Pezzo surfaces from the point of view of mixed Hodge theory \cite{positivedelpezzo}. 
\end{rem}

\bibliographystyle{hyperamsalpha}
\bibliography{biblio}
		
\end{document}